\documentclass[11pt]{article}

\usepackage{amsmath, amsthm}
\usepackage{amsmath, amsfonts}
\usepackage{amsmath, amssymb}
\usepackage{amsmath}
\usepackage[all]{xy}

\newtheorem{theorem}{Theorem}[subsection]
\newtheorem{definition}[theorem]{Definition}
\newtheorem{definition-lemma}[theorem]{Definition/Lemma}
\newtheorem{definition-explanation}[theorem]{Definition/Explanation}
\newtheorem{explanation-definition}[theorem]{Explanation/Definition}
\newtheorem{lemma}[theorem]{Lemma}
\newtheorem{lemma-definition}[theorem]{Lemma/Definition}
\newtheorem{proposition}[theorem]{Proposition}
\newtheorem{corollary}[theorem]{Corollary}
\newtheorem{remark}[theorem]{\it Remark}

\newtheorem{example}[theorem]{Example}

\newtheorem{notation}[theorem]{Notation}
\newtheorem{convention}[theorem]{\it Convention}

\numberwithin{equation}{subsection}

\newtheorem{sdefinition-lemma}[stheorem]{Definition/Lemma}
\newtheorem{sdefinition-explanation}[stheorem]{Definition/Explanation}
\newtheorem{sexplanation-definition}[stheorem]{Explanation/Definition}

\newtheorem{slemma-definition}[stheorem]{Lemma/Definition}

\newtheorem{sstheorem}{Theorem}[subsubsection]
\newtheorem{ssdefinition}[sstheorem]{Definition}
\newtheorem{ssdefinition-lemma}[sstheorem]{Definition/Lemma}
\newtheorem{ssdefinition-explanation}[sstheorem]{Definition/Explanation}
\newtheorem{ssexplanation-definition}[sstheorem]{Explanation/Definition}
\newtheorem{sslemma}[sstheorem]{Lemma}
\newtheorem{sslemma-definition}[sstheorem]{Lemma/Definition}
\newtheorem{ssproposition}[sstheorem]{Proposition}

\newtheorem{ssremark}[sstheorem]{\it Remark}

\newtheorem{ssexample}[sstheorem]{Example}

\input{epsfig.sty}

\newcommand{\Arg}{\mbox{\it Arg}\,}

 \newcommand{\Azscriptsize}{\mbox{\scriptsize\it A$\!$z}}

\newcommand{\Br}{\mbox{\it Br}\,}

\newcommand{\Calabi}{\mbox{\it Calabi}\,}

\newcommand{\Cover}{\mbox{\it Cover}\,}
\newcommand{\CP}{{\Bbb C}{\rm P}}

\newcommand{\Dersheaf}{\mbox{\it ${\cal D}$er}\,}

\newcommand{\Diff}{\mbox{\it Diff}\,}

\newcommand{\Endsheaf}{\mbox{\it ${\cal E}\!$nd}\,}
 
\newcommand{\Ext}{\mbox{\rm Ext}}

\newcommand{\GL}{\mbox{\it GL}}

\newcommand{\Hom}{\mbox{\it Hom}\,}
\newcommand{\Homsheaf}{\mbox{\it ${\cal H}$om}\,}

\newcommand{\Id}{\mbox{\it Id}}

\newcommand{\Imaginary}{\mbox{\it Im}\,}
\newcommand{\Image}{\mbox{\it Im}\,}

\newcommand{\Ker}{\mbox{\it Ker}\,}

 \newcommand{\scriptsizeLag}{\mbox{\scriptsize\it Lag}\,}

\newcommand{\Map}{\mbox{\it Map}\,}

\newcommand{\ModCategory}{\mbox{\it ${\cal M}$\!od}\,}

\newcommand{\MorphismCategory}{\mbox{\it ${\cal M}$\!orphism}\,}

\newcommand{\Quot}{\mbox{\it Quot}}
 \newcommand{\scriptsizeQuot}{\mbox{\scriptsize\it Quot}}

\newcommand{\Real}{\mbox{\it Re}\,}
\newcommand{\Rep}{\mbox{\it Rep}\,}

\newcommand{\SL}{\mbox{\it SL}}

 \newcommand{\scriptsizesLag}{\mbox{\scriptsize\it sLag}\,}

\newcommand{\Space}{\mbox{\it Space}\,}
 \newcommand{\smallSpace}{\mbox{\small\it Space}\,}

\newcommand{\Spec}{\mbox{\it Spec}\,}
 \newcommand{\boldSpec}{\mbox{\it\bf Spec}\,}
 
 \newcommand{\smallSpec}{\mbox{\small\it Spec}\,}

\newcommand{\Supp}{\mbox{\it Supp}\,}
 
\newcommand{\Sym}{\mbox{\it Sym}}

 \newcommand{\scriptsizecover}{\mbox{\scriptsize\it cover}\,}

\newcommand{\determinant}{\mbox{\it det}\,}
\newcommand{\dimm}{\mbox{\it dim}\,}

\newcommand{\et}{\mbox{\scriptsize\it \'{e}t}\,}

\newcommand{\length}{\mbox{\it length}\,}

\newcommand{\nc}{\mbox{\scriptsize\it nc}}
 
\newcommand{\scriptsizenoncommutative}{\mbox{\scriptsize\rm
                                             noncommutative}}

\newcommand{\pr}{\mbox{\it pr}}

\newcommand{\pt}{\mbox{\it pt}}
\newcommand{\rank}{\mbox{\it rank}\,}

\newcommand{\redscriptsize}{\mbox{\scriptsize\rm red}\,}

\newcommand{\reldimm}{\mbox{\it rel.dim}\,}

\newcommand{\trace}{\mbox{\it tr}\,}

\newcommand{\vol}{\mbox{\it vol}\,}

\begin{document}

\enlargethispage{24cm}

\begin{titlepage}

$ $

\vspace{-1.5cm} 

\noindent\hspace{-1cm}
\parbox{6cm}{\small February 2010}\
   \hspace{8cm}\
   \parbox[t]{5cm}{yymm.nnnn [math.SG]\\ D(6): A-brane, rev-1.}

\vspace{2cm}

\centerline{\large\bf
 D-branes and Azumaya noncommutative geometry:}
\vspace{1ex}
\centerline{\large\bf
 From Polchinski to Grothendieck}

\bigskip

\vspace{3em}

\centerline{\large
  Chien-Hao Liu
  \hspace{1ex} and \hspace{1ex}
  Shing-Tung Yau
}

\vspace{6em}

\begin{quotation}
\centerline{\bf Abstract}

\vspace{0.3cm}

\baselineskip 12pt  
{\small
 In this continuation of
  [L-Y3]    (arXiv:0709.1515 [math.AG]),
  [L-L-S-Y] (arXiv:0809.2121 [math.AG]),
  [L-Y4]    (arXiv:0901.0342 [math.AG]),
  [L-Y5]    (arXiv:0907.0268 [math.AG]), and
  [L-Y6]    (arXiv:0909.2291 [math.AG]),
 we give an overview of the posted part of the project and then
  take it as background to introduce
   Azumaya noncommutative $C^{\infty}$-manifold
   and the four aspects of morphisms therefrom to
   a projective complex manifold.
 This gives us then a description of
  supersymmetric D-branes of A-type in a Calabi-Yau manifold
  along the line of the Polchinski-Grothendieck Ansatz.
 The notion of K\"{a}hler differentials and their tensors
  for an Azumaya noncommutative space are introduced.
 Donaldson's picture of
  Lagrangian and special Lagrangian submanifolds
   as selected from the zero-locus of a moment-map
   on a related space of maps
  can be merged into the setting of morphisms
   from Azumaya manifolds with a fundamental module.
 As a pedagogical toy model for illustration,
  we study D-branes of A-type in a Calabi-Yau torus.
 Simple as it is, it reveals already several features of D-branes
  of A-type, including their assembling/disassembling.
 The short-vs.-long string wrapping behavior of
  matrix-strings in the string-theory literature
  can be produced in this context as well.
 In addition to the previous comparison with stringy works made,
 the 4th theme (subtitled:
  ``G\'{o}mez-Sharpe vs.\ Polchinski-Grothendieck")
  of Sec.~2.4 is to be read with the work
  [G-Sh] (arXiv:hep-th/0008150),
 while the 2nd theme (subtitled:
  ``Donagi-Katz-Sharpe vs.\ Polchinski-Grothendieck")
  of Sec.~4.2 is to be read with the work
  [D-K-S] (arXiv:hep-th/0309270).
 Sec.~4.3, though not yet ready to be subtitled
  ``Denef vs.\ Polchinski-Grothendieck", is to be read
  with the work [De] (arXiv:hep-th/0107152).
 Some directly related string-theory remarks are added
  to the end of each section.
} 
\end{quotation}

\bigskip

\baselineskip 12pt
{\footnotesize
} 

\bigskip

\noindent {\small MSC number 2010:
 14A22, 53C38, 81T30; 14D23, 15A54, 53D12, 81T75.
} 

\bigskip

\baselineskip 10pt
{\scriptsize
\noindent{\bf Acknowledgements.}
We thank
 Andrew Strominger and Cumrun Vafa
  for their influence on our understanding of strings and branes
  over the years.
C.-H.L.\ thanks in addition
 Yng-Ing Lee and Li-Sheng Tseng
  for discussions on the new part of the current D(6);
 Hungwen Chang, Robert Gompf, Margaret Symington, and Katrin Wehrheim
 (resp.\ Ilia Zharkov)
  for influencing his understanding of symplectic
  (resp.\ calibrated) geometry;
 Alina Marion/Xiaowei Wang
  for a discussion on {\it Quot}-schemes/stability;
 Angelo Vistoli
  for a communication, explanations, and literature guide
  on themes in [L-Y7];
 Frederik Denef (with Dionysios Anninos),
 Dennis Gaitsgory, Peter Kronheimer,
 Alina Marion, Cumrun Vafa, and Hao Xu
  for topic courses/seminar, fall 2009 - spring 2010;
 and Ling-Miao Chou for moral support.
 The project is supported by NSF grants DMS-9803347 and DMS-0074329.
} 

\end{titlepage}

\newpage
\enlargethispage{25cm}

\begin{titlepage}

$ $

\vspace{3em}

\centerline{\small\it
 Chien-Hao Liu dedicates this work and review}
\centerline{\small\it
 to Shiraz Minwalla and Mihnea Popa$^{\,a}$}
\centerline{\small\it for their crucial influence/contribution
 to the project before it is materialized}
\centerline{\small\it
 and to Ling-Miao Chou for her tremendous love.}

%
%
%
%
%
%

\vspace{6em}

{\small
$\hspace{1.6cm}\bullet$
A reflection on
 {\it string/M/F/$\cdots$? theory, branes, and dualities}$\,$:$^{\,b}$
\begin{figure}[htbp]
 \centerline{\psfig{figure=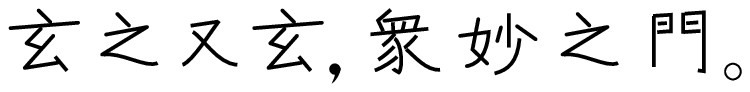,width=30pc,caption=}}
\end{figure}

\vspace{-2em}

\begin{itemize}
 \item[] $\hspace{2cm}$
 ($\,${\it Mystery and beyond mystery, door to all magics.}$\,$)
\end{itemize}

\begin{figure}[htbp]
 \centerline{\psfig{figure=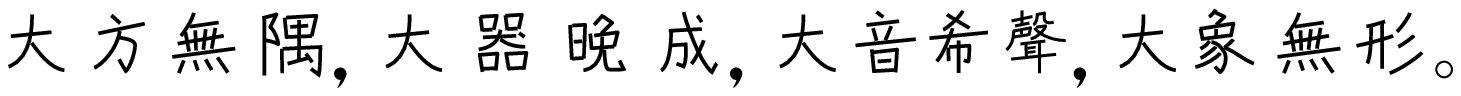,width=30pc,caption=}}
\end{figure}

\vspace{-2em}
\begin{itemize}
 \item[]
 $\hspace{2cm}$
 ($\,${\it So large that it has no bounds; \newline
  \mbox{$\hspace{2.2cm}$}
 so big that it takes a long time to make; \newline
  \mbox{$\hspace{2.2cm}$}
 so harmonious that it fits no tunes; \newline
  \mbox{$\hspace{2.2cm}$}
 so beautiful that it assumes no shapes.}$\,$)
\end{itemize}

} 

\vspace{2em}

\noindent{---------------------------------------------------}

{\footnotesize
$^{a\,}$(From C.-H.L.)
 Before moving up, I went to see {\it Prof.~Freed},
   who is for sure a role model for mathematicians
   who pursue physics issues.
 There in his office he delivered to me one of the most important advices
  I ever received,
  ``{\it No matter how you're interested in physics,
         you are still a mathematician  and
         you have to find your ground in mathematics.}"
 With the unexpected luck to come to
  one of the major and most active centers of algebraic geometry
  in the world
  after being exposed to Calabi-Yau spaces and toric geometry
  in Prof.~Candelas' weekly group meetings,
 I quietly decided to take algebraic geometry as my new starting point,
  despite a well-disposed warning passed to me
  as a quote of a famous mathematician to a renowned string-theorist:
  ``If you haven't learned algebraic geometry by the age of 28,
    then you'd better just forget about it."
 It is
   the friendly, concrete, and down-to-earth way
   {\it Joe Harris} thinks about and teaches algebraic geometry
    and
   {\it Mihnea Popa}'s insightful, systematic, and wide-spanned series
    of courses and seminars at Harvard
   from Grothendieck's foundations to research topics at the frontier
    during the years 2001--2005
  that turned around my supposedly doomed fate of failure.

 A miraculous coincidence happened:
 During these years
  {\it Shiraz Minwalla} was giving an equally insightful, systematic,
  and wide-spanned series of courses at Physics Department
  from quantum field theory foundation, supersymmetry,
   phase structures in a Wilson's theory-space,
   to basic and frontier topics in string theory.
 His enthusiastic lectures constantly filled the classroom with heat
  and turned the would-be terrifying courses into intellectual enrichment.
 Thus, two exceptionally energetic teachers
  - one from the mathematics side and the other from the physics side -
  and one common student
  with his mind quietly set on a central question:
  {\it What is a D-brane?}
 I would not expect this project to finally get started in year 2007
  without having met them both at the right time and right place
  in the first half of 2000s.
 With the above input from Mihnea and Shiraz
  and the initiation of this long-delayed and almost-abandoned project,
  the special unexpected sequence of lucks given to me
   -- namely an accidental stringy journey:
     {\it Thurston $\Rightarrow$ Alvarez $\&$ Nepomechie
     $\Rightarrow$ Candelas $\&$ Distler $\Rightarrow$ Yau} --
  at last have a purpose and acquire their meaning as a whole.

 Above all these, the daily summary of works to each other
  with {\it Ling-Miao} over the years has been providing me
  with tremendous momenta to the theme.
 Only an extremely lucky person is given simultaneously
  a demanding stringy detour/journey,
  an unusual opportunity, and
  a long-standing supporting soulmate  and
 it takes not only effort but also great fortune
  so that these coincidences and accidental encounters
  are not just in vain.

$^{b\,}$Lao-Tzu (600 B.C.), {\sl Tao-te Ching}
   ({\sl The Scripture on the Way and its Virtue}),
   excerpt from Chapter~1 and Chapter~41;
   English translation by Ling-Miao Chou.
} 

\end{titlepage}

\newpage
$ $

\vspace{-4em}  

\centerline{\sc D-Branes and Azumaya Geometry}

\vspace{2em}


\begin{flushleft}
{\Large\bf 0. Introduction and outline.}
\end{flushleft}
{\it D-branes}, defined in string theory as
 {\it boundary conditions for the end-points of open strings},
 appeared in the theory in the second half of 1980s
 and have become a central object in string theory since year 1995.
They reveal themselves in various faces/formats,
 depending on
 where they are looked at on the related {\it Wilson's
  theory-space}\footnote{Mathematicians
                 unfamiliar with this very important notion
                 from quantum field theory are referred to
                 [L-Y1: Appendix A.1]
                 for a brief introduction and literatures.
                In a sense, this is the grand-master moduli space/stack
                 with a built-in universal structure
                 that contains everything.}
 -- {\it either of superstring theory},
     locally parameterized by $(\tau^F_1, g_s)$ with
      $\tau^F_1$ being the tension of a fundamental open/closed string
       and
      $g_s$ being the string coupling constant related to
       the condensation of the dilaton field,
    {\it or of a $2$-dimensional superconformal field theory with boundary}
    ({\it $d=2$ SCFT}) --.

\bigskip

\begin{flushleft}
{\bf Where we are in the Wilson's theory-space of strings  and
     when we are in the history of D-branes.}
\end{flushleft}
In this review\footnote{In
           view of lots of issues still ahead,
           this is not yet the moment to write a review of this project.
          However, after a thought on a surprise contact
           from Vira Pobyzh/Richard Szabo,
          we believe that
           a brief overview of what we have been doing
            without being distracted by technical details and languages
           may still be a meaningful communication
           with mathematicians and string-theorists
           who work on D-branes.
         The writing remains a mathematical one.
         Nevertheless, string-theory-related remarks are collected
          in the end of each section,
          titled `{\it String-theoretical remarks}'.
         Mathematicians may ignore the string-theory stuff
          without influencing their understanding of
          the mathematical contents
         while string-theorists may replace
          the word ``{\it scheme}" by  ``{\it
           space completely and uniquely characterized by
           its local function rings
           with the latter allowed to have nilpotent elements}"  and
          the word ``{\it stack}"
           by ``{\it space obtained from a generalized gluing of schemes
           in a way that reflects the automorphism groups of objects
           the stack means to parameterize}"
           (via an Isom-functor construction),
         whenever either occurs, if they are not familiar with these
          two fundamental notions from modern algebraic geometry.
         We thank
          Mihnea Popa for several early discussions, spring 2002;
          William Oxbury and Andrei C\u{a}ld\u{a}raru
            for communicating their work
            in spring, 2002, and fall, 2007, respectively;
          Liang Kong and Eric Sharpe
           for sharing with us their insights on D-branes, fall 2007;
           and
          Si Li and Ruifang Song for the participation of [L-L-S-Y]
           (D(2)), spring 2008.
          }
  of [L-Y3] (D(1)), [L-L-S-Y] (D(2), with Si Li and Ruifang Song),
  [L-Y4] (D(3)), [L-Y5] (D(4)), [L-Y6] (D(5)),
  and some preparatory part of [L-Y7],
 we stand
  \begin{itemize}
   \item[$\cdot$] ({\it where})\\
    in the regime of the related Wilson's theory-space
     either
      {\it where the D-brane tension is small but still large enough
      compared to $\tau^F_1$}
      --
       so that the D-branes remain easily excitable
        by open strings with a boundary attached to them   and
       that they won't bend the causal structure of the target space-time
        to form a surrounding event horizon to close themselves up to
        black branes --
     or {\it where a boundary state of a $d=2$ SCFT remains having
      a space-time interpretation}
       -- so that ``branes" are really branes; and

   \item[$\cdot$] ({\it when})\\
    either at the year 1995 when
     Polchiski [Pol2] realized that
      D-branes serve as a source for Ramond-Ramond fields created
      by excitations of closed superstrings and
     Witten [Wi4] gave an immediate follow-up to consider bound systems
      of D-branes
    or, exactly speaking, at the {\it year 1988}
      when D-branes came into light, in Polchinski and Cai's work [P-C],
       and
      the mass-tower of fields thereupon and the dynamics of these fields
      follow respectively
       from open-superstring spectrum and
       from the vanishing requirement of the conformal anomaly
        on the superconformal field theory on the open-string world-sheet,
    (see [Pol4] for related references),
  \end{itemize}
  and ask, based on what string-theorists taught us,
  \begin{itemize}
   \item[$\cdot$]
   {\bf Q.} {\it What truly is a D-brane in its own right?}
  \end{itemize}

Our starting point is a paragraph in Polchinski's textbook
 [Pol4: vol.~I, Sec.~8.7, theme: {\it The D-brane action}, p.~272]
 concerning
 \begin{itemize}
  \item[$\cdot$] \parbox[t]{14cm}{\it
   a matrix-type noncommutative enhancement of
   target space-time when probed by stacked D-branes;}
 \end{itemize}
see also [Joh: Sections~4.10, 5.5, 9.7, 16.3]
 and Sec.~1.1 of the current work for more discussions/review.
It turns out that understanding this mysterious behavior of D-branes
 holds a key to realize a fundamental mathematical nature of D-branes.
In the current review,
we explain this particular aspect of D-branes,
 namely the one from
 \begin{itemize}
  \item[]\parbox[t]{14cm}{\it
   merging the above mysterious behavior of stacked D-branes
   and Grothendieck's viewpoint of local contravariant equivalence
   of function rings and geometries.}
 \end{itemize}
This brings in the notion of
 {\it Azumaya noncommutative schemes with a fundamental module}
 in the algebro-geometric category  or
 {\it Azumaya noncommutative manifolds with a fundamental module}
  in the differential/symplectic topological category.
The correct notion of morphisms therefrom to a string target-space
 gives us then a re-formulation of the notion of D-branes
 that can reproduce several key features of D-branes
 in the string-theory literature, originally derived
 from open-string-induced quantum field theory on D-branes,
 cf.\ [L-Y3], [L-Y4], and [L-Y5].
Furthermore,
in the algebro-geometric side,
 the moduli space/stack of such morphisms has a surprising feature
 of serving as a master moduli space/stack that simultaneously
 incorporates several different moduli spaces/stacks
 in commutative algebraic geometry,
 cf.~[L-L-S-Y] and [L-Y6].
In the symplecto-geometric side, the notion rings naturally with
 how Donaldson looks at special Lagrangian submanifolds
 in a Calabi-Yau manifold as a special class of maps
 into the Calabi-Yau space, cf.~[Don] and [Hi2].
This matches nicely with the role of D-branes
 as a master object in string theory.
All these together give us an evidence that
 \begin{itemize}
  \item[]\parbox[t]{14cm}{{\it
   the Azumaya-type noncommutative geometric structure
   on a D-brane world-volume},
   rather than on the string target-space(-time),  {\it
   can provide us with an alternative starting point
   to understanding D-branes.}}
 \end{itemize}
As a hindsight, this is a step one could already take in year 1988,
 instead of nearly twenty years later.
In particular, one may try to re-do everything about D-branes
 with this Ansatz.

\bigskip

\begin{flushleft}
{\bf The DOR triangle.}
\end{flushleft}
While this review and the so-far-posted part of the project discuss
 only D-branes, one should always keep in mind the mysteries of
 the other two closely related disciplines as well:
 (Cf.~bottom of the
  {\it D}-brane/{\it o}pen-string/{\it R}amond-Ramond-field triangle.)

 $$
  \xymatrix@!=3pc{
    & \mbox{\it D-branes\rule[-1ex]{0ex}{1ex}}\ar@<.6ex>[dr]\ar[dl] & \\
   **[l]\mbox{\it Open strings\hspace{1ex}}\ar@<.3ex>[rr]\ar@<.6ex>[ur]
    && **[r]\mbox{\hspace{1ex}\it Ramond-Ramond fields}
            \ar@<.3ex>[ll]\ar[ul]
  }
 $$

\bigskip

Indeed, it is a pursuit of understanding open Gromov-Witten theory
 that turned us back to re-thinking D-branes, cf.~[L-Y3: footnote~1].
Furthermore, it is known that
 Ramond-Ramond fields on a string target-space(-time)
  are not just differential forms thereupon.
Their complications are already manifest from
 both
  the F-theory interpretation of the rank-$0$ Ramond-Ramond field
  $C_{(0)}$ ([Va2] and [M-V])
 and
  the $\SL(2,{\Bbb Z})$-duality of type IIB superstring theory that
  exchanges the rank-$2$ Ramond-Ramond field $C_{(2)}$,
   which sources D-strings,
  with the $B$-field, which sources fundamental strings ([Schw]).
In particular, such yet-to-be-understood fundamentals of
 Ramond-Romand fields are necessary to realize
 how and why a general collection of D-branes can live
 on a compact Calabi-Yau space
 without violating the charge conservation law of D-branes.

\bigskip

\begin{flushleft}
{\bf String-theoretical remarks.}
\end{flushleft}
As an object developed for more than twenty years since late 1980s,
 that have led to numerous new insights to string theory,
string-theorists do have a very good reason to question that,
 \begin{itemize}
  \item[$\cdot$] \parbox[t]{14cm}{\it
   Didn't we string-theorists already know
         all the fundamentals of D-branes?}
 \end{itemize}
{For} string-theorists who do have this doubt,
 we suggest them to recall a most important example
 of a similar gap or concept-delay between physics and mathematics,
 namely the notion of ``path-integrals in quantum field theory".
It began with Richard Feynman's Ph.D.\ thesis
 under John A.\ Wheeler at Princeton, 1942,
 with title
 ``The principle of least action in quantum mechanics"
and has now become
 a central/standard language in quantum field theory.
It provides even the very definition of what it means
 when claiming two quantum field theories are ``the same",
 a concept that pervades string literatures nowadays.
Yet, despite the sixty-eight years that have passed,
mathematicians still look at this notion with amazement and awe.
During this time, there have been mathematical attempts to
 understand it --
 first in an infinite-dimensional analysis/measure-theory aspect,
 later in a functor-and-category aspect and in a combinatorics aspect,
 and more recently in a motive aspect.
Each of these reveals some mathematical nature of path-integrals
 behind physicists' formal rules.
With this example in mind, string-theorists may be willing to
 re-think about
{\it the phenomenon of space-time being enhanced to matrix-valued
     noncommutativity when probed by stacked D-branes}.
 \begin{itemize}
  \item[$\cdot$] \parbox[t]{14cm}{\it
   Is it really the space-time geometry that is enhanced  or,
   indeed, is it the world-volume of the stacked D-branes
   that is enhanced first?}
 \end{itemize}
While it looks plausible in some situations
 for a trading between noncommutativity from these two opposing aspects,
to our best understanding at the moment,
 such trading can only be at best partial and
 the full moduli problems different aspects lead to
  are different mathematically.
Indeed, the very question turns out to be related to another question:
 \begin{itemize}
  \item[$\cdot$] \parbox[t]{14cm}{\it
  Taking D-branes, say, of the lowest dimensions in a theory,
   as truly fundamental objects,  then
  can they recover their signatory features by themselves
   without resorting to open strings?}
 \end{itemize}

In this notes/review,
 we are not trying to tell string-theorists what they should think.
Rather, we explain our thoughts here
 in a hopefully languagewise-friendlier way than our previous works
and welcome string-theorists to ask themselves
 the above same questions and find their own answers.
This is not an issue of rigor;
 we regularly mix ourselves with string-theorists and
try to learn/appreciate the way they think
 and hence already pass that long time ago.
{\it This is an issue of what exactly is fundamentally
     at work for D-branes.}
In retrospect, the ansatz we propose in [L-Y3] (D(1)),
  cf.\ Sec.~1.1 of this review,
 could have been observed by string-theorists themselves in year 1988
 as well, due to its simplicity and
  the fact that the computation of open string spectrum and
  the induced massless field contents on stacked D-branes
  were already ready at that time,
if Grothendieck's EGA and SGA series of works
 had entered the stringy community before then.
Indeed, there have been efforts from different string-theory groups
  to understand better the foundation of D-branes.
E.g.~the work [G-Sh] of Tom\'{a}s G\'{o}mez and Eric Sharpe
 at year 2000 (cf.\ Sec.~2.4, Test (4), of the current review)  and
see [L-Y3: Remark~2.2.5] for a sample of other stringy groups.
They all influenced our thoughts on D-branes one way or another
 during the brewing years.
In Spring 2007, after a train of discussions
 with Duiliu-Emanuel Diaconescu in December 2006
 on open-string world-sheet instantons
    and a vanishing lemma of open Gromov-Witten invariants
    (cf.~[D-F] and [L-Y2])
 that turned us back to re-thinking D-branes,
we realized that
 {\it it is a carefully selected spirit of Grothendieck's works
      on commutative algebraic geometry that we have to take,
      but not necessarily its full contents}.
It is this realization that finally enabled us to re-read
 stringy works on D-branes in Grothendieck's eyes.
Occasionally, we wonder if the history were rerun  and
 the ansatz were observed by string-theorists either in 1995 or 1988,
how things would have been different
 in the physics and the mathematics side.
We would never know.
All we can do is explain and push this ansatz to the extreme
 and let string-theorists decide for themselves.
Overall, D-branes remain
 a very complicated and mysterious object to us.
We have yet much to learn and to be amazed.

\bigskip

The very limited short list of stringy literatures
 quoted in the current work and the more special ones
 in the previous D(1) - D(5) of the project
 are among those we constantly go back to
 to find new insights, new understandings, and new guides.
In string-theorists' standard, a paper passing five years
 can be regarded as an old paper.
Yet, these ``very old" papers remain to us a reservoir of inspirations.
They remain rich deposits of gold for string-theory-oriented
 mathematical minds.
On the other hand, while they influenced us greatly,
 they by no means reflect the activity of this field.
Unfamiliar readers are suggested to read
 the book `{\sl D-branes}' [Joh] by Johnson and
 the string textbook [B-B-Sc] by Becker, Becker, and Schwarz
 and the original (mainly physical) works quoted therein
 to gain balanced and more complete insights and feelings of
 the various aspects of D-branes
 in the first seven years to a decade, 1995--2002/1995--2006,
after the second revolution of string theory in 1995.

\bigskip

\bigskip

\noindent
{\bf Convention.}
 Standard notations, terminology, operations, facts in
  (1) physics aspects of D-branes;
  (2) supersymmetry;
  (3) (commutative) algebraic geometry / stacks;
  (4) associative rings and algebras;
  (5) symplectic / calibrated geometry;
  (6) sheaves on manifolds
  can be found respectively in
  $\;$(1) [Po], [Joh];
  $\;$(2) [W-B], [Arg], [Te];
  $\;$(3) [E-H], [Ha] / [L-MB];
  $\;$(4) [Pi], [Re];
  $\;$(5) [McD-S] / [Ha-L], [McL];
  $\;$(6) [Kas-S].
 \begin{itemize}
  \item[$\cdot$]
   All {\it schemes} are Noetherian over ${\Bbb C}$.

  \item[$\cdot$]
   All associative {\it rings} and {\it algebras} are unital.

  \item[$\cdot$]
   All {\it manifolds} are smooth, closed, and orientable
    unless otherwise noted.

  \item[$\cdot$]
   Index whose precise value is
    either clear from the text or irrelevant to the discussion
   is occasionally
    omitted to a $\bullet$ to make the expression cleaner/simpler.

  \item[$\cdot$]
   ${\Bbb C}$ as a {\it field} in algebra
   vs.\ ${\Bbb C}$
        as the {\it complex line} in complex/symplectic geometry.

  \item[$\cdot$]
   D-branes of A-type (resp.\ B-type) are supported
    on special Lagrangian cycles (resp.\ holomorphic cycles).
   The name follows [B-B-St], [O-O-Y], and [H-I-V].

  \item[$\cdot$]
   {\it Field} $B$ in the sense of quantum field theory
   vs.\ {\it base} scheme $B$ in algebraic geometry\\
   vs.\ D-{\it branes} of B-type in string theory.

  \item[$\cdot$]
   {\it Noncommutative algebraic geometry} is a very technical topic,
    with various nonequivalent points of view and formulations.
   For the current work,
    [Art] of Artin and [A-N-T] of Artin, Nesbitt, and Thrall
      on {\it Azumaya algebras},
    [K-R] of Kontsevich and Rosenberg
      on {\it noncommutative smooth spaces},  and
    [LeB1], [LeB2], [LeB3] of Le Bruyn
      on {\it noncommutative geometry@$n$}
    are particularly relevant.
   See [L-Y3: References] for more references.

   \item[$\cdot$]
   Mathematicians without quantum field theory background may still
    get {\it a feel of D-branes in physicists' perspective} from
    [Zw] of Zwiebach;
   see also [B-B-Sc] of Becker, Becker, and Schwarz
    for an updated account of string theory up to year 2006.
 \end{itemize}

\bigskip

\bigskip

\begin{flushleft}
{\bf Outline.}
\end{flushleft}
{\small
\baselineskip 12pt  
\begin{itemize}
 \item[0.]
  Introduction:
  \vspace{-.6ex}
  \begin{itemize}
   \item[$\cdot$]
    Where we are in the Wilson's theory-space of strings
    and when we are in the history\\ of D-branes.

   \item[$\cdot$]
    The DOR triangle.

   \item[$\cdot$]
    String-theoretical remarks.
  \end{itemize}

 \item[1.]
  D-branes and the Polchinski-Grothendieck Ansatz.
  \vspace{-.6ex}
  \begin{itemize}
   \item[1.1]
    From stacked D-branes to the Polchinski-Grothendieck Ansatz.

   \item[1.2]
    String-theoretical remarks. 
  \end{itemize}

 \item[2.]
  The algebro-geometric aspect:
  Azumaya noncommutative schemes with a fundamental module\\
  and morphisms therefrom.
  \vspace{-.6ex}
  \begin{itemize}
   \item[2.1]
    What should be the Azumaya-type noncommutative geometry and
    a morphism from\\ an Azumaya-type noncommutative space?

   \item[2.2]
    The four aspects of a morphism from an Azumaya scheme
     with a fundamental module\\ to a scheme.

   \item[2.3]
    D-branes in a $B$-field background
    \`{a} la Polchinski-Grothendieck Ansatz.
    \begin{itemize}
     \item[2.3.1]
      Nontrivial Azumaya noncommutative schemes
      with a fundamental module  and\\ morphisms therefrom.

     \item[2.3.2]
      Azumaya quantum schemes with a fundamental module
       and morphisms therefrom.
    \end{itemize}

   \item[2.4]
    Tests of the Polchinski-Grothendieck Ansatz for D-branes.

   \item[2.5]
    Remarks on general Azumaya-type noncommutative schemes.

   \item[$\cdot$]
    String-theoretical remarks on Sec.~2. 
  \end{itemize}

 \item[3.]
  The differential/symplectic topological aspect:
  Azumaya noncommutative $C^{\infty}$-manifolds\\
  with a fundamental module and morphisms therefrom.
  \vspace{-.6ex}
  \begin{itemize}
   \item[3.1]
    Azumaya noncommutative $C^{\infty}$-manifolds
     with a fundamental module and morphisms\\ therefrom.

   \item[3.2]
    Lagrangian morphisms and special Lagrangian morphisms:\\
    Donaldson and Polchinski-Grothendieck.

   \item[$\cdot$]
   String-theoretical remarks on Sec.~3.
  \end{itemize}

 \item[4.]
  D-branes of A-type on a Calabi-Yau torus and their transitions.
  \vspace{-.6ex}
  \begin{itemize}
   \item[4.1]
    The stack ${\frak M}^{\,0^{Az^f}}_{\,r}\!\!(C)$.

   \item[4.2]
    Special Lagrangian cycles with a bundle/sheaf
    on a Calabi-Yau torus\\ \`{a} la Polchinski-Grothendieck Ansatz.

   \item[4.3]
    Amalgamation/decomposition (or assembling/disassdembling)
    of special Lagrangian\\ cycles with a bundle/sheaf.

   \item[$\cdot$]
   String-theoretical remarks on Sec.~4.
  \end{itemize}
\end{itemize}
} 

\newpage

\section{D-branes and the Polchinski-Grothendieck Ansatz.}

We review in this section
 how the Polchinski-Grothendieck Ansatz for D-branes arises
 when a crucial open-string-induced behavior of fields
 on stacked D-branes is re-read from Grothendieck's viewpoint.
Readers are referred to [L-Y3: Sec.~2] for further discussions.

\bigskip

\subsection{From stacked D-branes to the Polchinski-Grothendieck
            Ansatz.}

{From} a fundamental question,
we are led to an intrinsic nature of D-branes.

\bigskip

\begin{flushleft}
{\bf What is a D-brane?}
\end{flushleft}
A {\it D-brane} (i.e.\ {\sl D}{\it irichlet} {\it mem}{\sl brane})
 is meant to be a boundary condition for open strings
 in whatever form it may take,
 depending on where we are in the related Wilson's theory-space.
A realization of D-branes that is most related to the current work
 is an embedding $f:X\rightarrow Y$ of a manifold $X$
 into the open-string target space-time $Y$
 with the end-points of open strings being required to lie in $f(X)$.
This sets up a $2$-dimensional Dirichlet boundary-value problem
 from the field theory on the world-sheet of open strings.
Oscillations of open strings with end-points in $f(X)$
 then create a mass-tower of various fields on $X$,
 whose dynamics is governed by open string theory.
This is parallel to the mechanism that
 oscillations of closed strings create fields in space-time $Y$,
 whose dynamics is governed by closed string theory.
Cf.~Figure 1-1-1.

\begin{figure}[htbp]
 \epsfig{figure=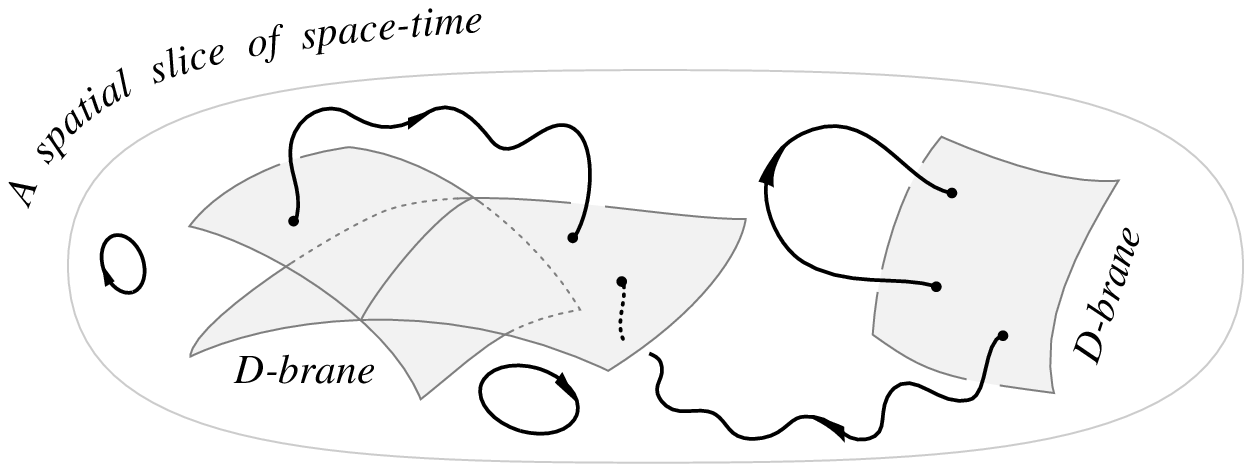,width=16cm}
 \centerline{\parbox{13cm}{\small\baselineskip 12pt
  {\sc Figure} 1-1-1.
  D-branes as boundary conditions for open strings in space-time.
  This gives rise to interactions of D-brane world-volumes
   with both open strings and closed strings.
  Properties of D-branes,
    including the quantum field theory on their world-volume and
              deformations of such,
   are governed by open and closed strings via this interaction.
  Both oriented open (resp.\ closed) strings and
   a D-brane configuration are shown.
  }}
\end{figure}

\noindent
Let $\xi:=(\xi^a)_a$ be local coordinates on $X$ and
  $\Phi:=(\Phi^a;\Phi^{\mu})_{a,\mu}$ be local coordinates on $Y$
  such that the embedding $f:X\hookrightarrow Y$ is locally
  expressed as
  $$
   \Phi\; =\; \Phi(\xi)\; =\; (\Phi^a(\xi); \Phi^{\mu}(\xi))_{a,\mu}\;
   =\; (\xi^a,\Phi^{\mu}(\xi))_{a,\mu}\,;
  $$
 i.e., $\Phi^a$'s (resp.\ $\Phi^{\mu}$'s) are local coordinates along
       (resp.\ transverse to) $f(X)$ in $Y$.
This choice of local coordinates removes redundant degrees of freedom
 of the map $f$, and
$\Phi^{\mu}=\Phi^{\mu}(\xi)$ can be regarded as (scalar) fields on $X$
 that collectively describes the postions/shapes/fluctuations
 of $X$ in $Y$ locally.
Here, both $\xi^a$'s, $\Phi^a$'s, and $\Phi^{\mu}$'s are ${\Bbb R}$-valued.
The open-string-induced gauge field on $X$ is locally given
 by the connection $1$-form $A=\sum_a A_a(\xi)d\xi^a$
 of a $U(1)$-bundle on $X$.

When $r$-many such D-branes $X$ are coincident/stacked,
 from the associated massless spectrum of (oriented) open strings
 with both end-points on $f(X)$
 one can draw the conclusion that
 \begin{itemize}
  \item[(1)]
   The gauge field $A=\sum_a A_a(\xi)d\xi^a$ on $X$ is enhanced to
    $u(r)$-valued.

  \item[(2)]
   Each scalar field $\Phi^{\mu}(\xi)$ on $X$ is also enhanced
    to matrix-valued.
 \end{itemize}
Property (1) says that there is now a $U(r)$-bundle on $X$.
But
 \begin{itemize}
  \item[$\cdot$]
   {\bf Q.}\ {\it What is the meaning of Property (2)?}
 \end{itemize}
{For} this, Polchinski remarks that:
 (Note:
  Polchinski's $X^{\mu}$ and $n$ $\;=\;$ our $\Phi^{\mu}$ and $r$.)
\begin{itemize}
 \item[$\cdot$]
 [{\sl quote from} [Pol4: vol.~I, Sec.~8.7, p.~272]]\hspace{1em}
 ``{\it
  For the collective coordinate $X^{\mu}$, however, the meaning
   is mysterious: the collective coordinates for the embedding of
   $n$ D-branes in space-time are now enlarged to $n\times n$ matrices.
  This `noncommutative geometry' has proven to play a key role in
   the dynamics of D-branes, and there are conjectures that
   it is an important hint about the nature of space-time.}"
\end{itemize}
(See also a comment in [Joh: Sec.~4.10 (p.~125)].)
{From} the mathematical/geometric perspective,
 \begin{itemize}
  \item[$\cdot$]
  Property (2) of D-branes, the above question, and Polchinski's remark
 \end{itemize}
 can be incorporated into the following single guiding question:
 \begin{itemize}
  \item[$\cdot$]
  {\bf Q.\ [D-brane]}$\;$
  {\it What is a D-brane intrinsically?}
 \end{itemize}
In other words, what is the {\it intrinsic} nature/definition
 of D-branes so that {\it by itself}
 it can produce the properties of D-branes
 (e.g.\ Property (1) and Property (2) above)
 that are consistent with, governed by, or originally produced by
 open strings as well?

\bigskip

\begin{flushleft}
{\bf From Polchinski to Grothendieck.}
\end{flushleft}
{To} understand Property (2), one has two perspectives:
 \begin{itemize}
  \item[(A1)]
   [{\it coordinate tuple as point}]\hspace{1em}
   A tuple $(\xi^a)_a$ (resp.\ $(\Phi^a; \Phi^{\mu})_{a,\mu}$)
    represents a point on the world-volume $X$ of the D-brane
    (resp.\ on the target space-time $Y$).

  \item[(A2)]
   [{\it local coordinates as generating set of
         local functions}]\hspace{1em}
   Each local coordinate $\xi^a$ of $X$
    (resp.\ $\Phi^a$, $\Phi^{\mu}$ of $Y$)
    is a local function on $X$ (resp.\ on $Y$)  and
   the local coordinates $\xi^a$'s
    (resp.\ $\Phi^a$'s and $\Phi^{\mu}$'s) together
    form a generating set of local functions on the world-volume $X$
    of the D-brane (resp.\ on the target space-time $Y$).
 \end{itemize}
While Aspect (A1) leads one to the anticipation of a noncommutative space
 from a noncommutatization of the target space-time $Y$
 when probed by coincident D-branes,
Aspect (A2) of Grothendieck leads one to
 a different -- seemingly dual but not quite -- conclusion:
  a noncommutative space from a noncommutatization of
  the world-volume $X$ of coincident D-branes,
 as follows.

Denote by ${\Bbb R}\langle \xi^a\rangle_{a}$
  (resp.\ ${\Bbb R}\langle \Phi^a; \Phi^{\mu}\rangle_{a, \mu}$)
 the local function ring on the associated local coordinate chart
 on $X$ (resp.\ on $Y$).
Then the embedding $f:X\rightarrow Y$,
  locally expressed as
  $\Phi=\Phi(\xi)
   =(\Phi^a(\xi); \Phi^{\mu}(\xi))_{a,\mu}=(\xi^a; \Phi^{\mu}(\xi))$,
 is locally contravariantly equivalent to
 a ring-homomorphism\footnote{For
             string-theorists:
            I.e.\ pull-back of functions from the target-space $Y$
             to the domain-space $X$ via $f$.
            See Sec.~2.1 for more about Grothendieck's philosophy
             for constructing `geometric spaces' and `morphisms'
             among them; cf.\ [E-H] and [Ha].}
 $$
  f^{\sharp}\;:\;
   {\Bbb R}\langle \Phi^a; \Phi^{\mu}\rangle_{a, \mu}\;
   \longrightarrow\; {\Bbb R}\langle \xi^a\rangle_{a}\,,
  \hspace{1em}\mbox{generated by}\hspace{1em}
  \Phi^a\;\longmapsto\; \xi^a\,,\;
  \Phi^{\mu}\;\longmapsto\;\Phi^{\mu}(\xi)\,.
 $$
When $r$-many such D-branes are coincident, $\Phi^{\mu}(\xi)$'s
 become $M_r({\Bbb C})$-valued.
Thus, $f^{\sharp}$ is promoted to a new local ring-homomorphism:
 $$
  \hat{f}^{\sharp}\;:\;
   {\Bbb R}\langle \Phi^a; \Phi^{\mu}\rangle_{a, \mu}\;
   \longrightarrow\; M_r({\Bbb C}\langle \xi^a\rangle_{a})\,,
  \hspace{1em}\mbox{generated by}\hspace{1em}
  \Phi^a\;\longmapsto\; \xi^a\cdot{\mathbf 1}\,,\;
  \Phi^{\mu}\;\longmapsto\;\Phi^{\mu}(\xi)\,.
 $$
Under Grothendieck's contravariant local equivalence of function rings
 and spaces, $\hat{f}^{\sharp}$ is equivalent to saying that we have
 now a map $\hat{f}: X_{\scriptsizenoncommutative}\rightarrow Y$,
  where $X_{\scriptsizenoncommutative}$ is the new domain-space,
   associated now to the enhanced function-ring
   $M_r({\Bbb C}\langle \xi^a\rangle_{a})$.
Thus, the D-brane-related noncommutativity in Polchinski's treatise
 [Pol4], as recalled above, implies the following ansatz
 when it is re-read from the viewpoint of Grothendieck:

\bigskip

\noindent
{\bf Polchinski-Grothendieck Ansatz [D-brane: noncommutativity].}
{\it
 The world-volume of a D-brane carries a noncommutative structure
 locally associated to a function ring of the form $M_r(R)$,
 where $r\in {\Bbb Z}_{\ge 1}$ and
   $M_r(R)$ is the $r\times r$ matrix ring over $R$.
} 

\bigskip

\noindent
We call a geometry associated to a local function-rings of matrix-type
 {\it Azumaya-type noncommutative geometry};
cf.~[Art] and [L-Y6: footnote~23].
Note that
 when the closed-string-created $B$-field on the open-string target
  space(-time) $Y$ is turned on,
 $R$ in the Ansatz can become noncommutative itself;
cf.~[S-W], [L-Y6], and Sec.~2.3.2 of the current review.
Cf.~Figure~1-1-2.

\bigskip

\begin{figure}[htbp]
 \epsfig{figure=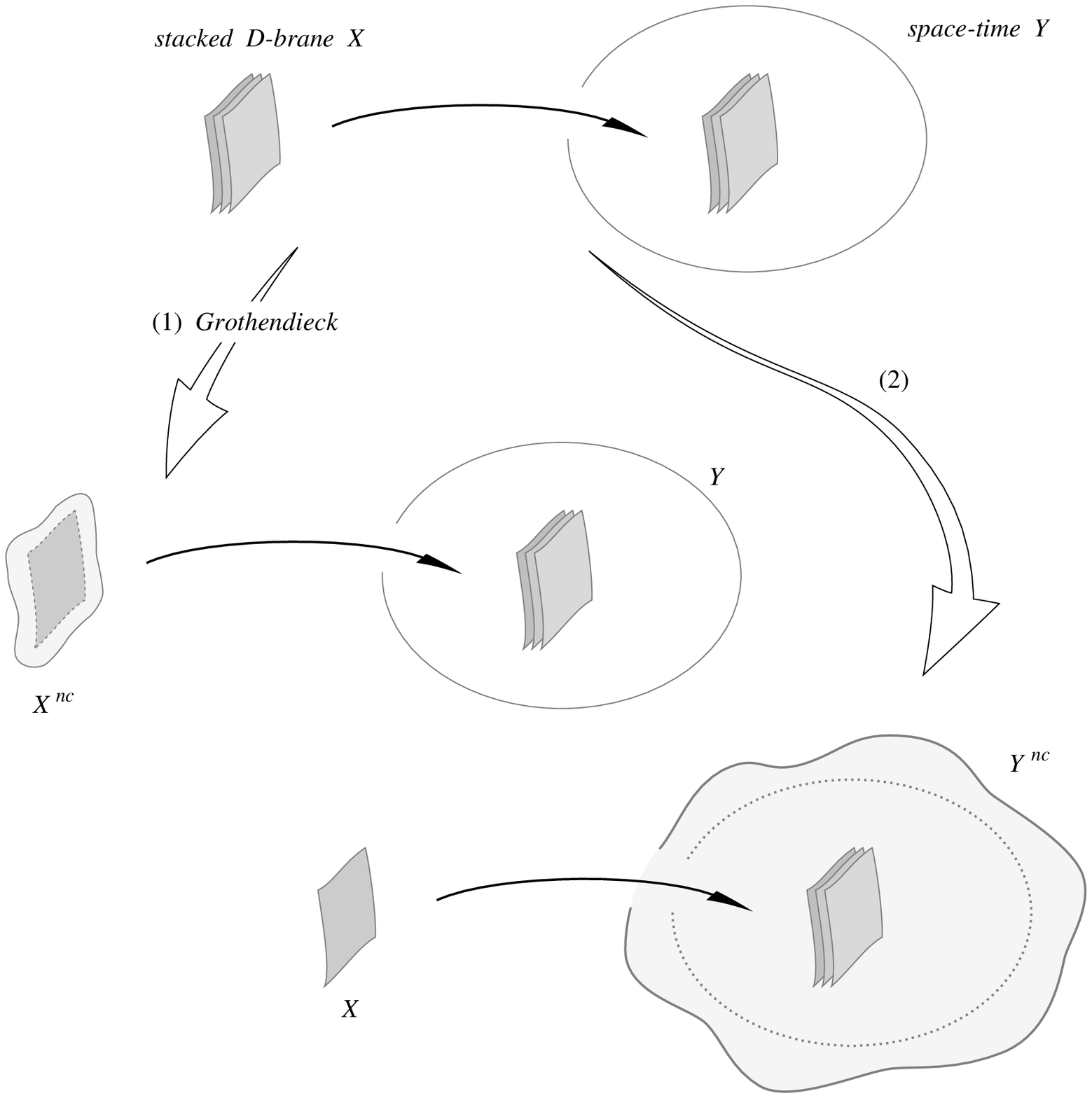,width=16cm}
 \centerline{\parbox{13cm}{\small\baselineskip 12pt
  {\sc Figure} 1-1-2.
  Two counter (seemingly dual but not quite) aspects
   on noncommutativity related to coincident/stacked D-branes:
   (1) (= (A2) in the text)
       noncommutativity of D-brane world-volume
       as its fundamental/intrinsic nature
   versus
   (2) (= (A1) in the text)
       noncommutativity of space-time as probed by stacked D-branes.
  (1) leads to the {\it Polchinski-Grothendieck Ansatz} and
   is more fundamental from Grothendieck's viewpoint
   of contravariant equivalence of the category of local geometries
   and the category of function rings.
  }}
\end{figure}

\noindent
An additional statement hidden in this Ansatz that follows from
 mathematical naturality is that
\begin{itemize}
 \item[$\cdot$] {\it
  fields on $X$ are local sections of sheaves ${\cal F}$
   of modules of the structure sheaf ${\cal O}_X^{\nc}$
   of $X$ associated to the above noncommutative structure.}
\end{itemize}
Furthermore, this noncommutative structure on D-branes
 (or D-brane world-volumes)
 is more fundamental than that of space-time in the sense that,
\begin{itemize}
 \item[$\cdot$] {\it
  from Grothendieck's equivalence,
   the noncommutative structure of space-time, if any, can be detected
   by a D-brane only when the D-brane probe itself is noncommutative.}
\end{itemize}

When D-branes are taken as fundamental objects as strings,
 we no longer want to think of their properties as derived
 from open strings.
Rather, D-branes should have their own intrinsic nature
 in discard of open strings.
Only that when D-branes co-exist with open strings in space-time,
 their nature has to be compatible/consistent with
 the originally-open-string-induced properties thereon.
It is in this sense that we think of a D-brane world-volume
 as an Azumaya-type noncommutative space, following the Ansatz,
 on which other additional compatible structures
  -- in particular, a Chan-Paton module -- are defined.

\bigskip

\subsection{String-theoretical remarks.} 

Now that a detailed explanation of the ansatz is re-given in Sec.~1.1,
Readers from string-theory side are suggested to re-think about
the theme `String-theoretical remarks' in Sec.~0.
Further remarks follow.

\bigskip

\begin{flushleft}
{\bf How about Lie-algebra type function-rings?}
\end{flushleft}
Naively, one may feel that
 even if a noncommutative structure does emerge on the world-volume
  of stacked D-branes, it should be of Lie-algebra type.
There are reasons against taking Lie algebras as the function rings
 locally for the D-brane world-volume geometry:
 \begin{itemize}
  \item[$\cdot$] ({\it physical reason})\\
   A completely physical reason is that
    indeed Lie algebra is not truly the algebra that is used
    in describing the action for the open-string-induced
    quantum field theory on the world-volume of stacked D-branes
    before taking trace.
   While it is possible that
    the potential term for the fields
     governing the deformation of D-branes in such an action
    is given by a combination of Lie-brackets of these fields,
    the kinetic term for them before taking an appropriate trace is not.
   Since every associative (unital) algebra is associated also
    with a Lie algebra by taking commutators,
   it is really the matrix product of these fields
    that is truly behind the action for stacked D-branes.

  \item[$\cdot$] ({\it mathematical reason})\\
   There are also mathematical reasons:
   A Lie algebra is non-associative and
    has no identity element with respect to the Lie product.
   Such an algebra is difficult to do geometry
    due to the difficulty to introduce the notion of open sets
    by ``localizing the algebra"
    (i.e.\ inverting some elements in the Lie algebra).
   Without this concept, the passage from local to global
    via gluing open sets becomes obstructed.
   Furthermore, there is no algebra-homomorphism from
    a commutative ring $R$ to a Lie algebra $L$.
   (Even if one ignores the issue of the identity element,
     then any such algebra-homomorphism is the zero-homomorphism.)
   Thus, no morphisms
    $$
     \Space(L)\; \longrightarrow\; \Space(R)\, :=\, \Spec R
    $$
    can be constructed.
   This makes it difficult to talk about
    a stacked D-brane, say, in a Calabi-Yau space
    as a morphism from the Lie-type-noncommutative D-brane world-volume
    (even if such a notion exists) to the Calabi-Yau space.
 \end{itemize}

\bigskip

\begin{flushleft}
{\bf Need supersymmetry or ${\boldmath B}$-field?}
\end{flushleft}
The Azumaya-type noncommutative structure on the D-brane world-volume
 stated in the ansatz occurs
 whether or not there is a supersymmetry for the D-brane configuration
  or a $B$-field background on the target-space(-time).
The latter may influence this structure but is not the cause for it;
 cf.~[L-Y6] (D(5)).
The only cause for this structure is simply
 open-strings!\footnote{For
                   this reason, as a hindsight,
                    this noncommutativity-enhancement phenomenon and
                    the ansatz should really be observed in year 1988
                   since every necessary ingredient was readily there;
                   cf.~Sec.~0.}
This suggests that it is indeed a very fundamental property of D-branes.
See also [L-Y4: footnote~1] (D(3)) for a related remark.\footnote{We
                            thank Lubos Motl for a related comment.}

\bigskip

\begin{flushleft}
{\bf A comparison with quantum mechanics and quantization of strings.}
\end{flushleft}
{For} string-theorists who still feel uncomfortable about this ansatz,
 the following comparison is worthy of a thought.\footnote{We
                          thank Cumrun Vafa for the conceptual point
                           addressed here
                           made in his string-theory course, spring 2010,
                          and a discussion.}

In the study of quantum mechanics or quantization of strings,
there are bosonic fields $X^{\mu}$ on the particle world-line
 or the string world-sheet
 that are associated to the space-time coordinates.
Collectively, they describe the positions and deformations
 of the particle or string when time flows.
Quantization of the particle or string moving in a space-time
 renders $X^{\mu}$ operator-valued.
Thus, formally, in this process the space-time coordinates become
 operator-valued as well.
 \begin{itemize}
  \item[$\cdot$] {\bf Q.}\
  \parbox[t]{13cm}{\it {\bf [quantized or not quantized]}$\;$
   When a particle or string moving in a space-time gets quantized,
   does the space-time itself get quantized as well, or not?}
 \end{itemize}
Replacing `{\it particle}' or `{\it string}'
  by `{\it coincident/stacked D-branes}' and
  `{\it get(s) quantized}$\,$' by
 `{\it have/ (has) $X^{\mu}$ become matrix-valued}$\,$',
 one can ask exactly the same question for D-branes.
If one answers the above question by:
``Yes, the space-time is also (enforced to be) quantized",
 then for the replaced question for D-branes,
 one is led to the currently more preferred view
  in the string-theory community
  that the space-time become matrix-ring-type-noncommutatively enhanced.
{\it If one answers instead:
 ``No, the space-time is not quantized; it remains classical",
 then for the replaced question for D-branes,
  one is led exactly to the Polchinski-Grothendieck Ansatz!}

\bigskip

\section{The algebro-geometric aspect:
         Azumaya noncommutative\\ schemes with a fundamental module
         and morphisms\\ therefrom.}

Once one understands the fundamental Azumaya-algebra nature of
 the local function-rings of D-branes (and their world-volume),
since the D-branes ``observed" by open-strings are their image
 in the target-space(-time) via a map, cf.\ Sec.~1.1,
an immediate question is then
 \begin{itemize}
  \item[$\cdot$] {\bf Q.}\ \parbox[t]{13cm}{\it {\bf [morphism]}$\;$
   What should be a morphism from an Azumaya noncommutative space
   $X^{\!A\!z}$ to a string target-space(-time) $Y$?}
 \end{itemize}
In this section, we address this technical problem
 in the algebro-geometric category.
This is where different viewpoints/formulations of
 a ``noncommutative geometry" may lead to different answers.
It turns out that
 one cannot hope to extend Grothendieck's construction
 of commutative algebraic geometry in full
 to the noncommutative case\footnote{This
                is a standing research problem for noncommutative
                 algebraic geometers for more than three decades
                 (or even nearly eight decades
                  beginning with the work of Ore, year 1931,
                  that touches upon the notion of
                  ``noncommutative localizations").
                There are various technical issues for the notion
                 of a ``noncommutative scheme"
                 either as a topological space or as a category,
                 pursued by several groups of mathematicians.
                See [L-Y3: References].}
and, hence, {\it we have to make a choice,
 guided by what truly matters for D-branes}.
This is the place we start to diverge
 from the various existing versions of ``noncommutative geometry".

\bigskip

\subsection{What should be the Azumaya-type noncommutative geometry
    and\\ a morphism from an Azumaya-type noncommutative space?}

\begin{flushleft}
{\bf How to construct/understand a ``geometric space"?}
\end{flushleft}
We may start with a related fundamental question:
\begin{itemize}
 \item[$\cdot$]
  {\bf Q.}\ {\it
  How do we ``construct/understand" a ``geometric space $X$"?}
\end{itemize}
The lesson one learns from Grothendieck's construction of modern
 (commutative) algebraic geometry and some later understandings
 is that there are (at least) four ways:
\begin{itemize}
 \item[(1)] ({\it as a ringed topological space})\hspace{1em}
  a point-set $X$ with a topology on it
   together with a sheaf - with respect to that topology - of rings
   (i.e.\ the structure sheaf ${\cal O}_X$)
   that encodes the data of local function rings of $X$,
  this gives rise to Grothendieck's theory of schemes;

 \item[(2)] ({\it as a functor of points})\hspace{1em}
  fix a collection of ``basic spaces" and see how they maps to $X$;

 \item[(3)] ({\it as a base for sheaves of modules})\hspace{1em}
  instead of the ringed topological space $(X,{\cal O}_X)$ itself,
  one looks at the category $\ModCategory_X$
   of modules of the structure sheaf ${\cal O}_X$ of $X$;

 \item[(4)] ({\it as a probe})\hspace{1em}
  fix a collection of ``basic spaces" and see how $X$ maps to them.
\end{itemize}
It turns out that in commutative algebraic geometry,
 Methods (1), (2), and (3) are equivalent (cf.~[E-H], [Ha], and [Ro])
while Method (4) in general is much weaker\footnote{In
                     the realm of projective algebraic geometry,
                    Method (4) is part of Mori's Program.}
 than any of Methods (1), (2), and (3).

Methods (1) and (3) can be applied to understand ``$\Space
M_r(R)$"
 in the case $R$ is a commutative ${\Bbb C}$-algebra.
Only that it gives us back $\Spec R$ and whatever information/``geometry"
 hidden in $M_r(R)$ is completely lost;
cf.\ Morita equivalence.\footnote{{\it Remark
                    for noncommutative algebraic geometers}:
                   In this work, we do not adopt
                    the Morita-equivalenct-type attitude toward
                    noncommutative geometry.
                   In particular, any space $X$ with
                    a noncommutative structure sheaf ${\cal O}_X^{nc}$
                    is called a `noncommutative space' in our convention
                    to emphasize this ${\cal O}_X^{nc}$.
                   We thank Lieven Le Bruyn for a remark
                    that provokes us a re-thought on our naming and
                    the foundation of noncommutative algebraic geometry.}
Method (2) can be pursued for $M_r(R)$
 and is related to smearing lower-dimensional D-branes
  in a higher-dimensional one.
We should postpone this to later part of the project.

Comparing Sec.~1.1
 on how we are led to the Polchinski-Grothendieck Ansatz,
 in which a D-brane plays the role
 of a domain of a map into an open-string target-space(-time),
one sees that: among the four methods,
 \begin{itemize}
  \item[$\cdot$]
   {\it it is the weakest Method (4) that fits best
   the purpose of D-branes}.
 \end{itemize}
Furthermore,
 \begin{itemize}
  \item[$\cdot$]
   {\it it is `morphisms therefrom' that play the central roles
    for understanding D-branes};
   whether one can truly build a satisfying/reasonable $\Space M_r(R)$
    as a ringed topological space is a secondary issue
    since one has $\Spec R$ at least.\footnote{I.e.\
                         for the purpose of D-branes,
                         one has to take
                          a carefully selected spirit of
                           Grothendieck's works
                           on commutative algebraic geometry
                          but not its full contents.}
 \end{itemize}
In other words, we don't focus on what exactly $\Space M_r(R)$ is.
Rather, we address {\it how $\Space M_r(R)$ can map to other spaces},
 through which we get a feel/sense/manifestation of the hidden geometry
 of $\Space M_r(R)$ and, hence, D-branes.
This leads us thus to the next theme.

\bigskip

\begin{flushleft}
{\bf How to construct/define a morphism without spaces?}
\end{flushleft}
The discussion in the previous theme leads us to the next question:
 \begin{itemize}
  \item[$\cdot$] {\bf Q.}\ \parbox[t]{13cm}{\it
   {\bf [morphism without spaces]}$\;$
   How to construct/define a morphism without having topological spaces
    as the domain and the target to begin with?}
 \end{itemize}
An answer to this question, which we follow, is that
 \begin{itemize}
  \item[$\cdot$] ({\it space})\hspace{2.8em}
   take the notion of a ``{\it space}" solely
    as an equivalence class $[{\cal S}]$ of gluing systems ${\cal S}$
    of (associative unital) rings,
   where an equivalence ${\cal S}\sim {\cal S}^{\prime}$
    is defined via a notion of common refinement of the systems
    via localizations of the rings;
  denote symbolically this ``space" by $\Space [{\cal S}]$;

  \item[$\cdot$] ({\it morphism})\hspace{1em}
   a {\it morphism}
    $\varphi: \Space [{\cal R}]\rightarrow \Space [{\cal S}]$
    can then be defined contravariantly
   as an equivalence class $[\varphi^{\sharp}]$ of
    gluing systems of ring-homomorphisms $\varphi^{\sharp}$
    from any representative of $[{\cal S}]$
    to fine-enough representatives of $[{\cal R}]$.
 \end{itemize}
Thus, a key issue here goes back to the old one since Ore
 at year 1931: {\it localizations of (associative unital) rings}.
This notion affects what class of morphisms $\varphi$ one can define.

This is again a standing question for noncommutative algebraic
geometers. Not to let this block our move towards D-branes, we
take in this project the simplest class of localizations: {\it
central localizations},
 i.e.\ localization by a multiplicatively closed system of elements
 (including the identity $1$ of the ring) in the center of the ring.
Thus, as long as the very technical issue of localization is concerned,
 it is the same as in the realm of commutative algebraic geometry.
This frees us
 to focus on true D-brane-related geometric themes,
 which already reveal very rich contents.
Restricted to central localizations,
the above general discussions can be polished into an encapsuled format,
 which gives Definition~2.2.2 in the next subsection
 when the target-space is a usual (commutative) scheme.

\begin{example}
{\bf [D0-branes on ${\Bbb A}^1$].} {\rm
 Here we illustrate the above discussions by considering
  {\it D0-branes of rank $r$} on the (complex) affine line
  ${\Bbb A}^1=\Spec {\Bbb C}[z]$.
 Along the above line of thoughts,
 such a D-brane is described by a {\it morphism}
  $\,\varphi:(\Space (M_r({\Bbb C}),{\Bbb C}^r) \rightarrow {\Bbb A}^1\,$
  from the {\it Azumaya point} $\pt^{A\!z} := \Space(M_r({\Bbb C}))$
  with the fundamental (left-)$M_r({\Bbb C})$-module $E={\Bbb C}^r$
  to ${\Bbb A}^1$.\footnote{In
                  this example, we hiddenly specify a trivialization
                   of the fundamental module $E$
                   to make the discussion explicit throughout.
                  Rigorously speaking, one should not do so.}
 To unravel it, it is then described equivalently but contravariantly
  by a {\it ${\Bbb C}$-algebra homomorphism}
  $\varphi^{\sharp}:{\Bbb C}[z]\rightarrow M_r({\Bbb C})$.
 The latter is determined
  by an assignment $z\mapsto m_{\varphi}\in M_r({\Bbb C})$.
 The {\it image} $\varphi(\pt^{A\!z})$ of the Azumaya point $\pt^{A\!z}$
  under $\varphi$ is given by the subscheme
  $\Spec({\Bbb C}[z]/\Ker(\varphi^{\sharp}))$ of ${\Bbb A}^1$.
 The {\it Chan-Paton module/bundle/sheaf} on D-branes
  as observed by open strings in ${\Bbb A}^1$
  corresponds to the {\it push-forward} $\varphi_{\ast}E$
  of $E$ to ${\Bbb A}^1$, defined by the ${\Bbb C}[z]$-module structure
  of $E$ via $\varphi^{\sharp}$.
 In this way,
  {\it even without knowing what $\Space(M_r({\Bbb C}))$ really is,
       one can still do geometry}.
 See [L-Y3: Sec.~4.1] (D(1)). Cf.~{\sc Figure}~2-1-1.

\begin{figure}[htbp]
 \epsfig{figure=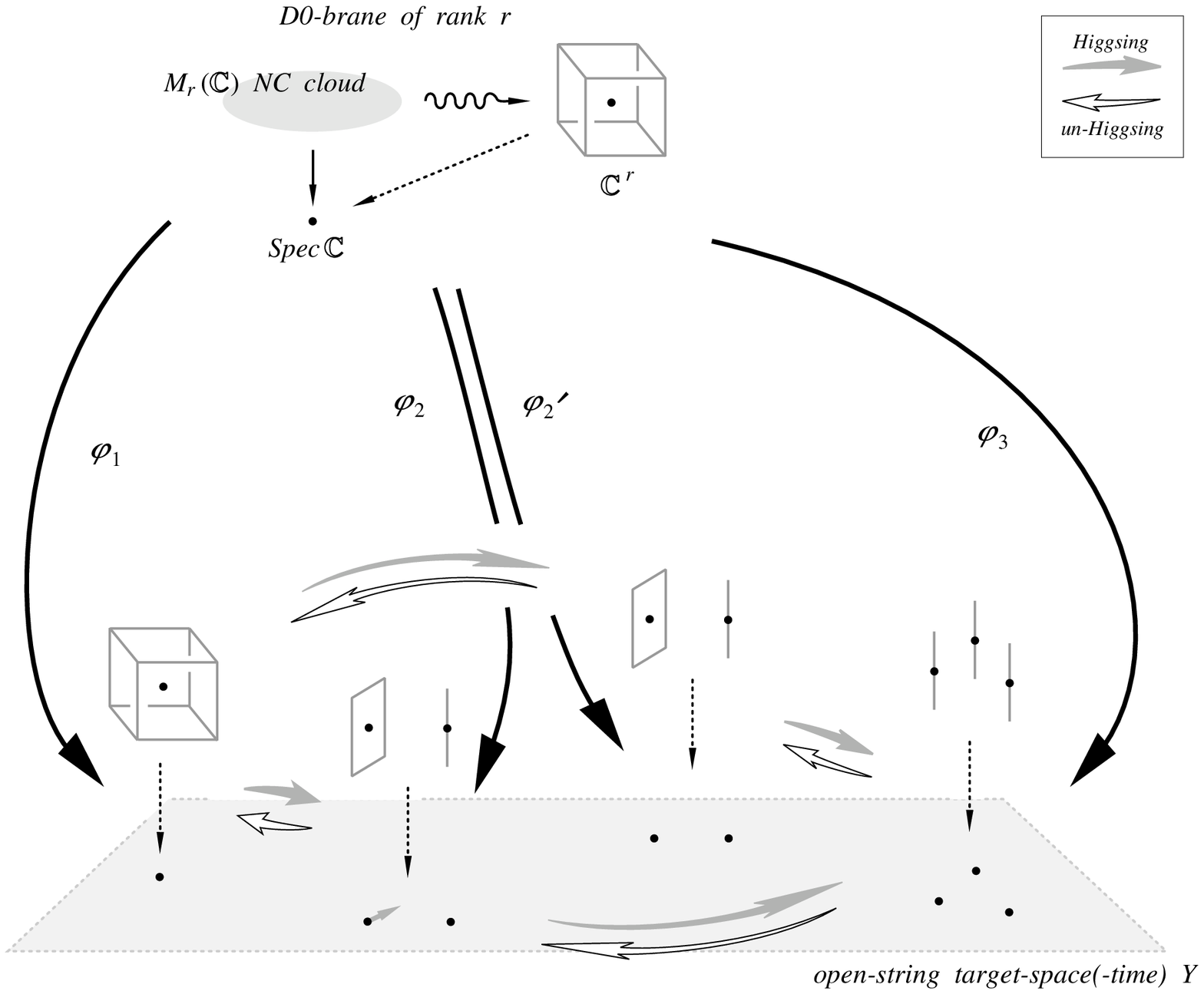,width=16cm}
 \centerline{\parbox{13cm}{\small\baselineskip 12pt
  {\sc Figure}~2-1-1.
 Despite that {\it Space}$\,M_r({\Bbb C})$ may look only
   one-point-like,
  under morphisms
  the Azumaya ``noncommutative cloud" $M_r({\Bbb C})$
  over {\it Space}$\,M_r({\Bbb C})$ can ``split and condense"
  to various image schemes with a rich geometry.
 The latter image schemes can even have more than one component.
 The Higgsing/un-Higgsing behavior of the Chan-Paton module of
   D$0$-branes on $Y$ ($={\Bbb A}^1$ in Example) occurs
  due to the fact that
   when a morphism
    $\varphi:$ {\it Space}$\,M_r({\Bbb C}) \rightarrow Y$
    deforms,
   the corresponding push-forward $\varphi_{\ast}E$
    of the fundamental module $E={\Bbb C}^r$
    on {\it Space}$\,M_r({\Bbb C})$ can also change/deform.
 These features generalize to morphisms
  from Azumaya schemes with a fundamental module to a scheme $Y$.
 Despite its simplicity, this example already hints at
  a richness of Azumaya-type noncommutative geometry.
 In the figure, a module over a scheme is indicated by
  a dotted arrow $\xymatrix{ \ar @{.>}[r] &}$.
 }}
\end{figure}
}\end{example}

\bigskip

Readers are referred to [L-Y3: Sec.~1 and Sec.~2.1] (D(1))
 for more details and related references.

\bigskip

\subsection{The four aspects of a morphism from an Azumaya scheme
    with\\ a fundamental module to a scheme.}

In this subsection, we give the four equivalent descriptions
 for a morphism
 from an Azumaya noncommutative scheme with a fundamental module
 to a (commutative) scheme.
See [L-Y3] (D(1)) and [L-L-S-Y] (D(2)) for more discussions.

\bigskip

\begin{flushleft}
{\bf I.\ The fundamental setting.}
\end{flushleft}
Let $X$ and $Y$ be schemes over ${\Bbb C}$.
We assume that both $X$ and $Y$ are projective
 for the convenience of, e.g., addressing moduli problems.
However, several basic definitions given
 do not require this condition.
Summing up all the previous considerations gives us then the
following fundamental definitions.

\begin{definition} {\bf [(commutative) surrogate].} {\rm
 Let ${\cal O}_X^{A\!z}$ be a coherent sheaf of
  (associative unital) ${\cal O}_X$-algebras on $X$,
 locally modeled on $M_r({\cal O}_U)$
  for an affine open subset $U$ of $X$.
 Let ${\cal O}_X\subset {\cal A}\subset {\cal O}_X^{A\!z}$
  be a commutative ${\cal O}_X$-subalgebra of ${\cal O}_X^{A\!z}$.
 Then $X_{\cal A}:= \boldSpec {\cal A}$ is called a (commutative)
  {\it surrogate} of $X^{\!A\!z}:=(X,{\cal O}_X^{A\!z})$.
} \end{definition}

\noindent
One should think of $X_{\cal A}$ as a finite scheme over
  and dominating $X$
 that is itself canonically dominated by $X^{\!A\!z}$.
An affine cover of $X_{\cal A}$ corresponds to
 a gluing system of algebras from central localizations.
Following this, the notion of morphisms from $X^{\!A\!z}$ to $Y$,
 as an equivalence class of gluing systems of ring-homomorphisms
  with respect to covers, can be phrased as

\begin{definition} {\bf [morphism].} {\rm
 A {\it morphism} from $X^{\!A\!z}$ to $Y$,
  in notation $\varphi: X^{\!A\!z}\rightarrow Y$,
  is an equivalence class of pairs
  $$
   ({\cal O}_X \subset {\cal A}
               \subset {\cal O}_X^{\Azscriptsize}\;,\;
     f:X_{\cal A}:=\boldSpec{\cal A}\rightarrow Y)\,,
  $$
  where
  \begin{itemize}
   \item[(1)]
    ${\cal A}$ is a commutative ${\cal O}_X$-subalgebra
     of ${\cal O}_X^{A\!z}$;

   \item[(2)]
    $f:X_{\cal A} \rightarrow Y$
    is a morphism of (commutative) schemes;

   \item[(3)]
    two such pairs
     $({\cal O}_X \subset {\cal A}_1
                \subset {\cal O}_X^{A\!z}\;,\;
      f_1:X_{{\cal A}_1}\rightarrow Y)$ and
     $({\cal O}_X \subset {\cal A}_2
                \subset {\cal O}_X^{A\!z}\;,\;
      f_2:X_{{\cal A}_2}\rightarrow Y)$
     are equivalent, in notation
     $$
      ({\cal O}_X \subset {\cal A}_1
                  \subset {\cal O}_X^{A\!z}\;,\;
       f_1:X_{{\cal A}_1}\rightarrow Y)\;
      \sim\;
      ({\cal O}_X \subset {\cal A}_2
                  \subset {\cal O}_X^{A\!z}\;,\;
           f_2:X_{{\cal A}_2}\rightarrow Y)\,,
     $$
    if there exists a third pair
     $({\cal O}_X \subset {\cal A}_3
                  \subset {\cal O}_X^{A\!z}\;,\;
      f_3:X_{{\cal A}_3}\rightarrow Y)$
     such that
      ${\cal A}_3 \subset {\cal A}_i$ and that
      the induced diagram
       \begin{eqnarray*}
       \xymatrix{
        X_{{\cal A}_i}\ar[drr]^{f_i}\ar[d] &&   \\
        X_{{\cal A}_3}\ar[rr]^{f_3}        && Y \\
        }
       \end{eqnarray*}
       commutes, for $i=1,\, 2$.
  \end{itemize}
 To improve clearness, we denote the set of pairs
  associated to $\varphi$ by the bold-faced {\boldmath $\varphi$}.
}\end{definition}

\begin{definition}
{\bf [associated surrogate, canonical presentation, and image].}
{\rm
 Let
  $$
   {\cal A}_{\varphi}\;
    =\; \cap_{({\cal O}_X\subset {\cal A} \subset {\cal O}^{A\!z},
               f:X_{\cal A}\rightarrow Y)
              \in \mbox{\scriptsize\boldmath $\varphi$}}\,
         {\cal A}\,.
  $$
 Then
  ${\cal O}_X\subset {\cal A}_{\varphi}\subset {\cal O}_X^{A\!z}$
  and there exists a unique
   $f_{\varphi}: X_{\varphi} :=\boldSpec{\cal A}_{\varphi}\rightarrow Y$
  such that the induced diagram
   \begin{eqnarray*}
    \xymatrix{
     X_{\cal A}\ar[drr]^f\ar[d]         &&    \\
     X_{\varphi}\ar[rr]^{f_{\varphi}}   && Y  \\
     }
   \end{eqnarray*}
   commutes,
   for all
    $({\cal O}_X \subset {\cal A} \subset {\cal O}_X^{A\!z},
      f:X_{\cal A}\rightarrow Y) \in$ {\boldmath $\varphi$}.
 We shall call the pair
  $$
   ({\cal O}_X \subset {\cal A}_{\varphi}
               \subset {\cal O}_X^{A\!z}\;,\;
     f_{\varphi}:
      X_{\varphi}:= \boldSpec{\cal A}_{\varphi}\rightarrow Y)\,,
  $$
  which is canonically associated to $\varphi$,
  {\it the (canonical) presentation} for $\varphi$.
 The scheme $X_{\varphi}$, which dominates $X$, is called
  the {\it surrogate of $X^{\!A\!z}$ associated to $\varphi$}.
 We will denote the built-in morphism $X_{\varphi}\rightarrow C$
  by $\pi_{\varphi}$.
 The subscheme $f_{\varphi}(X_{\varphi})$ of $Y$
  is called the {\it image} of $X^{\!A\!z}$ under $\varphi$
  and will be denoted $\Image\varphi$ or $\varphi(X^{\!A\!z})$
  interchangeably.
}\end{definition}

\begin{remark} {\it $[$minimal property of $X_{\varphi}$$\,]$.}
{\rm
 By construction,
  \begin{itemize}
   \item[$\cdot$] {\it
    there exists no ${\cal O}_X$-subalgebra
    ${\cal O}_X \subset {\cal A}^{\prime}\subset {\cal A}_{\varphi}$
    such that
     $f_{\varphi}$ factors as the composition of morphisms
     $X_{\varphi}
      \rightarrow \boldSpec{\cal A}^{\prime} \rightarrow Y$}.
  \end{itemize}
 We will call this feature
  the {\it minimal property} of the surrogate
  $X_{\varphi}$ of $X^{\!A\!z}$ associated to $\varphi$.
}\end{remark}

\begin{definition}
{\bf [Azumaya scheme with a fundamental module, morphism].} {\rm
 An {\it Azumaya scheme with a fundamental module}
  is a triple $(X,{\cal O}_X^{\!A\!z}, {\cal E})$,
  denoted also as $(X^{\!A\!z},{\cal E})$,
  where
   ${\cal E}$ is a locally free ${\cal O}_X$-module of rank $r$ and
   ${\cal O}_X^{\!A\!z}=\Endsheaf_{{\cal O}_X}({\cal E})$.
 (I.e.\ ${\cal E}$ is equipped with
        a fixed (left)-${\cal O}_X^{A\!z}$-module structure.)
 A {\it morphism} $\varphi:(X^{\!A\!z},{\cal E})\rightarrow Y$
  is simply a morphism from $X^{\!A\!z}$ to $Y$
  as defined in Definition~2.2.2.
}\end{definition}

\begin{definition}
{\bf [Chan-Paton sheaf/module].}
{\rm
 Given a morphism
  $\varphi: (X^{A\!z},{\cal E}) \rightarrow Y$
  with its canonical presentation
  $({\cal O}_X \subset {\cal A}_{\varphi}
               \subset {\cal O}_X^{A\!z}\;,\;
     f_{\varphi}:X_{\varphi}\rightarrow Y)$.
 Then ${\cal E}$ is automatically a ${\cal O}_{X_{\varphi}}$-module
  in notation, $_{{\cal O}_{X_{\varphi}}}{\cal E}$.
 Define the {\it push-forward} $\varphi_{\ast}{\cal E}$
  of ${\cal E}$ to $Y$ under $\varphi$ by
  $f_{\varphi\,\ast}(_{{\cal O}_{X_{\varphi}}}{\cal E})$.
 It is a coherent ${\cal O}_Y$-module
  supported on $\Image\varphi = f_{\varphi}(X_{\varphi})$.
 We will call it also the {\it Chan-Paton sheaf or module
  on $\varphi(X^{A\!z})$ associated to ${\cal E}$ under $\varphi$}.
} \end{definition}

\begin{definition}
{\bf [isomorphism between morphisms].}
{\rm
 Two morphisms
    $\varphi_1:(X_1^{\!A\!z},{\cal E}_1)\rightarrow Y$ and
    $\varphi_2:(X_2^{\!A\!z},{\cal E}_2)\rightarrow Y$
  from Azumaya schemes with a fundamental module to $Y$
  are said to be {\it isomorphic}
 if there exists an isomorphism
   $h:X_1 \stackrel{\sim}{\rightarrow} X_2$
   with a lifting
    $\widetilde{h}:
     {\cal E}_1 \stackrel{\sim}{\rightarrow} h^{\ast}{\cal E}_2$
   such that
    \begin{itemize}
     \item[$\cdot$]
      $\widetilde{h}:
       {\cal A}_{\varphi_1}
       \stackrel{\sim}{\rightarrow} h^{\ast}{\cal A}_{\varphi_2}$,

     \item[$\cdot$]
      the following diagram commutes
      \begin{eqnarray*}
       \xymatrix{
         X_{\varphi_2}\ar[drr]^{f_{\varphi_2}}\ar[d]_{\widehat{h}}
                                                          &&      \\
         X_{\varphi_1}\ar[rr]^{f_{\varphi_1}}             && Y\; .\\
       }
      \end{eqnarray*}
    \end{itemize}
  Here,
  we denote the induced isomorphism
    ${\cal O}_{X_1}^{A\!z} \stackrel{\sim}{\rightarrow}
                             h^{\ast}{\cal O}_{X_2}^{A\!z}$
    of ${\cal O}_{X_1}$-algebras
    (or ${\cal A}_1 \stackrel{\sim}{\rightarrow} h^{\ast}{\cal A}_2$
        of their respective
        ${\cal O}_{X_{\bullet}}$-subalgebras in question)
    via
     $\widetilde{h}:{\cal E}_1
       \stackrel{\sim}{\rightarrow} h^{\ast}{\cal E}_2$
   still by $\widetilde{h}$ and
  $\widehat{h}:X_{\varphi_2}\stackrel{\sim}{\rightarrow} X_{\varphi_1}$
   is the scheme-isomorphism associated to
  $\widetilde{h}: {\cal A}_{\varphi_1}
    \stackrel{\sim}{\rightarrow} h^{\ast}{\cal A}_{\varphi_2}$.
} \end{definition}

The notion of a {\it family of morphisms} from
 (nonfixed) Azumaya schemes with a fundamental module to $Y$
 can also be defined accordingly.
We refer readers to [L-L-S-Y: Sec.~2.1].
When this fundamental setting is translated to the following three
 equivalent settings in the realm of commutative algebraic geometry,
it becomes standard how the family version should be formulated.

\bigskip

\begin{flushleft}
{\bf II.\ As a torsion sheaf on $X\times Y$.}
\end{flushleft}
The minimal property of $X_{\varphi}$ implies that
 the map $(\pi_{\varphi},f_{\varphi}):X_{\varphi}\rightarrow X\times Y$
 is indeed an embedding and $_{{\cal O}_{X_{\varphi}}}{\cal E}$ can be
 identified as a torsion sheaf $\widetilde{\cal E}$ on $X\times Y$
 that is flat over $X$ of relative length $r$.
The converse also holds:

\begin{lemma}
{\bf [Azumaya without Azumaya, morphisms without morphisms].}
 A morphism $\varphi:(X^{\!A\!z},{\cal E})\rightarrow Y$
   from an Azumaya scheme $X^{\!A\!z}$ over $X$
    with a fundamental module ${\cal E}$ of rank $r$ is given by
  a coherent ${\cal O}_{X\times Y}$-module $\widetilde{\cal E}$
   on $(X\times Y)/X$ that is flat over $X$ of relative length $r$.
\end{lemma}

\begin{proof}
 Let
  $\pr_1:X\times Y\rightarrow X$ and $\pr_2:X\times Y\rightarrow Y$
  be the projection maps.
 Then ${\cal E}$ is recovered by $\pr_{1\ast}\widetilde{\cal E}$.
  ${\cal O}_X^{A\!z}$ is thus also recovered.
 The scheme-theoretical support $\Supp(\widetilde{\cal E})$
  gives $X_{\varphi}$ with $\pi_{\varphi}$ and $f_{\varphi}$
  recovered by the restriction of $\pr_1$ and $\pr_2$ respectively.
 As $X_{\varphi}$ is already embedded in $X\times Y$,
  the minimal property for the surrogate associated to a morphism
  is automatically satisfied.

\end{proof}

This equivalent translation of the notion `morphism' to the realm of
 commutative algebraic geometry is technically important.
However, to relate to D-branes directly, it is conceptually important
 to keep the Azumaya geometry in the fundamental settings in mind.

\bigskip

\begin{flushleft}
{\bf III.\ As a map to the stack
           ${\frak M}^{\,0^{Az^f}}_{\,r}\!\!(Y)$.}
\end{flushleft}
Let ${\frak M}^{\,0^{Az^f}}_{\,r}\!\!(Y)$ be the moduli stack of
 morphisms from an Azumaya point
 with a fundamental module of rank $r$ to $Y$.
{From} Aspect~II of morphisms,
 ${\frak M}^{\,0^{Az^f}}_{\,r}\!\!(Y)$ is identical to
 the Artin stack of $0$-dimensional ${\cal O}_Y$-modules
 of length $r$.
Interpret Aspect~II of morphism as a flat family of $0$-dimensional
 ${\cal O}_Y$-module over $X$,
then a morphism $\varphi:(X^{\!A\!z},{\cal E})\rightarrow Y$
 defines a morphism
 $\phi:X\rightarrow {\frak M}^{\,0^{Az^f}}_{\,r}\!\!(Y)$;
and vice versa.

\bigskip

\begin{flushleft}
{\bf IV.\ As a $\GL_r({\Bbb C})$-equivariant map.}
\end{flushleft}
A morphism $\phi:X\rightarrow {\frak M}^{\,0^{Az^f}}_{\,r}\!\!(Y)$
 is equivalent to a morphism $\widetilde{\phi}$ of schemes
 in the commutative diagram
 $$
  \xymatrix{
   \mbox{\bf Isom}\,(\phi,\pi)
    \ar[r]^-{\widetilde{\phi}} \ar[d]_-{pr_1}
     & \mbox{\it Atlas} \ar[d]^-{\pi} \\
    X \ar[r]^-{\phi} & {\frak M}^{\,0^{Az^f}}_{\,r}\!\!(Y)
  }
 $$
 from the Isom-functor construction.
Here
 $\pi:\mbox{\it Atlas}\rightarrow {\frak M}^{\,0^{Az^f}}_{\,r}\!\!(Y)$
 is an atlas of ${\frak M}^{\,0^{Az^f}}_{\,r}\!\!(Y)$.
As the moduli stack of
  $0$-dimensional ${\cal O}_Y$-modules of length $r$,
 one can choose {\it Atlas} in the diagram to be
 the open subscheme
 $$
  \Quot^{\,H^0}\!({\cal O}_Y^{\oplus r},r)\;
   :=\;
   \{\, {\cal O}_Y^{\oplus r}
          \rightarrow \widetilde{\cal E}\rightarrow 0\,,\;
        \length\widetilde{\cal E}=r\,,\;
        H^0({\cal O}_Y^{\oplus r})
          \rightarrow H^0(\widetilde{\cal E})\rightarrow 0\,
      \}
 $$
 of the Quot-scheme $\Quot({\cal O}_Y^{\oplus r},r)$
 of isomorphism classes of $0$-dimensional quotients
 of ${\cal O}_Y^{\oplus}$ with length $r$.
The latter is known to be projective; thus $\Quot^{\,H^0}\!({\cal
O}_Y^{\oplus r},r)$ is quasi-projective
 and it parameterizes morphisms from
 an Azumaya point with a fundamental module $V$ of rank $r$
  together with a decoration
  ${\Bbb C}^r\stackrel{\sim}{\rightarrow} V$.
The $\GL_r({\Bbb C})$-action on $H^0({\cal O}_Y^{\oplus r})$
 induces a right $\GL_r({\Bbb C})$-action
 on $\Quot^{\,H^0}\!({\cal O}_Y^{\oplus r},r)$.

With this choice of atlas for ${\frak M}^{\,0^{Az^f}}_{\,r}\!\!(Y)$,
 $\pr_1:\mbox{\bf Isom}\,(\phi,\pi)\rightarrow X$ becomes a principal
  $\GL_r({\Bbb C})$-bundle $\pr:P\rightarrow X$ over $X$  and
 $\widetilde{\phi}:P\rightarrow \Quot^{\,H^0}\!({\cal O}_Y^{\oplus r},r)$
  is a $\GL_r({\Bbb C})$-equivariant morphism.

Conversely,
 given a principal $\GL_r({\Bbb C})$-bundle $P$ over $X$ and
 a $\GL_r({\Bbb C})$-equivariant morphism
 $\widetilde{\phi}:P
   \rightarrow \Quot^{\,H^0}\!({\cal O}_Y^{\oplus r},r)$.
The pullback of the universal sheaf to $(P\times Y)/X$ and
 the basic descent theory reproduce the torsion sheaf on $X\times Y$
 in Aspect~II.

\vspace{3em}

\begin{flushleft}
{\bf $\Quot^{\,H^0}\!({\cal O}_Y^{\oplus r},r)$
     as the representation-theoretical atlas.}
\end{flushleft}
As Aspect~IV and $\Quot^{\,H^0}\!({\cal O}_Y^{\oplus r},r)$
 will play some role in Sec.~4, let us discuss more about it.

\begin{definition}
{\bf [representation-theoretical atlas].} {\rm
We shall call the quasi-projective scheme
 $\Quot^{\,H^0}\!({\cal O}_Y^{\oplus r},r)$
 the {\it representation-theoretical atlas} of
 ${\frak M}^{\,0^{Az^f}}_{\,r}\!\!(Y)$.
}\end{definition}

\noindent
The purpose of this theme is to explain
 why $\Quot^{\,H^0}\!({\cal O}_Y^{\oplus r},r)$ generalizes
 the notion of the `representation scheme for an algebra'
 and, hence, the above name.

\begin{lemma}
{\bf [from global to local].}
 There exists a finite cover
  $\amalg_{\alpha} U_{\alpha}\rightarrow Y$ of $Y$
  by affine open subsets
 such that any morphism $\varphi:(\pt^{A\!z},V)\rightarrow Y$
  with $V$ of rank $r$ must have its image $\varphi(\pt^{A\!z})$
  contained in some $U_{\alpha}$.
\end{lemma}

\begin{proof}
 Since $Y$ is projective, we only need to prove the lemma
  for $Y={\Bbb P}^n$ for some $n$.
 In this case, take any $k$-many distinct hyperplanes
  $H_{\alpha}$, $\alpha=1,\,\ldots,\,k$, in general positions
  in ${\Bbb P}^n$ (i.e.\ any $n+1$ of them has empty intersection)
  with $k>nr$ and let $U_{\alpha}={\Bbb P}^n-H_{\alpha}$.
 The finite cover $\amalg_{\alpha} U_{\alpha}\rightarrow {\Bbb P}^n$
  of ${\Bbb P}^n$ will do.

\end{proof}

Let
 $\amalg_{\alpha} U_{\alpha}\rightarrow Y$ be such a cover of $Y$
  with $U_{\alpha}=\Spec R_{\alpha}$  and
 $\Rep(R_{\alpha},M_r({\Bbb C}))$
  be the representation-scheme
  that parameterizes ${\Bbb C}$-algebra-homomorphisms
   from $R_{\alpha}$ to $M_r({\Bbb C})$.
It follows from Sec.~2.1 that
 $\Rep(R_{\alpha},M_r({\Bbb C}))$
  is precisely the moduli space of morphisms
  from an Azumaya point with a fundamental module $V$ of rank $r$
  with a decoration ${\Bbb C}^r\stackrel{\sim}{\rightarrow} V$.
Consequently,
 $\amalg_{\alpha}\,\Rep(R_{\alpha},M_r({\Bbb C}))
    \rightarrow \Quot^{\,H^0}\!({\cal O}_Y^{\oplus r},r)$
 gives a finite cover of $\Quot^{\,H^0}\!({\cal O}_Y^{\oplus r},r)$.
In this sense, $\Quot^{\,H^0}\!({\cal O}_Y^{\oplus r},r)$
 generalizes the notion of `representation-schemes'.

When $R_{\alpha}$ is presented in a generator-relator form:
 $$
  R_{\alpha}\;=\;
  {\Bbb C}[z_{\alpha,1},\,\cdots\,z_{\alpha,k_{\alpha}}]/
    (h_{\alpha,1},\,\cdots\,h_{\alpha,l_{\alpha}})\,,
 $$
 where $h_{\alpha,j}$'s are polynomials in $z_{\alpha,i}$'s,
the representation-scheme $\Rep(R_{\alpha},M_r({\Bbb C})$
 is realized accordingly in a standard way
as an affine subscheme in ${\Bbb A}^{r^2k_{\alpha}}$
 described by the ideal generated
 by entries in a system of
 $(\mbox{\tiny\hspace{-1.6ex}$\begin{array}{c} k_{\alpha}\\
                              2\end{array}$}\hspace{-.8ex})+l_{\alpha}$
 matrix polynomials associated to
  the commutativity among $z_{\alpha,i}$'s and
  the relators $h_{\alpha,j}$.
{From} this aspect, one may think of
 $\Quot^{\,H^0}\!({\cal O}_Y^{\oplus r},r)$
 as an enhancement of $Y$,
 with the gluing law between affine charts of $Y$
  enhanced to a corresponding system of matrix gluing law
   between representation-schemes
    associated to the affine charts of $Y$.

It should be noted that
while $\Quot^{\,H^0}\!({\cal O}_Y^{\oplus r},r)$
 has such an elegant intrinsic representation-theoretical meaning,
for $Y$ of dimension $\ge 3$ and $r>\!>0$
the detail of $\Quot^{\,H^0}\!({\cal O}_Y^{\oplus r},r)$
 is beyond any means of approaching at the moment.

\bigskip

\subsection{D-branes in a ${\boldmath B}$-field background
            \`{a} la Polchinski-Grothendieck Ansatz.}

A $B$-field on a space-time $Y$ is a connection on
 a gerbe ${\cal Y}$ over $Y$.
It can be presented
 as a \v{C}ech $0$-cochain $(B_i)_i$ of local $2$-forms $B_i$
 with respect to a cover ${\cal U}=\{U_i\}_i$ on $Y$
 such that on $U_i\cap U_j$,
$B_i-B_j=d\Lambda_{ij}$ for some real $1$-forms $\Lambda_{ij}$
 that satisfies
 $\Lambda_{ij}+\Lambda_{jk}+\Lambda_{ki}=-\sqrt{-1}d\log\alpha_{ijk}$
  on $U_i\cap U_j \cap U_k$,
 where $(\alpha_{ijk})_{ijk}$ is a \v{C}eck $2$-cocycle of
  $U(1)$-valued functions on $Y$
In the algebro-geometric language,
 $(\alpha_{ijk})_{ijk}$ is given by a presentation of
 an equivalence class
 $\alpha_B\in \check{C}^2_{\et}(Y,{\cal O}_Y^{\ast})$
 of \'{e}tale \v{C}ech $2$-cocycles with values in ${\cal O}_Y^{\ast}$.
Through its coupling to the open-string current
 on an open-string world-sheet
 with boundary on a D-brane world-volume $X\subset Y$,
a background $B$-field on $Y$ induces
 a twist to the gauge field $A$
 on the Chan-Paton vector bundle $E$ on $X$
 that renders $E$ itself a twisted vector bundle
  with the twist specified by
  $\alpha_B|_X\in \check{C}^2_{\et}(X,{\cal O}_X^{\ast})$;
 e.g.\ [Hi2], [Ka1], and [Wi4].
Furthermore,
 the $2$-point functions on the open-string world-sheet
  with boundary on $X$ indicate that
 the D-brane world-volume is deformed
  to a deformation-quantization type noncommutative geometry
  in a way that is governed by the $B$-field
  (and the space-time metric);
cf.~[C-H1], [C-H2], [Ch-K], [Scho], and [S-W].

We review here
how these two effects from a $B$-field background
  on the target-space(-time) $Y$
 modify the notion
  of Azumaya schemes with a fundamental module and
  of morphisms therefrom.
More general discussions, details, and references
 are referred to [L-Y6] (D(5)).

\bigskip

\subsubsection{Nontrivial Azumaya noncommutative schemes
    with a fundamental module and morphisms therefrom.}

\begin{flushleft}
{\bf Polchinski-Grothendieck Ansatz with the \'{e}tale
     topology adaptation.}
\end{flushleft}
Recall the Polchinski-Grothendieck Ansatz in Sec.~1.1.
{For} the moment,
 let $X\hookrightarrow Y$ be an embedded submanifold of $Y$.
In the {\it smooth differential-geometric setting} of Polchinski,
 the word ``{\it locally}" in the ansatz means
 ``locally in the {\it $C^{\infty}$-topology}".
This can be generalized to adapt the ansatz to fit various settings:
 ``locally" in the {\it analytic} (resp.\ {\it Zariski}) {\it topology}
 for the {\it holomorphic} (resp.\ {\it algebro-geometric}) {\it setting}.
These are enough to study D-branes in a space(-time) without
 a background $B$-field.
The Azumaya structure sheaf ${\cal O}_X^{A\!z}$
 that encodes the matrix-type noncommutative structure on $X$
 in these cases is of the form $\Endsheaf_{{\cal O}_X}({\cal E})$
 with ${\cal E}$ the Chan-Paton module,
 a locally free ${\cal O}_X$-module of rank $r$
 on which the Azumaya ${\cal O}_X$-algebra ${\cal O}_X^{A\!z}$
 acts tautologically
 as a simple/fundamental (left) ${\cal O}_X^{A\!z}$-module.
This leads to the case reviewed in Sec.~2.2.
${\cal O}_X^{A\!z}$ in this case corresponds to the zero-class
 in the Brauer group $\Br(X)$ of $X$.

On pure mathematical ground, one can further adapt the ansatz
 for $X$ equipped with any Grothendieck topolgy/site.
On string-theoretic ground, as recalled at the beginning of
 the current subsection,
when a background $B$-field on $Y$ is turned on,
 the Chan-Paton module ${\cal E}$ on $X$ becomes twisted and
 is no longer an honest sheaf of ${\cal O}_X$-modules on $X$.
The interpretation of ``{\it locally}" in the ansatz
  in the sense of (small) {\it \'{e}tale topology} on $X$
 becomes forced upon us.
This corresponds to the case
 when the Azumaya structure sheaf ${\cal O}_X^{A\!z}$ on $X$
 represents a non-zero class in $\Br(X)$.
We will call the resulting $(X,{\cal O}_X^{A\!z})$
 a {\it nontrivial Azumaya (noncommutative) scheme}.

In this subsubsection,
we review the most basic algebro-geometric aspect of D-branes
 along the line of the {\it Polchinski-Grothendieck Ansatz}
 but {\it with} this {\it \'{e}tale topology adaptation}
  on the D-brane or D-brane world-volume.
Readers are referred
 to [L-Y6: Sec.~1] (D(5)) for a highlight on gerbes
  and (general) Azumaya algebras over a scheme;  and
 to, e.g., [Br], [Ch], [C\u{a}], [Lie], and [Mi]
  for detailed treatment.
In this review, we will confine ourselves only
 to the language of twisted sheaves.

\bigskip

\begin{flushleft}
{\bf Twisted sheaves \`{a} la C\u{a}ld\u{a}raru.}
\end{flushleft}
Given an \'{e}tale cover $p:U^{(0)}:=\amalg_{i\in I}U_i\rightarrow X$
  of $X$,
 we will adopt the following notations:
 \begin{itemize}
  \item[$\cdot$]
   $U_{ij}\, :=\, U_i\times_X U_j\, =:\, U_i\cap U_j\,$,
   $\; U_{ijk}\, :=\, U_i\times_X U_j\times_X U_k\,
                  =:\, U_i\cap U_j\cap U_k\,$;

  \item[$\cdot$]
   $\xymatrix{
     \cdots\;\ar@<1.2ex>[r] \ar@<.4ex>[r]
             \ar@<-.4ex>[r] \ar@<-1.2ex>[r]
      & U^{(2)}
        := U\times_X U\times_X U
    }$\\
   $\xymatrix{
      \hspace{10em}
      \ar@<.8ex>[rrr]^-{p_{12},\, p_{13},\, p_{23}}
           \ar[rrr] \ar@<-.8ex>[rrr]
      &&& U^{(1)}
          := U\times_X U
          \ar@<.4ex>[rr]^-{p_1,\, p_2} \ar@<-.4ex>[rr]
      && U^{(0)} \ar[r]^-p & X
     }$\\ \\
    are the projection maps from fibered products as indicated;
   the restriction of these projections maps to respectively
    $U_{ijk}$ and $U_{ij}$ will be denoted the same;

  \item[$\cdot$]
   the pull-back of an ${\cal O}_{U_i}$-module ${\cal F}_i$ on $U_i$
    to $U_{ij}$, $U_{ji}$, $U_{ijk}, \,\cdots\,$
    via compositions of these projection maps will be denoted by
   ${\cal F}_i|_{U_{ij}}$, ${\cal F}_i|_{U_{ji}}$,
   ${\cal F}_i|_{U_{ijk}}$, $\cdots\,$ respectively.
 \end{itemize}

\begin{ssdefinition}
{\bf [$\alpha$-twisted ${\cal O}_X$-module on an \'{e}tale cover of $X$].}
{\rm ([C\u{a}: Definition~1.2.1].)} {\rm
 Let $\alpha\in \check{C}^2_{\et}(X,{\cal O}_X^{\ast})$
  be a \v{C}ech $2$-cocycle in the \'{e}tale topology of $X$.
 An {\it $\alpha$-twisted ${\cal O}_X$-module on an \'{e}tale cover of}
  $X$ is a triple
  $$
   {\cal F}\; =\; ( \{U_i\}_{i\in I},\,
                    \{{\cal F}_i\}_{i\in I},\,
                    \{\phi_{ij}\}_{i,j\in I} )
  $$
  that consists of the following data
  \begin{itemize}
   \item[$\cdot$]
    an \'{e}tale cover
     $p:U:=\amalg_{i\in I}U_i\rightarrow X$ of $X$
     on which $\alpha$ can be represented as a $2$-cocycle:\\
     $$
      \alpha\; =\;
       \{\, \alpha_{ijk}\,:\,
        \alpha_{ijk} \in \Gamma( U_{ijk}, {\cal O}_X^{\ast} )\;
       \mbox{with
        $\alpha_{jkl}\alpha_{ikl}^{-1}\alpha_{ijl}\alpha_{ijk}^{-1}=1$
        on $U_{ijkl}$ for all $i,j,k,l\in I$} \,\}\,,
     $$
    such a cover will be called
     an {\it $\alpha$-admissible \'{e}tale cover} of $X$;

   \item[$\cdot$]
    ${\cal F}_i$ is a sheaf of ${\cal O}_{U_i}$-modules on $U_i$;

   \item[$\cdot$] ({\it gluing data})$\;$
    $\phi_{ij}: {\cal F}_i|_{U_{ij}} \rightarrow {\cal F}_j|_{U_{ij}}$
     is an ${\cal O}_{U_{ij}}$-module isomorphism
    that satisfies
    \begin{itemize}
     \item[(1)]
      $\phi_{ii}$ is the identity map for all $i\in I$;

     \item[(2)]
      $\phi_{ij}=\phi_{ji}^{-1}$ for all $i,j\in I$;

     \item[(3)] ({\it twisted cocycle condition})\;
      $\phi_{ki}\circ \phi_{jk}\circ \phi_{ij}$
       is the multiplication by $\alpha_{ijk}$ on ${\cal F}_i|_{U_{ijk}}$.
    \end{itemize}
  \end{itemize}
 ${\cal F}$ is said to be {\it coherent}
   (resp.\ {\it quasi-coherent}, {\it locally free})
  if ${\cal F}_i$ is a coherent
   (resp.\ quasi-coherent, locally free)
   ${\cal O}_{U_i}$-module for all $i\in I$.
 A {\it homomorphism}
   $$
    h\;:\; {\cal F}\;=\; ( \{U_i\}_{i\in I},\,
                     \{{\cal F}_i\}_{i\in I},\,
                     \{\phi_{ij}\}_{i,j\in I} )\;
      \longrightarrow\;
     {\cal F}^{\prime}\; =\; ( \{U_i\}_{i\in I},\,
                     \{{\cal F}^{\prime}_i\}_{i\in I},\,
                     \{\phi^{\prime}_{ij}\}_{i,j\in I} )
    $$
   between $\alpha$-twisted ${\cal O}_X$-modules
   on the \'{e}tale cover $p$ of $X$
  is a collection
   $\{h_i:{\cal F}_i\rightarrow {\cal F}^{\prime}_i\}_{i\in I}$,
    where $h_i$ is an ${\cal O}_{U_i}$-module homomorphism,
   such that $\phi^{\prime}_{ij}\circ h_i= h_j\circ \phi_{ij}$
    for all $i,j\in I$.
 In particular, $h$ is an {\it isomorphism} if all $h_i$ are isomorphisms.
 Denote by $\ModCategory(X,\alpha, p)$
  the category of $\alpha$-twisted ${\cal O}_X$-modules
  on the \'{e}tale cover $p:U^{(0)}\rightarrow X$ of $X$.
}\end{ssdefinition}

Given an $\alpha$-twisted sheaf ${\cal F}$ on the \'{e}tale cover
  $p:U\rightarrow X$ of $X$.
let $p^{\prime}:U^{\prime}\rightarrow X$
  be an \'{e}tale refinement of $p:U\rightarrow X$.
Then
 $\alpha$ can be represented also on $p^{\prime}:U^{\prime}\rightarrow X$
  and
 ${\cal F}$ on $p$ defines an $\alpha$-twisted ${\cal O}_X$-module
  ${\cal F}^{\prime}$ on $p^{\prime}$ via the pull-back
 under the built-in \'{e}tale cover $U^{\prime}\rightarrow U$ of $U$.
This defines an equivalence of categories:
 $$
  \ModCategory(X,\alpha,p)\;
   \longrightarrow\; \ModCategory(X,\alpha,p^{\prime})\,.
 $$
([C\u{a}: Lemma~1.2.3, Lemma~1.2.4, Remark~1.2.5].)

\begin{ssdefinition}
{\bf [$\alpha$-twisted ${\cal O}_X$-module on $X$].} {\rm
 An {\it $\alpha$-twisted ${\cal O}_X$-module on $X$}
  is an equivalence class $[{\cal F}]$
  of $\alpha$-twisted ${\cal O}_X$-modules ${\cal F}$
  on \'{e}tale covers of $X$,
  where the equivalence relation is generated by \'{e}tale refinements
   and descents by
   \"{e}tale covers of $X$ on which $\alpha$ can be represented.
 An ${\cal F}^{\prime}\in [{\cal F}]$ is called a {\it representative}
  of the $\alpha$-twisted ${\cal O}_X$-module $[{\cal F}]$.
 For simplicity of terminology,
  we will also call ${\cal F}^{\prime}$ directly
  an $\alpha$-twisted ${\cal O}_X$-module on $X$.
}\end{ssdefinition}

\noindent
Cf.\ [C\u{a}: Corollary~1.2.6 and Remark~1.2.7].

\bigskip

Standard notions of ${\cal O}_X$-modules,
 in particular
  \begin{itemize}
   \item[$\cdot$]
    the {\it scheme-theoretic support} $\Supp{\cal E}$,

   \item[$\cdot$]
    the {\it dimension} $\dimm{\cal E}$, and

   \item[$\cdot$]
    {\it flatness over a base $S$}
  \end{itemize}
 of an $\alpha$-twisted sheaf ${\cal E}$ on $X$ (or on $X/S$)
 are defined via a(ny) presentation of ${\cal E}$
 on an $\alpha$-admissible \'{e}tale cover $U\rightarrow X$.

Standard operations on ${\cal O}_{\bullet}$-modules apply to
 twisted ${\cal O}_{\bullet}$-modules
 on appropriate admissible \'{e}tale covers
 by applying the operations component by component over the cover.
These operations apply then to twisted ${\cal O}_X$-modules as well:
They are defined on representatives of twisted sheaves
 in such a way that they pass to each other
  by pull-back and descent under \'{e}tale refinements
  of admissible \'{e}tale covers.
In particular:

\begin{ssproposition}
{\bf [basic operations on twisted sheaves].}
{\rm ([C\u{a}: Proposition~1.2.10].)}
 (1)
 Let ${\cal F}$ and ${\cal G}$ be
  an $\alpha$-twisted and a $\beta$ -twisted ${\cal O}_X$-module
  respectively,
  where $\alpha,\beta\in\check{C}^2_{\et}(X,{\cal O}_X^{\ast})$.
 Then
  ${\cal F}\otimes_{{\cal O}_X}{\cal G}$ is an $\alpha\beta$-twisted
   ${\cal O}_X$-module  and
  $\Homsheaf_{{\cal O}_X}({\cal F},{\cal G})$
   is an $\alpha^{-1}\beta$-twisted ${\cal O}_X$-module.
 In particular,
  if ${\cal F}$ and ${\cal G}$ are both
   $\alpha$-twisted ${\cal O}_X$-modules,
  then $\Homsheaf_{{\cal O}_X}({\cal F},{\cal G})$
   descends to an (ordinary/untwisted) ${\cal O}_X$-module,
   still denoted by $\Homsheaf_{{\cal O}_X}({\cal F},{\cal G})$,
   on $X$.

 (2)
 Let $f:X\rightarrow Y$ be a morphism of schemes$/{\Bbb C}$
  and $\alpha\in\check{C}^2_{\et}(Y,{\cal O}_Y^{\ast})$.
 Note that
  an $\alpha$-admissible \'{e}tale cover of $Y$ pulls back to
  an $f^{\ast}\alpha$-admissible \'{e}tale over of $X$ under $f$,
 through which the pull-back and push-forward of a related
  twisted sheaf can be defined.
 If ${\cal F}$ is an $\alpha$-twisted ${\cal O}_Y$-module on $Y$,
  then $f^{\ast}{\cal F}$ is an $f^{\ast}\alpha$-twisted
   ${\cal O}_X$-module on $X$.
 If ${\cal F}$ is an $f^{\ast}\alpha$-twisted
   ${\cal O}_X$-module on $X$,
  then $f_{\ast}{\cal F}$ is an $\alpha$-twisted
   ${\cal O}_Y$-module on $Y$.
\end{ssproposition}

\bigskip

\begin{flushleft}
{\bf Morphisms from Azumaya schemes with a twisted fundamental module.}
\end{flushleft}

\begin{ssdefinition}
{\bf [Azumaya scheme with a fundamental module].}
{\rm
 An {\it Azumaya scheme with a fundamental module in class $\alpha$}
  is a tuple
  $$
   (X^{A\!z},{\cal E})\;
    :=\; (X,\,
          {\cal O}_X^{A\!z}
            = \Endsheaf_{{\cal O}_X}({\cal E}),\,
          {\cal E})\,,
  $$
  where
   $X=(X,{\cal O}_X)$ is a (Noetherian) scheme (over ${\Bbb C}$),
   $\alpha\in \check{C}^2_{\et}(X,{\cal O}_X^{\ast})$ represents
    a class $[\alpha]\in\Br(X)\subset H^2_{\et}(X,{\cal O}_X^{\ast})$,
    and
   ${\cal E}$ is a locally-free coherent $\alpha$-twisted
    ${\cal O}_X$-module on $X$.
 A {\it commutative surrogate} of $(X^{A\!z},{\cal E})$
  is a scheme $X_{\cal A}:=\boldSpec{\cal A}$,
  where
   ${\cal O}_X\subset {\cal A}\subset \Endsheaf_{{\cal O}_X}({\cal E})$
   is an inclusion sequence of commutative ${\cal O}_X$-subalgebras
   of $\Endsheaf_{{\cal O}_X}({\cal E})$.
 Let $\pi: X_{\cal A}\rightarrow X$
  be the built-in dominant finite morphism.
 Then
  $\cal E$ is tautologically a $\pi^{\ast}\alpha$-twisted
   ${\cal O}_{X_{\cal A}}$-module on $X_{\cal A}$,
   denoted by $_{{\cal O}_{X_{\cal A}}}{\cal E}$.
 We say that $X^{A\!z}$ is an {\it Azumaya scheme of rank $r$}
   if ${\cal E}$ has rank $r$ and
  that it is a {\it nontrivial} (resp.\ {\it trivial})
   Azumaya scheme if $[\alpha]\ne 0$ (resp.\ $[\alpha]=0$).
}\end{ssdefinition}

Let
 $Y$ be a (commutative, Noetherian) scheme/${\Bbb C}$  and
 $\alpha_B\in \check{C}^2_{\et}(Y,{\cal O}_Y^{\ast})$
  be the \'{e}tale \v{C}ech cocycle associated to a fixed $B$-field
  on $Y$.

\begin{ssdefinition}
{\bf [morphism with {\boldmath $B$}-field background].}
{\rm
 Let $(X^{A\!z},{\cal E})$ be an Azumaya scheme with a fundamental module
  in the class $\alpha\in\check{C}^2_{\et}(X,{\cal O}_X^{\ast})$.
 Then,
 a {\it morphism} from $(X^{A\!z},{\cal E})$ to $(Y,\alpha_B)$,
  in notation $\varphi: (X^{A\!z},{\cal E})\rightarrow (Y,\alpha_B)$,
  is a pair
  $$
   ({\cal O}_X \subset {\cal A}_{\varphi}
               \subset {\cal O}_X^{\Azscriptsize}\;,\;
     f_{\varphi}: X_{\varphi}:=\boldSpec{\cal A}_{\varphi}
                  \rightarrow Y)\,,
  $$
  where
  \begin{itemize}
   \item[$\cdot$]
    ${\cal A}_{\varphi}$ is a commutative ${\cal O}_X$-subalgebra
    of ${\cal O}_X^{A\!z}$,

   \item[$\cdot$]
    $f_{\varphi}:X_{\varphi} \rightarrow Y$
    is a morphism of (commutative) schemes,
  \end{itemize}
  that satisfies the following properties:
  \begin{itemize}
   \item[(1)]
   ({\it minimal property of $X_{\varphi}$})$\;$
    there exists no ${\cal O}_X$-subalgebra
     ${\cal O}_X \subset {\cal A}^{\prime}\subset {\cal A}_{\varphi}$
    such that
     $f_{\varphi}$ factors as the composition of morphisms
     $X_{\varphi}
      \rightarrow \boldSpec{\cal A}^{\prime} \rightarrow Y$;

   \item[(2)]
   ({\it matching of twists on $X_{\varphi}$})$\;$
    let $\pi_{\varphi}:X_{\varphi}\rightarrow X$
     be the built-in finite dominant morphism,
    then $\pi_{\varphi}^{\ast}\alpha=f_{\varphi}^{\ast}\alpha_B$
     in $\check{C}^2_{\et}(X_{\varphi},{\cal O}_{X_{\varphi}}^{\ast})$.
  \end{itemize}
 $X_{\varphi}$ is called the
  {\it surrogate of $X^{\Azscriptsize}$ associated to $\varphi$}.
 Condition (2) implies that
  $\varphi_{\ast}{\cal E}
   := f_{\varphi\,\ast}(_{{\cal O}_{X_{\varphi}}}{\cal E})$
  is an $\alpha_B$-twisted ${\cal O}_Y$-module on $Y$,
  supported on
  $\Image(\varphi) :=\varphi(X^{A\!z})
   := f_{\varphi}(X_{\varphi})$,
   where the last is the usual scheme-theoretic image
    of $X_{\varphi}$ under $f_{\varphi}$.\footnote{In
                              other words,
                              a morphism from $(X^{A\!z},{\cal E})$
                               to $(Y,\alpha_B)$ is a usual morphism
                               $\varphi: X^{A\!z} \rightarrow Y$
                               from the (possibly nontrivial)
                               Azumaya scheme $X^{A\!z}$ to $Y$
                              subject to the twist-matching Condition (2)
                              so that $\varphi_{\ast}{\cal E}$
                               remains a twisted sheaf
                               in a way that is compatible with
                               the $B$-field background on $Y$.}

 Given two morphisms
    $\varphi_1:(X_1^{A\!z},{\cal E}_1)\rightarrow (Y,\alpha_B)$ and
    $\varphi_2:(X_2^{A\!z},{\cal E}_2)\rightarrow (Y,\alpha_B)$,
 a {\it morphism} $\varphi_1\rightarrow\varphi_2$
  from $\varphi_1$ to $\varphi_2$ is a pair $(h, \widetilde{h})$,
  where
  \begin{itemize}
   \item[$\cdot$]
    $h:X_1\rightarrow X_2$ is an isomorphism of schemes
      with $h^{\ast}\alpha_2=\alpha_1$,
      where $\alpha_i$ is the underlying class of ${\cal E}_i$
       in $\check{C}^2_{\et}(X_i,{\cal O}_{X_i}^{\ast})$;

   \item[$\cdot$]
    $\widetilde{h}:{\cal E}_1
                    \stackrel{\sim}{\rightarrow} h^{\ast}{\cal E}_2$
     be an isomorphism of twisted sheaves on $X_1$
     that satisfies
     \begin{itemize}
      \item[$\cdot$]
       $\widetilde{h}:
        {\cal A}_{\varphi_1}
        \stackrel{\sim}{\rightarrow} h^{\ast}{\cal A}_{\varphi_2}$,

      \item[$\cdot$]
       the following diagram commutes
       \begin{eqnarray*}
        \xymatrix{
          X_{\varphi_2}\ar[drr]^{f_{\varphi_2}}\ar[d]_{\widehat{h}}
                                                           &&      \\
          X_{\varphi_1}\ar[rr]^{f_{\varphi_1}}             && Y\; .\\
        }
       \end{eqnarray*}
     \end{itemize}
    Here,
    we denote both of
     the induced isomorphisms,
      $\, {\cal O}_{X_1}^{A\!z} \stackrel{\sim}{\rightarrow}
                              h^{\ast}{\cal O}_{X_2}^{A\!z}$
       and
      ${\cal A}_{\varphi_1}
        \stackrel{\sim}{\rightarrow} h^{\ast}{\cal A}_{\varphi_2}\,$,
      of ${\cal O}_{X_1}$-algebras still by $\widetilde{h}$  and
    $\widehat{h}:X_{\varphi_2}
                 \stackrel{\sim}{\rightarrow} X_{\varphi_1}$
     is the scheme-isomorphism associated to
    $\widetilde{h}: {\cal A}_{\varphi_1}
      \stackrel{\sim}{\rightarrow} h^{\ast}{\cal A}_{\varphi_2}$.
  \end{itemize}

 This defines the category $\MorphismCategory_{A\!z^f}(Y,\alpha_B)$
  of morphisms
  from Azumaya schemes with a fundamental module
  to $(Y,\alpha_B)$.
}\end{ssdefinition}

\begin{ssdefinition}
{\bf [D-brane and Chan-Paton module].} {\rm
 Following the previous Definition,
  $\varphi(X^{A\!z})$
   is called the {\it image D-brane} on $(Y,\alpha_B)$  and
  $\varphi_{\ast}{\cal E}$ the {\it Chan-Paton module/sheaf}
   on the image D-brane.
 Similarly, for {\it image D-brane world-volume}
  if $X$ is served as a (Wicked-rotated) D-brane world-volume.
}\end{ssdefinition}

\bigskip

\begin{flushleft}
{\bf Others aspects of a morphism in the twisted case.}
\end{flushleft}
The above theme gives Aspect I, `The fundamental setting',
 of a morphism.
Aspects II, III, IV of a morphism in Sec.~2.2
 can be generalized to the twisted case as well.
Moreover, there is now a new Aspect~V of a morphism:
 namely, a description in terms of a morphism
 $\breve{\varphi}:({\cal X}^{A\!z},{\cal F})\rightarrow {\cal Y}_{\alpha_B}$
   from the Azumaya ${\cal O}_X^{\ast}$-gerbe
   with a fundamental module, associated to $(X^{\!A\!z},{\cal E})$,
   to the gerbe ${\cal Y}_{\alpha_B}$, associated to $(Y,\alpha_B)$.
As we won't use them for the rest of the current work,
 their discussions are omitted.
We refer readers to [L-Y6: Sec.~2.2] (D(5)) themes:
 `Azumaya without Azumaya and morphisms without morphisms' and
`The description in terms of morphisms from Azumaya gerbes
 with a fundamental module to a target gerbe'
 for a discussion of Aspects~II and V of the twisted case.

\bigskip

\subsubsection{Azumaya quantum schemes with a fundamental module\\
               and morphisms therefrom.}

In this subsubsection, we review
 how the second effect
  - namely, the deformation quantization -
  of the background $B$-field to a smooth D-brane world-volume $X$
 can be incorporated into Azumaya geometry
  along the line of the Polchinski-Grothendieck Ansatz.
We focus on the case when the deformation quantizations that occur
 are modelled directly on that for phase spaces in quantum mechanics.
This brings in the sheaf ${\cal D}$ of differential operators
 and ${\cal D}$-modules.

\bigskip

\begin{flushleft}
{\bf Weyl algebras, the sheaf ${\cal D}$
     of differential operators, and ${\cal D}$-modules.}
\end{flushleft}
Let
 $X$ be a smooth variety over ${\Bbb C}$,
 $\Theta_X=\Dersheaf_{{\Bbb C}}({\cal O}_X,{\cal O}_X)$
  be the sheaf of ${\Bbb C}$-derivations on ${\cal O}_X$,  and
 $\Omega_X$ be the sheaf of K\"{a}hler differentials on $X$.
We recall a few necessary objects and facts for our study.
Their details are referred to
 [Bern], [Bj], and [B-E-G-H-K-M]$\,$:
\begin{itemize}
 \item[(1)]
  the {\it Weyl algebra}
  $$
    A_n({\Bbb C})\; :=\;
     {\Bbb C}\langle x_1,\,\cdots\,,\,x_n,
       \partial_1,\,\cdots\,,\,\partial_n\rangle
         /([x_i,x_j]\,,\, [\partial_i,\partial_j]\,,\,
           [\partial_i,x_j]-\delta_{ij}\, :\, 1\le i,j\le n)\,,
  $$
  which is the algebra of differential operators acting on
   ${\Bbb C}[x_1,\,\cdots\,,\,x_n]$ by formal differentiation;
  here,
    ${\Bbb C}\langle\,\cdots\,\rangle$ is
     the unital associative ${\Bbb C}$-algebra generated
      by elements $\cdots$ indicated,
    $[\;\,,\,\;]$ is the commutator,
    $\delta_{ij}$ is the Kronecker delta, and
    $(\,\cdots\,)$ is the $2$-sided ideal generated by $\cdots$ indicated;

 \item[(2)]
 the {\it sheaf ${\cal D}_X$
  of (linear algebraic) differential operators}
  on $X$,
  which is the sheaf of unital associative algebras
  that extends ${\cal O}_X$ by new generators from the sheaf $\Theta_X$;

 \item[(3)]
  {\it ${\cal D}_X$-modules}
   (or directly {\it ${\cal D}$-modules} when $X$ is understood),
   which are sheaves on $X$ on which ${\cal D}_X$ acts from the left.
\end{itemize}

\smallskip

\begin{sslemma}
{\bf [$A_n({\Bbb C})$ simple].}
 $A_n({\Bbb C})$ is a simple algebra:
  the only $2$-sided ideal therein is the zero ideal $(0)$.
\end{sslemma}

\smallskip

\begin{ssproposition}
{\bf [${\cal O}$-coherent ${\cal D}$-module].}
 Let ${\cal M}$ be a ${\cal D}_X$-module
  that is coherent as an ${\cal O}_X$-module.
 Then, ${\cal M}$ is ${\cal O}_X$-locally-free.
 Furthermore, in this case,
 the action of ${\cal D}_X$ on ${\cal M}$ defines a flat connection
  $\nabla:{\cal M}\rightarrow {\cal M}\otimes\Omega_X$ on ${\cal M}$
  by assigning $\nabla_{\!\xi}\,s=\xi\cdot s$
   for $s\in{\cal M}$ and $\xi\in \Theta_X$;
 the converse also holds.
 This gives an equivalence of categories:
 $$
  \left\{\rule{0em}{1.2em}
   \begin{array}{c}
    \mbox{${\cal O}_X$-coherent ${\cal D}_X$-modules}
   \end{array}
  \right\}\;
  \longleftrightarrow\;
  \left\{
   \begin{array}{l}
     \mbox{coherent locally free ${\cal O}_X$-modules}\\
     \mbox{with a flat connection}
   \end{array}
  \right\}\,.
 $$
\end{ssproposition}

\bigskip

\begin{flushleft}
{\bf ${\cal D}$ as the structure sheaf of
     the deformation quantization of the cotangent bundle.}
\end{flushleft}
{From} the presentation of the Weyl algebra $A_n({\Bbb C})$,
 which resembles the quantization of a classical phase space
  with the position variable $(x_1,\,\cdots\,,\,x_n)$ and
  the dual momentum variable
   $(p_1,\,\cdots\,,\,p_n)=(\partial_1,\,\cdots\,,\,\partial_n)$,
  and
the fact that
 ${\cal D}_X$ is locally modelled on the pull-back of
 $A_n({\Bbb C})$ over ${\Bbb A}^n$ under an \'{e}tale morphism
 to ${\Bbb A}^n$,
the sheaf ${\cal D}_X$ of algebras with the built-in inclusion
 ${\cal O}_X\subset {\cal D}_X$ can be thought of
 as the structure sheaf of a noncommutative space
 from the quantization\footnote{The word
                         ``quantization" has received various meanings
                          in mathematics. Here, we mean solely the
                          one associated to quantum mechanics.
                          This particular quantization is also called
                           {\it deformation quantization}.}
 of the cotangent bundle, i.e.\ the total space
 $\mbox{\boldmath $\Omega$}_X$ of the sheaf $\Omega_X$, of $X$.

\smallskip

\begin{ssdefinition}
{\bf [canonical deformation quantization of cotangent bundle].}
{\rm
 We will formally denote this noncommutative space by
  $\Space {\cal D}_X =: Q\mbox{\boldmath $\Omega$}_X$ and
 call it
  the {\it canonical deformation quantization}
  of $\mbox{\boldmath $\Omega$}_X$.
}\end{ssdefinition}

\smallskip

A special class of {\it morphisms} from or to $\Space {\cal D}_X$
 can be defined contravariantly
 as homomorphisms of sheaves of ${\Bbb C}$-algebras.

\smallskip

\smallskip

\begin{ssexample}
{\bf [$A_n({\Bbb C})$].} {\rm
 The noncommutative space $\Space (A_n({\Bbb C}))$
  defines a deformation quantization of
  $\mbox{\boldmath $\Omega$}_{{\Bbb A}^n}$.
 Recall the presentation of $A_n({\Bbb C})$.
 The ${\Bbb C}$-algebra homomorphism
  $$
   \begin{array}{cccccl}
    f_{(k)}^{\sharp} & :
     & {\Bbb C}[y_1,\,\cdots\,,\,y_n]
     & \longrightarrow     & A_n({\Bbb C}) \\[.6ex]
    && y_i & \longmapsto   & x_i\,,        & i=1,\,\ldots\,,\,k, \\[.6ex]
    && y_j & \longmapsto   & \partial_j\,, & j=k+1,\,\ldots\,,\, n\,,
   \end{array}
   $$
 defines a dominant morphism
  $f_{(k)}:\Space (A_n({\Bbb C}))\rightarrow {\Bbb A}^n$,
  $k=0,\,\ldots\,,\,n$.
 The ${\Bbb C}$-algebra automorphism
  $A_n({\Bbb C})\rightarrow A_n({\Bbb C})$ with
   $x_i\mapsto \partial_i$ and $\partial_i\mapsto -x_i$
   defines the {\it Fourier transform} on $\Space (A_n({\Bbb C}))$.
 Note that, since $A_n({\Bbb C})$ is simple,
  any morphisms to $\Space (A_n({\Bbb C}))$ is dominant
  (i.e.\ the related ${\Bbb C}$-algebra homomorphism
   from $A_n({\Bbb C})$ is injective).
}\end{ssexample}

\bigskip

\begin{flushleft}
{\bf $\alpha$-twisted ${\cal O}_X$-coherent ${\cal D}_X$-modules and
     enlargements of ${\cal O}_X^{A\!z}$ by ${\cal D}_X$.}
\end{flushleft}
Let
 $\alpha\in \check{C}_{\et}(X,{\cal O}_X^{\ast})$  and
 $\,{\cal F}\; =\; ( \{U_i\}_{i\in I},\,
                   \{{\cal F}_i\}_{i\in I},\,
                   \{\phi_{ij}\}_{i,j\in I} )\,$
  be an $\alpha$-twisted ${\cal O}_X$-module.

\smallskip
\begin{ssdefinition}
{\bf [connection on ${\cal F}$].} {\rm
 A {\it connection} $\nabla$ on ${\cal F}$
  is a set $\{\nabla_i\}_{i\in I}$
  where
   $\nabla_i: {\cal F}_i \rightarrow
      {\cal F}_i\otimes_{{\cal O}_{U_i}}\Omega_{U_i}$
   is a connection on ${\cal F}_i$,
 that satisfies
  $\phi_{ij}\circ (\nabla_i|_{U_{ij}})
    = (\nabla_j|_{U_{ij}})\circ \phi_{ij}$.
 $\nabla$ is said to be {\it flat}
  if $\nabla_i$ is flat for all $i\in I$.
}\end{ssdefinition}

\smallskip

\noindent
Note that the existence of an $\alpha$-twisted ${\cal O}_X$-module
 with a connection imposes
 a condition on $\alpha$ that $\alpha$ has a presentation
 $(\alpha_{ijk})_{ijk}$ with
 $d\alpha := (d\alpha_{ijk})_{ijk}=(0)_{ijk}$;
    i.e.\ $\alpha_{ijk}\in{\Bbb C}^{\ast}$ for all $i,j,k$.

As the proof of Proposition~2.3.2.2 is local,
 it generalizes to
 $\alpha$-twisted ${\cal O}_X$-coherent ${\cal D}_X$-modules$\,$:

\smallskip

\begin{ssproposition}
{\bf [$\alpha$-twisted ${\cal O}$-coherent ${\cal D}$-module].}
 Let ${\cal M}$ be a ${\cal D}_X$-module
  that is $\alpha$-twisted ${\cal O}_X$-coherent.
 Then, ${\cal M}$ is an $\alpha$-twisted ${\cal O}_X$-locally-free.
 Furthermore, in this case,
 the action of ${\cal D}_X$ on ${\cal M}$ defines a flat connection
  $\nabla:{\cal M}\rightarrow {\cal M}\otimes\Omega_X$ on ${\cal M}$
  by assigning $\nabla_{\!\xi}\,s=\xi\cdot s$
   for $s\in{\cal M}$ and $\xi\in \Theta_X$;
 the converse also holds.
 This gives an equivalence of categories:
 $$
  \left\{\rule{0em}{1.2em}
   \begin{array}{c}
    \mbox{$\alpha$-twisted ${\cal O}_X$-coherent ${\cal D}_X$-modules}
   \end{array}
  \right\}\;
  \longleftrightarrow\;
  \left\{
   \begin{array}{l}
     \mbox{$\alpha$-twisted coherent locally free ${\cal O}_X$-}\\
     \mbox{modules with a flat connection}
   \end{array}
  \right\}\,.
 $$
\end{ssproposition}

\smallskip

Let ${\cal E}$ be an $\alpha$-twisted ${\cal O}_X$-coherent
  ${\cal D}_X$-module.
Then the ${\cal D}_X$-module structure on ${\cal E}$ induces
  a natural ${\cal D}_X$-module structure on the (ordinary)
  ${\cal O}_X$-module
  ${\cal O}_X^{A\!z} := \Endsheaf_{{\cal O}_X}({\cal E})$.
We will denote both the connection on ${\cal E}$ and
 on ${\cal O}_X^{A\!z}$ by $\nabla$.
As
 both ${\cal O}_X^{A\!z}:=\Endsheaf_{{\cal O}_X}({\cal E})$ and
  ${\cal D}_X$ act now on ${\cal E}$  and
 ${\cal D}_X$ acts also on ${\cal O}_X^{A\!z}$,
one can define a sheaf ${\cal O}_X^{A\!z,{\cal D}}$ of unital
 associative algebras generated by ${\cal O}_X^{A\!z}$ and ${\cal D}_X$
 as follows:
 \begin{itemize}
  \item[$\cdot$]
   Over a (Zariski) open subset $U$ of $X$,
    ${\cal O}_X^{A\!z,{\cal D}}(U)$ is the unital associative
    ${\Bbb C}$-algebra generated by
    ${\cal O}^{A\!z}_X(U)\cup {\cal D}_X(U)$
   subject to the following rules$\,$:
    \begin{itemize}
     \item[(1)]
      for $\phi_1,\,\phi_2\in {\cal O}_X^{A\!z}(U)$,
      $\phi_1\cdot\phi_2\in {\cal O}_X^{A\!z,{\cal D}}(U)$
       coincides with the existing
       $\phi_1\phi_2\in {\cal O}_X^{A\!z}(U)\,$;

     \item[(2)]
      for $\eta_1,\,\eta_2\in {\cal D}_X(U)$,
      $\eta_1\cdot\eta_2\in {\cal O}_X^{A\!z,{\cal D}}(U)$
       coincides with the existing
       $\eta_1\eta_2\in{\cal D}_X(U)\,$;

     \item[(3)] ({\it Leibniz rule})\hspace{1em}
      for $\phi\in {\cal O}_X^{A\!z}(U)$ and
           $\xi\in\Theta_X(U)\subset {\cal D}_X(U)$,
       $$
        \xi\cdot\phi\; =\; (\nabla_{\!\xi}\,\phi)\,+\, \phi\cdot \xi\,.
       $$
    \end{itemize}
 \end{itemize}
In notation,
 ${\cal O}_X^{A\!z,{\cal D}}
    := {\Bbb C}\langle {\cal O}_X^{A\!z},{\cal D}_X\rangle^{\nabla}$.

\smallskip

\begin{ssdefinition}
{\bf [Azumaya quantum scheme with fundamental module].} {\rm
The noncommutative space
 $$
  (X^{A\!z,{\cal D}},\,{\cal E}^{\nabla})\;
  :=\; (X,\,
       {\cal O}_X^{A\!z,{\cal D}}
       = {\Bbb C}\langle
          \Endsheaf_{{\cal O}_X}({\cal E}), {\cal D}_X\rangle^{\nabla},\,
       ({\cal E},\nabla) )
 $$
 will be called an {\it Azumaya quantum scheme with a fundamental module
 in the class $\alpha$}.
}\end{ssdefinition}

\smallskip

\noindent
Caution that ${\cal O}_X\subset {\cal O}_X^{A\!z,{\cal D}}$
 in general does not lie in the center of ${\cal O}_X^{A\!z,{\cal D}}$.

\smallskip

\begin{ssremark}
{$[\,$${\cal E}^{\nabla}$ as a module
      over $\Space({\cal O}_X^{A\!z,{\cal D}}) $$\,]$.}
{\rm
 The full notation for $X^{A\!z,{\cal D}}$ in Definition~5.1.7
   is meant to make two things manifest:
   \begin{itemize}
    \item[(1)]
     There is a built-in diagram of dominant morphisms of $X$-spaces$\,$:
      $$
       \xymatrix @R=1em @C=-1em {
        & X^{A\!z,{\cal D}} := \Space {\cal O}_X^{A\!z,{\cal D}}
          \ar[ld] \ar[rd] \ar[dd] & \\
        **[l]X^{A\!z} := \Space {\cal O}_X^{A\!z} \ar[rd]
         && **[r]Q\mbox{\boldmath $\Omega$}_X := \Space {\cal D}_X\,.
                 \ar[ld] \\
        & X  &
       }
      $$
     $\Space {\cal O}_X^{A\!z,{\cal D}}$ is the major space
        one should focus on.
     The other three spaces
       - $\Space {\cal O}_X^{A\!z}$, $\Space {\cal D}_X$, and $X$ -
      should be treated as auxiliary spaces that are built into
      the construction to encode a special treatment
      that takes care of the issue of
      localizations of noncommutative rings in the current situation;
      cf.\ the next item.

     \item[(2)]
      Despite the fact that ${\cal O}_X$ is in general not
             in the center of ${\cal O}_X^{A\!z,{\cal D}}$,
       there is a notion of localization and open sets on
       $\Space {\cal O}^{A\!z,{\cal D}}$ induced by those on $X$.
      I.e.\ $\Space{\cal O}_X^{A\!z,{\cal D}}$ has a built-in topology
       induced from the (Zariski) topology of $X$.
      Thus, one can still have the notion of
       {\it gluing systems of morphisms} and {\it sheaves}
       with respect to this topology.
   \end{itemize}
  In particular, ${\cal E}^{\nabla}$ is a sheaf of
   ${\cal O}_X^{A\!z,{\cal D}}$-modules
   supported on the whole $\Space{\cal O}^{A\!z,{\cal D}}$
   with this topology.
}\end{ssremark}

\smallskip

\begin{ssremark}
{$[\,$Azumaya algebra over ${\cal D}_X$$\,]$.}
{\rm
 Note that ${\cal O}_X^{A\!z,{\cal D}}$ can also be thought of
  as an {\it Azumaya algebra over ${\cal D}_X$}
  in the sense that it is a sheaf of algebras on $X$,
  locally modelled on the matrix ring
   $M_r({\cal D}_U)$ over ${\cal D}_U$
  for $U$ an affine \'{e}tale-open subset of $X$.
}\end{ssremark}

\smallskip

\begin{ssremark}
{\it $[\,$partially deformation-quantized target$\,]$.} {\rm
 {From} the fact that Weyl algebras are simple,
  it is anticipated that a morphism to
  a totally deformation-quantized space $Y=\mbox{\boldmath $\Omega$}_W$
  is a dominant morphism.
 In general, one may take $Y$ to be a partial deformation quantization
  of a space along a foliation.
 E.g.\  a deformation quantization of {\boldmath $\Omega$}$_{W/B}$
  along the fibers of a fibration $W/B$.
 For compact $Y$, one may consider
  the deformation quantization along torus fibers
  of a space fibered by even-dimensional tori.\footnote{Though
                                   we do not touch this here,
                                   readers should be aware that
                                    this is discussed in numerous
                                    literatures.}
 (Cf.~Example~2.3.2.11.)
}\end{ssremark}

\bigskip

\begin{flushleft}
{\bf Higgsing and un-Higgsing of quantum D-branes
     via deformations of morphisms.}
\end{flushleft}
We give here an example of morphisms from $X$ with the new structure
 to a target-space $Y$ being the total space {\boldmath $\Omega$}$_W$
 of the cotangent bundle $\Omega_W$ of a smooth variety $W$.
It illustrates also the Higgsing/un-Higgsing behavior of D-branes
 in the current deformation-quantized situation.

\begin{ssexample}
{\bf [Higgsing/un-Higgsing of D-brane].} {\rm
 Let
  $(X^{A\!z,{\cal D}},{\cal E}^{\nabla})$
   be the affine Azumaya quantum scheme with a fundamental module
   associated to
    the ring
     $R:= {\Bbb C}\langle M_2({\Bbb C}[z]), \partial_z\rangle$
     (with the implicit relation $[\partial_z,z]=1$ and
           the identification of ${\Bbb C}[z]$
                   with the center of $M_2({\Bbb C}[z])$)  with
    the $R$-module $N:={\Bbb C}[z]\oplus {\Bbb C}[z]$,
     on which $M_2({\Bbb C}[z])$ acts by multiplication and
      $\partial_z$ acts by formal differentiation,  and
  $Y$ be the partially deformation-quantized space
   $Q_{\lambda}\mbox{\boldmath $\Omega$}_{{\Bbb A}^2/{\Bbb A}^1}$
   associated to the ring
   $S_{\lambda} :=
    {\Bbb C}\langle u,v,w \rangle/([v,w], [u,v], [u,w]-\lambda)$,
    where $\lambda\in{\Bbb C}$.
 Note\footnote{Also, we take the convention that
                $\partial_z\cdot m$ means the product in
                 ${\Bbb C}\langle M_2({\Bbb C}[z]), \partial_z\rangle$
                 and
                $\partial_z m$ means entry-wise formal differentiation
                 of $m$, for $m\in M_2({\Bbb C}[z])$.}
  that the action of $\partial_z$ on $N$
  induces an action of $\partial_z$ on $M_2({\Bbb C}[z])$
  by the entry-wise formal differentiation and
 the ${\Bbb A}^2/{\Bbb A}^1$ corresponds to
  ${\Bbb C}[v]\hookrightarrow{\Bbb C}[v,w]$.
 Consider the following special class of morphisms:
  $$
   \hspace{10em}
   \begin{array}{ccc}
    X & \xymatrix{\ar[rrr]^-{\varphi_{(A,B)}} &&&}          & Y \\[.6ex]
    R & \xymatrix{&&& \ar[lll]_-{\varphi_{(A,B)}^{\sharp}}}
      & S_{\lambda}\\[.6ex]
    \lambda\partial_z + A
      & \xymatrix{&&& \ar @{|->}[lll]} & u \\[.6ex]
    B & \xymatrix{&&& \ar @{|->}[lll]} & v \\[.6ex]
    z & \xymatrix{&&& \ar @{|->}[lll]} & w \\[.6ex]
   \end{array}\,,\hspace{2em}
   \mbox{$A,\, B\;\in\; M_2({\Bbb C}[z])\,$,}
  $$
  subject to
  $[\lambda\partial_z+A, B]\,=\,0\,$.
  (The other two constraints,
   $[B,z]\,=\,0\,$ and $\,[\lambda\partial_z+A, z]-\lambda\,=\,0\,$,
  are automatic.)
 Let
  $$
   A\; =\;
    \left[\begin{array}{cc} a_1 & a_2 \\ a_3 & a_4 \end{array}\right]
   \hspace{2em}\mbox{and}\hspace{2em}
   B\; =\;
    \left[\begin{array}{cc} b_1 & b_2 \\ b_3 & b_4 \end{array}\right]\,,
  $$
  where $a_i$, $b_j\in {\Bbb C}[z]$ and assume that $\lambda\ne 0$.
 Then, the associated system
   $\lambda\partial_zB+[A,B]=0$
   of homogeneous linear ordinary differential equations on $B$
   has a solution
  if and only if $A$ satisfies
   $$
    (a_1-a_4)^2 + 4 a_2a_3\;=\;0\,.
   $$
  Under this condition on $A$, the system has four fundamental solutions:
   $$
    \begin{array}{lcl}
    B_1  & =
     & \left[ \begin{array}{llll}
             1+\lambda^{-2}a_2a_3z^2
              &&& \lambda^{-1}a_2z
                  -\frac{1}{2}\lambda^{-2}(a_1-a_4)a_2 z^2
                  \hspace{2.3ex}  \\[.6ex]
             -\lambda^{-1}a_3z -\frac{1}{2}\lambda^{-2}(a_1-a_4)a_3 z^2
              &&& -\lambda^{-2}a_2a_3 z^2
            \end{array}
       \right]\,, \\[4ex]
    B_2  & =
     & \left[ \begin{array}{lll}
             \lambda^{-1}a_3z - \frac{1}{2}\lambda^{-2}(a_1-a_4)a_3 z^2
              && \hspace{1.6em}
                 1 - \lambda^{-1}(a_1-a_4)z -\lambda^{-2}a_2a_3 z^2\\[.6ex]
             - \lambda^{-2}a_3^2 z^2
               && \hspace{1.6em}
                  -\lambda^{-1}a_3z +\frac{1}{2}\lambda^{-2}(a_1-a_4)a_3z^2
            \end{array}
       \right]\,, \\[4ex]
    B_3  & =
     & \left[ \begin{array}{lll}
             -\lambda^{-1}a_2z -\frac{1}{2}\lambda^{-2}(a_1-a_4)a_2 z^2
              && \hspace{1.3ex}
                 - \lambda^{-2}a_2^2 z^2 \\[.6ex]
             1 + \lambda^{-1}(a_1-a_4)z - \lambda^{-2}a_2a_3 z^2
              && \hspace{1.3ex}
                 \lambda^{-1}a_2z + \frac{1}{2}\lambda^{-2}(a_1-a_4)a_2z^2
                 \hspace{1em}
            \end{array}
       \right]\,, \\[4ex]
    B_4  & =
     & \left[ \begin{array}{llll}
             -\lambda^{-2}a_2a_3 z^2
              &&& \hspace{1.3ex}
                  - \lambda^{-1}a_2z
                  + \frac{1}{2}\lambda^{-2}(a_1-a_4)a_2 z^2
                  \hspace{.9ex} \\[.6ex]
             \lambda^{-1}a_3z + \frac{1}{2}\lambda^{-2}(a_1-a_4)a_3 z^2
              &&& \hspace{1.3ex}
                  1 + \lambda^{-2}a_2a_3 z^2
            \end{array}
       \right]\,.
    \end{array}
   $$
  Denote this solution space by ${\Bbb C}^4_A$ with coordinates
   $(\hat{b}_1,\,\hat{b}_2,\,\hat{b}_3,\,\hat{b}_4)$
   and the correspondence
   $$
    (\hat{b}_1,\,\hat{b}_2,\,\hat{b}_3,\,\hat{b}_4)\;\;\;
     \longleftrightarrow\;\;\;
     \hat{b}_1 B_1\,+\,\hat{b}_2B_2\,+\,\hat{b}_3B_3\,+\,\hat{b}_4B_4\;
     =:\; B_{(\hat{b}_1,\hat{b}_2,\hat{b}_3,\hat{b}_4)}\,.
   $$
  Then,
   \begin{itemize}
    \item[$\cdot$] {\it
     the degree-$0$ term
     $B_{(0)}$ of
     $B=B_{(\hat{b}_1,\hat{b}_2,\hat{b}_3,\hat{b}_4)}$ (in $z$-powers)
     is given by
     $\left[\begin{array}{ll}
             \hat{b}_1 & \hat{b}_2 \\[.6ex] \hat{b}_3 & \hat{b}_4
            \end{array} \right]\,$,}

    \item[$\cdot$] {\it
     the characteristic polynomial of $B$ is identical to
     that of $B_{(0)}$.}
   \end{itemize}
 It follows that the image $\Image\varphi_{(A,B)}$ of $\varphi_{(A,B)}$
  is a (complex-)codimension-$1$ sub-quantum scheme in $Y$
  whose associated ideal in $S_{\lambda}$ contains the ideal
   $$
    \left( v^2-\trace B_{(0)}\,v + \determinant B_{(0)} \right)\,.
   $$

 Let $\mu_-$ and $\mu_+$ be the eigen-values of $B_{(0)}$.

 \bigskip

 \noindent
 {\it Case $(a):\,$ $\nu_-\ne\nu_+$.}\hspace{1ex}
  In this case,  the above ideal $((v-\nu_-)(v-\nu_+))$
   coincides with $\Ker\varphi_{(A,B)}^{\sharp}$ and, hence,
   describes precisely $\Image\varphi_{(A,B)}\subset Y$.
  Since $\varphi_{(A,B)}^{\sharp}(v)=B$,
   let $N_-:= \Ker(B-\nu_-)\subset N$.
  This is a rank-$1$ ${\Bbb C}[z]$-submodule of
    ${\Bbb C}[z]\oplus {\Bbb C}[z]$
    that is invariant also under
     $\varphi_{(A,B)}^{\sharp}(S_{\lambda})$.
  This gives $N_-$ a $S_{\lambda}/(v-\nu_-)$-module structure
   that has rank-$1$ as ${\Bbb C}[w]$-module.
  Similarly,
   $N_+:= \Ker(B-\nu_+)\subset N$ is invariant under
     $\varphi_{(A,B)}^{\sharp}(S_{\lambda})$
    and has a $\varphi_{(A,B)}^{\sharp}$-induced
     $S_{\lambda}/(v-\nu_+)$-module structure that is of rank-$1$
      as ${\Bbb C}[w]$-module.
  Let
   $$
    \begin{array}{crl}
     Z & :=
       & \Image\varphi_{(A,B)}\;\;
         =\;\; \Space (S_{\lambda}/((v-\nu_-)(v-\nu_+)))\\[.6ex]
       & =
       & \Space (S_{\lambda}/(v-\nu_-))
            \cup \Space (S_{\lambda}/(v-\nu_+))\;\;
         =:\;\; Z_-\cup Z_+
    \end{array}
   $$
   be the two connected components of the quantum subscheme
   $\Image\varphi_{(A,B)}\subset Y$  and
  denote the ${\cal O}_{Z_-}$-modules associated to $N_-$ and $N_+$
   by $(_{S_{\lambda}}N_-)^{\sim}$ and $(_{S_{\lambda}}N_+)^{\sim}$
   respectively.
  Then
   $$
    \varphi_{(A,B),\ast}{\cal E}\;
     =\; (_{S_{\lambda}}N_-)^{\sim}\oplus (_{S_{\lambda}}N_+)^{\sim}
    \hspace{1em}\mbox{with
     $(_{S_{\lambda}}N_-)^{\sim}$ supported on $Z_-$ and
     $(_{S_{\lambda}}N_+)^{\sim}$ on $Z_+$}\,.
   $$

 \bigskip

 \noindent
 {\it Case $(b):\,$ $\nu_-=\nu_+=\nu$.}\hspace{1ex}
 In this case, $\Ker\varphi_{(A,B)}^{\sharp}$ can be either
  $(v-\nu)$ or $((v-\nu)^2)$ and both situations happen.
 \begin{itemize}
  \item[$\cdot$]
   When $\Ker\varphi_{(A,B)}^{\sharp} = (v-\nu)$,
   $N={\Bbb C}[z]\oplus{\Bbb C}[z]$
    has a $\varphi_{(A,B)}^{\sharp}$-induced
    $S_{\lambda}/(v-\nu)$-module structure  and
   $\varphi_{(A,B),\ast}{\cal E}$ has support
    $\Image\varphi_{(A,B)}=\Space(S_{\lambda}/(v-\nu))\subset Y$.

  \item[$\cdot$]
   When $\Ker\varphi_{(A,B)}^{\sharp} = ((v-\nu)^2)$,
   $N={\Bbb C}[z]\oplus{\Bbb C}[z]$
    has a $\varphi_{(A,B)}^{\sharp}$-induced
    $S_{\lambda}/((v-\nu)^2)$-module structure  and
   $\varphi_{(A,B),\ast}{\cal E}$ has support
    $Z:=\Image\varphi_{(A,B)}=\Space(S_{\lambda}/((v-\nu)^2))\subset Y$.
   It contains an ${\cal O}_Z$-submodule
    $(_{S_{\lambda}}N_0)^{\sim}$, associated to
    $N_0:=\Ker(v-\nu)\subset N$, that is supported on
    $Z_0:=\Space(S_{\lambda}/(v-\nu))\subset Z$.
   In other words, in the current situation, $\varphi_{(A,B),\ast}{\cal E}$
    not only is of rank-$2$ as a ${\Bbb C}[w]$-module
    but also has a built-in $\varphi_{(A,B)}$-induced filtration
    $(_{S_{\lambda}}N_0)^{\sim}\subset \varphi_{(A,B),\ast}{\cal E}$.
 \end{itemize}

 \bigskip

 \noindent
 Thus, by varying $(A,B)$ in the solution space of
   $\lambda\partial_zB+[A,B]=0$
  so that the eigen-values of $B_{(0)}$ change from being distinct
   to being identical and vice versa,
 one realizes the Higgsing and un-Higgsing phenomena of D-branes
   in superstring theory for the current situation
  as deformations of morphisms from Azumaya quantum schemes to
   the open-string quantum target-space $Y$:

 \bigskip
 \bigskip

 \hspace{1ex}
 \xymatrix{
  \framebox[17.6em][c]{\parbox{16.6em}{\it
   deformations of morphisms $\varphi$\\
   from Azumaya deformation-quantized\\ schemes
   with a fundamental module\\
   to a deformation-quantized target $Y$}}
   \ar @2{->}[rr]
   && \framebox[13.6em][c]{\parbox{11.6em}{\it
       Higgsing and un-Higgsing\\
       of Chan-Paton modules\\
       on (image) D-branes on $Y$}}
 } 

 \bigskip
 \bigskip

 This concludes the example. See also {\sc Figure}~2-1-1.

}\end{ssexample}

\bigskip

\subsection{Tests of the Polchinski-Grothendieck Ansatz for D-branes.}

If the Polchinski-Grothendieck Ansatz is truly fundamental for D-branes
 and the notion of morphisms formulated above does capture D-branes,
then we should be able to see what string-theorists see
 in quantum-field-theory language
 solely by our formulation.
In this subsection, we collect six basic tests in this regard on
 string-theory works, 1995--2008, from our first group of examples.
This group is guided by the following question:
 \begin{itemize}
  \item[$\cdot$] {\bf Q.}\
  \parbox[t]{13cm}{\it {\bf [QFT vs.\ maps]}$\;$
   Can we reconstruct the geometric object that arises
    in a quantum-field-theoretical study of D-branes
    through morphisms from Azumaya noncommutative spaces?}
 \end{itemize}
This subsection is not to be read alone.
Rather, we recommend readers to go through
 the quoted string-theory works on each theme first
 and then compare.

\bigskip

\noindent
{\bf (1)}
\parbox[t]{15cm}{{\bf
 Bershadsky-Sadov-Vafa:
 Classical and quantum moduli space of D$0$-branes.}\\
 ({\it Bershadsky-Sadov-Vafa vs.\ Polchinski-Grothendieck}$\,$;\\[.2ex]
 [B-V-S1], [B-V-S2], and [Va1], 1995.)}
 \begin{itemize}
  \item[]
  The moduli stack ${\frak M}^{\,0^{Az^f}}_{\bullet}\!\!(Y)$
   of morphisms from Azumaya point with a fundamental module
   to a {\it smooth variety $Y$ of complex dimension $2$}
   contains various substacks with different coarse moduli space.
  One choice of such gives rise to the symmetric product
   $S^{\bullet}(Y)$ of $Y$
  while another choice gives rise to the Hilbert scheme
   $Y^{[\bullet]}$ of points on $Y$.
  The former play the role of the classical moduli and
   the latter quantum moduli space of D0-branes studied
  in [Va1] and in [B-V-S1], [B-V-S2].

  \vspace{-1.6ex}
  \item[] $\hspace{1.2em}$
  See [L-Y3: Sec.~4.4] (D(1)),
       theme: `A comparison with the moduli problem of gas of D0-branes
       in [Vafa1] of Vafa' for more discussions.
 \end{itemize}

\bigskip

\noindent
{\bf (2)}
\parbox[t]{15cm}{{\bf
 Douglas-Moore and Johnson-Myers:\\
 D-brane probe to an ADE surface singularity.}\\
({\it Douglas-Moore/Johnson-Myers
      vs.\ Polchisnki-Grothendiecek}$\,$;\\[.2ex]
 [Do-M], 1996, and [J-M], 1996.)}
 \begin{itemize}
  \item[]
   Here, we are compared with the setting of Douglas-Moore [Do-M].
   The notion of `morphisms from an Azumaya scheme
    with a fundamental module'
    can be formulated as well when the target $Y$ is a stack.
   In the current case, $Y$ is the {\it orbifold associated to an
    ADE surface singularity}. It is a {\it smooth Deligne-Mumford stack}.
   Again, the stack ${\frak M}^{\,0^{Az^f}}_{\bullet}\!\!(Y)$
    of morphisms from Azumaya points with a fundamental
    module to the orbifold $Y$ contains various substacks
    with different coarse moduli space.
   An appropriate choice of such gives rise to the resolution of
    ADE surface singularity.

  \vspace{-1.6ex}
  \item[] $\hspace{1.2em}$
  See [L-Y4] (D(3)) for a brief highlight of [Do-M],
   details of the Azumaya geometry involved, and more references.
 \end{itemize}

\bigskip

\noindent
{\bf (3)}
\parbox[t]{15cm}{{\bf
 Klebanov-Strassler-Witten: D-brane probe to a conifold.}\\
({\it Klebanov-Strassler-Witten vs.\ Polchinski-Grothendieck}$\,$;
 [Kl-W], 1998, and [Kl-S], 2000.)}
 \begin{itemize}
  \item[]
   Here, the problem is related to the moduli stack
    ${\frak M}^{\,0^{Az^f}}_{\bullet}\!\!(Y)$
    of morphisms from Azumaya points with a fundamental module
    to a local conifold $Y$, a singularity Calabi-Yau $3$-fold,
    whose complex structure is given by
    $Y=\Spec({\Bbb C}[z_1,z_2,z_3,z_4]/(z_1z_2-z_3z_4))$.
   Again, different resolutions of the conifold singularity of $Y$
    can be obtained by choices of substacks from
    ${\frak M}^{\,0^{Az^f}}_{\bullet}\!\!(Y)$,
   as in Tests (1) and (2).
   Such a resolution corresponds to a low-energy effective geometry
    ``observed" by a stacked D-brane probe to $Y$
    when there are no fractional/trapped brane sitting
    at the singularity {\boldmath $0$} of $Y$.

  \vspace{-1.6ex}
  \item[] $\hspace{1.2em}$
   New phenomenon arises when there are fractional/trapped  D-branes
   sitting at {\boldmath $0$}. Instead of resolutions of the conifold
   singularity of $Y$, a low-energy effective geometry ``observed"
   by a D-brane probe is a complex deformation of $Y$ with topology
   $T^{\ast}S^3$ (the cotangent bundle of $3$-sphere).
   From the Azumaya geometry point of view, two things happen:
   \begin{itemize}
    \item[$\cdot$]
     Taking {\it both} the (stacked-or-not) D-brane probe
       and the trapped brane(s) into account,
     the Azumaya geometry on the D-brane world-volume remains.

    \item[$\cdot$]
     A noncommutative-geometric enhancement of $Y$ occurs
      via morphisms
      $$
      \xymatrix{
       & \Xi=\Space R_{\Xi} \ar[d]^{\pi^{\Xi}} \\
       Y \ar @{^{(}->}[r] & {\Bbb A}^4  &.
      }
      $$
   \end{itemize}
   Here,
    ${\Bbb A}^4=\Spec({\Bbb C}[z_1,z_2,z_3,z_4])$,
    $$
     R_{\Xi}\;=\;
      \frac{ {\Bbb C}\langle\,\xi_1,\xi_2,\xi_3,\xi_4\, \rangle }
      { ([\xi_1\xi_3, \xi_2\xi_4]\,,\, [\xi_1\xi_3, \xi_1\xi_4]\,,\,
         [\xi_1\xi_3, \xi_2\xi_3]\,,\, [\xi_2\xi_4, \xi_1\xi_4]\,,\,
         [\xi_2\xi_4, \xi_2\xi_3]\,,\, [\xi_1\xi_4, \xi_2\xi_3]) }
    $$
     with
      ${\Bbb C}\langle\,\xi_1,\xi_2,\xi_3,\xi_4\, \rangle$
       being the associative (unital) ${\Bbb C}$-algebra
       generated by $\xi_1,\,\xi_2,\,\xi_3,\,\xi_4$ and
      $[\bullet\,,\bullet^{\prime}\,]$ being the commutator,
     $Y\hookrightarrow {\Bbb A}^4$ via the definition of $Y$ above,
       and
     $\pi^{\Xi}$ is specified by the ${\Bbb C}$-algebra homomorphism
      $$
      \begin{array}{cccccl}
       \pi^{\Xi,\sharp}  & :
        & {\Bbb C}[z_1,z_2,z_3,z_4]  & \longrightarrow  &  R_{\Xi}\\
       && z_1 & \longmapsto  & \xi_1\xi_3 \\
       && z_2 & \longmapsto  & \xi_2\xi_4 \\
       && z_3 & \longmapsto  & \xi_1\xi_4 \\
       && z_4 & \longmapsto  & \xi_2\xi_3  & .\\
      \end{array}
      $$
   One is thus promoted to studying the stack
    ${\frak M}^{\,0^{Az^f}}_{\bullet}\!\!(\Space R_{\Xi})$,
    of morphisms from Azumaya points with a fundamental module to
    $\Space R_{\Xi}$, following Sec.~2.1.

  \vspace{-1.6ex}
  \item[] $\hspace{1.2em}$
   To proceed, we need the following notion:

  \item[]
  {\bf Definition 2.4.3.1.
  [superficially infinitesimal deformation].} {\rm
   Given associative (unital) rings,
    $R=\langle\,r_1,\,\ldots\,,r_m\,\rangle/\!\!\sim$ and $S$,
   that are finitely-presentable and
    a ring-homomorphism $h:R\rightarrow S$.
   A {\it superficially infinitesimal deformation} of $h$
     {\it with respect to the generators}
     $\{r_1,\,\ldots\,,r_m\}$ {\it of} $R$
    is a ring-homomorphism $h_{\varepsilon}:R\rightarrow S$
     such that
      $h_{\varepsilon}(r_i)=h(r_i)+\varepsilon_i$
       with $\varepsilon_i^2=0$,
      for $i=1,\,\ldots\,,m$.
    } 

  \item[]
   When $S$ is commutative,
    a superficially infinitesimal deformation of
     $\, h_{\varepsilon}:R\rightarrow S\,$
     is an infinitesimal deformation of $h$
     in the sense that $h_{\varepsilon}(r)=h(r)+\varepsilon_r$
      with $(\varepsilon_r)^2=0$, for all $r\in R$.
   This is no longer true for general noncommutative $S$.
   The $S$ plays the role of the Azumaya algebra
    $M_{\bullet}({\Bbb C})$ in our current test.
   It turns out that
   a morphism $\varphi:\pt^{A\!z}\rightarrow \Space R_{\Xi}$
    that projects by $\pi^{\Xi}$ to the conifold singularity
    {\boldmath $0$}$\in Y$ can have
     superficially infinitesimal deformations $\varphi^{\prime}$
    such that the image $(\pi^{\Xi}\circ \varphi^{\prime})(\pt^{A\!z})$
     contains not only {\boldmath $0$} but also points in
     ${\Bbb A}^4-Y$.
   Indeed there are abundant such superficially infinitesimal
    deformations.
   Thus, beginning with a substack ${\cal Y}$ of
    ${\frak M}^{\,0^{Az^f}}_{\bullet}\!\!(\Space R_{\Xi})$,
    that projects onto $Y$
    via $\varphi\mapsto \Image(\pi^{\Xi}\circ\varphi)$,
    one could use a $1$-parameter family
    of superficially infinitesimal deformations of $\varphi\in {\cal Y}$
    to drive ${\cal Y}$ to a new substack ${\cal Y}^{\prime}$
    that projects to {\boldmath $0$}$\cup Y^{\prime}\subset {\Bbb A}^4$,
    where $Y^{\prime}$ is smooth (i.e.\ a deformed conifold).
   It is in this way that a deformed conifold $Y^{\prime}$
    is detected by the D-brane probe via the Azumaya structure
    on the common world-volume of the probe and the trapped brane(s).
   Cf.~{\sc Figure}~2-4-1.
   \begin{figure}[htbp]
    \epsfig{figure=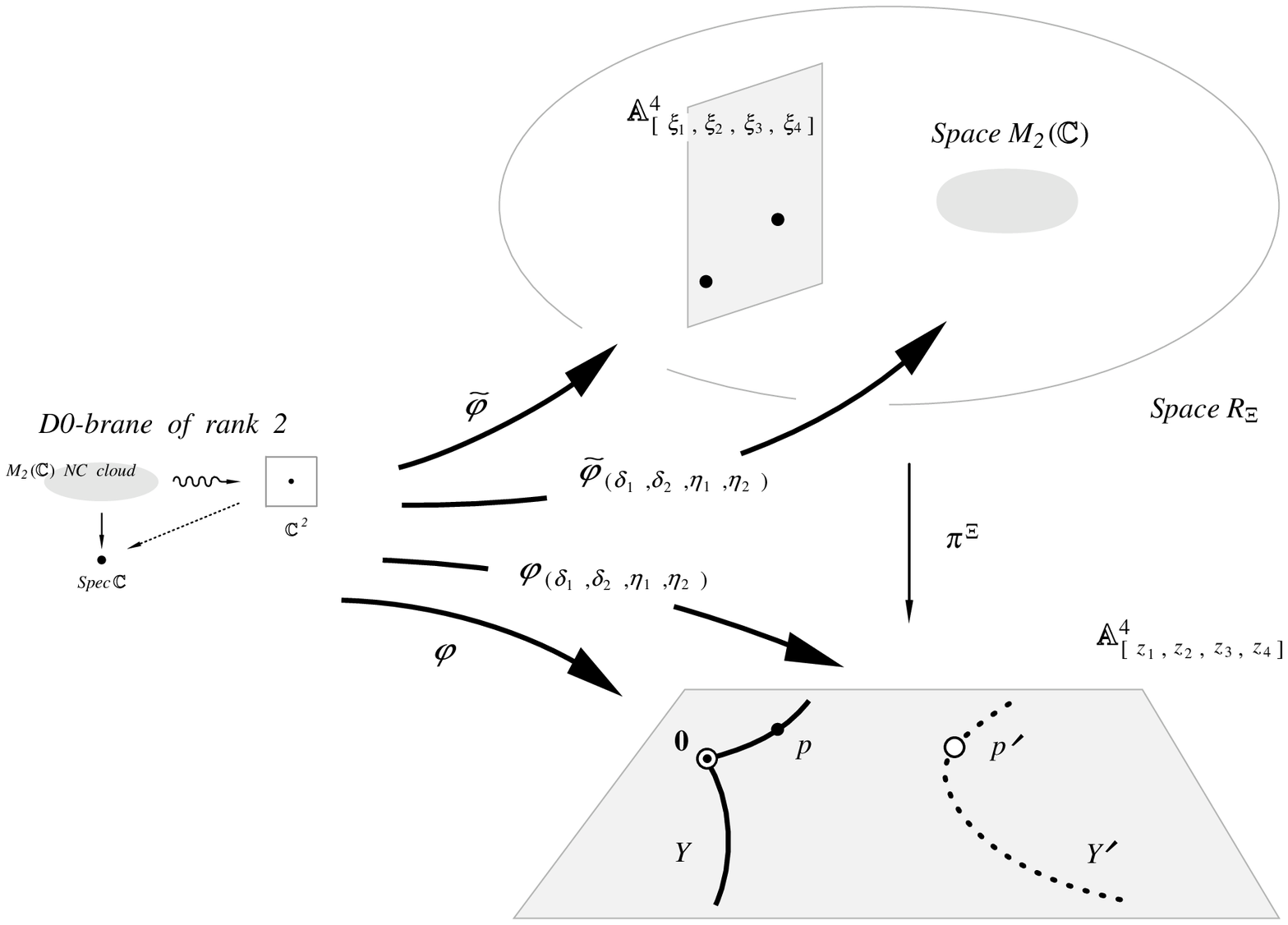,width=16cm}
    \centerline{\parbox{13cm}{\small\baselineskip 12pt
     {\sc Figure}~2-4-1.
     A generic superficially infinitesimal deformation
      $\widetilde{\varphi}_{(\delta_1,\delta_2,\eta_1,\eta_2)}$
      of $\widetilde{\varphi}$ has
      a noncommutative image $\simeq$ {\it Space}$\,M_2({\Bbb C})$.
     It then descends to
      ${\Bbb A}^4_{[z_1,z_2,z_3,z_4]}
       :=\smallSpec({\Bbb C}[z_1,z_2,z_3,z_4])$
       and
      becomes a pair of ${\Bbb C}$-points
       on ${\Bbb A}^4_{[z_1,z_2,z_3,z_4]}$.
     One of the points is the conifold singularity
      ${\mathbf 0}=V(z_1,z_2,z_3,z_4)\in Y$  and
     the other is the point
      $p^{\prime}
       =V(\,z_1-a_1b_1-\delta_1\eta_1\,,\, z_2-a_2b_2-\delta_2\eta_2\,,\,
            z_3-a_1b_2-\delta_1\eta_2\,,\, z_4-a_2b_1-\delta_2\eta_1\,)\,$
      {\it off} $\,Y$ (generically).
    Through such deformations, any ${\Bbb C}$-point on
     ${\Bbb A}^4_{[z_1,z_2,z_3,z_4]}$ can be reached.
    Thus, one can realizes a deformation $Y^{\prime}$ of $Y$
      in ${\Bbb A}^4_{[z_1,z_2,z_3,z_4]}$
     by a subvariety in {\it Rep}$\,(R_{\Xi},M_2({\Bbb C}))$.
    This is the Azumaya-geometry origin of
     the phenomenon in Klebanov-Strassler [Kl-S] that
      a trapped D-brane sitting on the conifold singularity
      may give rise to a deformation of the moduli space of
      SQFT on the D$3$-brane probe, turning a conifold to
      a deformed conifold.
    Our D$0$-brane here corresponds to the internal part of
     the effective-space-time-filling D$3$-brane world-volume
     of [Kl-S].
    In this figure,
     ${\Bbb A}^4_{[\xi_1,\xi_2,\xi_3,\xi_4]}
       :=\smallSpec({\Bbb C}[\xi_1,\xi_2,\xi_3,\xi_4])
        =\smallSpec(R_{\Xi}/[R_{\Xi},R_{\Xi}])$
      is the maximal commutative subspace of $\smallSpace R_{\Xi}$
       and
     $(\delta_1,\delta_2,\eta_1,\eta_2)$
       parameterizes the superficially infinitesimal deformations
       of $\varphi$ in the current situation.
    }}
   \end{figure}

  \vspace{-1.6ex}
  \item[] $\hspace{1.2em}$
  See [L-Y5] (D(4)) for a brief highlight of [Kl-W] and [Kl-S],
   details of the Azumaya geometry involved, and more references.
 \end{itemize}

\bigskip

\noindent
{\bf (4)}
\parbox[t]{15cm}{{\bf
 G\'{o}mez-Sharpe:
 Information-preserving geometry, schemes, and D-branes.}\\
({\it G\'{o}mez-Sharpe vs.\ Polchisnki-Grothendieck}$\,$;
 [G-Sh], 2000.)}
 \begin{itemize}
  \item[]
   Among the various groups who studied the foundation of D-branes,
    this is a work that is very close to us in spirit.
   There, G\'{o}mez and Sharpe began with the quest:
    [G-Sh: Sec.~1]
   \begin{quote} ``{\it
    As is well-known, on $N$ coincident D-branes,
     $U(1)$ gauge symmetries are enhanced to $U(N)$ gauge symmetries,
    and scalars that formerly described normal motions of the branes
     become $U(N)$ adjoints.
    People have often asked what the deep reason for this behavior is
     -- what does this tell us about the geometry seen by D-branes?
     }",
   \end{quote}
   like us.
   They observed by comparing colliding D-branes with
    colliding torsion sheaves in algebraic geometry
    that it is very probable that
    \begin{quote}
     {\it coincident D-branes
     should carry some fuzzy structure} --
     {\it perhaps a nonreduced scheme structure}
    \end{quote}
    though the latter may carry more information
    than D-branes do physically.
   Further study on such nilpotent structure was done in [D-K-S];
   cf.~Sec.~4.2: theme
    `The generically filtered structure on the Chan-Patan bundle
     over a special Lagrangian cycle on
     a Calabi-Yau torus'
    of the current review.

  \vspace{-1.6ex}
  \item[] $\hspace{1.2em}$
  From our perspective,
   \begin{quote}
    {\it the (commutative) scheme/nilpotent structure
    G\'{o}mez and Sharpe proposed/ observed on a stacked D-brane
    is the manifestation/residual of the Azumaya (noncommutative)
    structure on an Azumaya space with a fundamental module
   when the latter forces itself into a commutative space/scheme
    via a morphism}.
   \end{quote}
  This connects our work to [G-Sh].
 \end{itemize}

\bigskip

\noindent
{\bf (5)}
\parbox[t]{15cm}{{\bf
 Sharpe: $B$-field, gerbes, and D-brane bundles.}\\
({\it Sharpe vs.\ Polchinski-Grothendieck}$\,$;
 [Sh2], 2001.)}
 \begin{itemize}
  \item[]
   Recall that a $B$-field on the target space(-time) $Y$
    specifies a gerbe ${\cal Y}_B$ over $Y$ associated to
    an $\alpha_B\in \check{C}^2_{\et}(Y,{\cal O}_Y^{\ast})$
    determined by the $B$-field.
  A morphism $\varphi:(X^{\!A\!z},{\cal E})\rightarrow (Y,\alpha_B)$
   from a general Azumaya scheme with a twisted fundamental module
   to $(Y,\alpha_B)$ can be lifted to a morphism
   $\breve{\varphi}:({\cal X}^{A\!z},{\cal F})\rightarrow {\cal Y}_B$
   from an Azumaya ${\cal O}_X^{\ast}$-gerbe with a fundamental module
   to the gerbe ${\cal Y}_B$.
  In this way, our setting is linked to Sharpe's picture of
   gerbes and D-brane bundles in a $B$-field background.

  \vspace{-1.6ex}
  \item[] $\hspace{1.2em}$
  See [L-Y6: Sec.~2.2] (D(5)) theme:
  `The description in term of morphisms from Azumaya gerbes
    with a fundamental module to a target gerbe'
  for details of the construction.
 \end{itemize}

\bigskip

\noindent
{\bf (6)}
\parbox[t]{15cm}{{\bf
 Dijkgraaf-Hollands-Su{\l}kowski-Vafa: Quantum spectral curves.}\\
({\it Dijkgraaf-Hollands-Su{\l}kowski-Vafa
      vs.\ Polchisnki-Grothendieck}$\,$;\\[.2ex]
 [D-H-S-V], 2007, and [D-H-S], 2008.)}
 \begin{itemize}
  \item[]
   Here we focus on a particular theme in these works: the notion of
    {\it quantum spectral curves from the viewpoint of D-branes}.
   Let
    $C$ be a smooth curve,
    ${\cal L}$ an invertible sheaf on $C$,
    ${\cal E}$ a coherent locally-free ${\cal O}_C$-module,  and
    {\boldmath ${\cal L}$}$=\boldSpec(\Sym^{\bullet\,}({\cal L}^{\vee}))$
     be the total space of ${\cal L}$.
   Here, ${\cal L}^{\vee}$ is the dual ${\cal O}_C$-module of ${\cal L}$.
   Then one has the following canonical one-to-one correspondence:
   $$
    \left\{
     \begin{array}{l}
      \mbox{${\cal O}_C$-module homomorphisms}\\
      \phi:{\cal E}\rightarrow {\cal E}\otimes{\cal L}
     \end{array}
    \right\}\;
     \longleftrightarrow\;
    \left\{
     \begin{array}{l}
       \mbox{morphisms
        $\varphi:(C^{A\!z},{\cal E})\rightarrow$
                                    {\boldmath ${\cal L}$}}\\
       \mbox{as spaces over $C$}
     \end{array}
    \right\}
   $$
   induced by the canonical isomorphisms
   $$
   \Hom_{{\cal O}_C}({\cal E},{\cal E}\otimes{\cal L})\;
     \simeq\; \Gamma({\cal E}^{\vee}\otimes{\cal E}\otimes{\cal L})\;
     \simeq\; \Hom_{{\cal O}_C}
               ({\cal L}^{\vee},\,\Endsheaf_{{\cal O}_C}({\cal E}))\,.
   $$
   Let $\Sigma_{({\cal E},\phi)}\subset\;${\boldmath ${\cal L}$}
    be the (classical) spectral curve
    associated to the Higgs/spectral pair $({\cal E},\phi)$;
   cf.~e.g.\ [B-N-S], [Hi1], and [Ox].
   Then, for $\varphi$ corresponding to $\phi$,
    $\Image\varphi\subset \Sigma_{({\cal E},\phi)}$.
   Furthermore, if $\Sigma_{({\cal E},\phi)}$ is smooth,
    then $\Image\varphi=\Sigma_{({\cal E},\phi)}$.
   This gives a {\it morphism-from-Azumaya-space interpretation of
    spectral curves}.

  \vspace{-1.6ex}
  \item[] $\hspace{1.2em}$
   To address the notion of `quantum spectral curve',
    let ${\cal L}$ be the sheaf $\Omega_C$ of differentials on $C$.
   Then the total space {\boldmath $\Omega$}$_C$ of $\Omega_C$ admits
    a canonical ${\Bbb A}^1$-family
    $Q_{{\Bbb A}^1}\mbox{\boldmath $\Omega$}_C$
    of deformation quantizations with the central fiber
    $Q_0\mbox{\boldmath $\Omega$}_C=\mbox{\boldmath $\Omega$}_C$.
   Let
    $({\cal E},\phi:{\cal E}\rightarrow {\cal E}\otimes\Omega_C)$
     be a spectral pair  and
    $\varphi: (C^{A\!z},{\cal E})\rightarrow \mbox{\boldmath $\Omega$}_C$
     be the corresponding morphism.
   Denote the fiber of $Q_{{\Bbb A}^1}\mbox{\boldmath $\Omega$}_C$
    over $\lambda\in {\Bbb A}^1$ by
    $Q_{\lambda}\mbox{\boldmath $\Omega$}_C$.
   Then, due to the fact that the Weyl algebras are simple algebras,
    the spectral curve $\Sigma_{({\cal E},\phi)}$
    in {\boldmath $\Omega$}$_C$ in general may not have a direct
    deformation quantization into $Q_{\lambda}\mbox{\boldmath $\Omega$}_C$
    by the ideal sheaf of $\Sigma_{({\cal E},\phi)}$
    in ${\cal O}_{\mbox{\scriptsize\boldmath $\Omega$}_C}$
   since this will only give
    ${\cal O}_{Q_{\lambda}\mbox{\scriptsize\boldmath $\Omega$}_C}$,
     which corresponds to the empty subspace of
     $Q_{\lambda}\mbox{\boldmath $\Omega$}_C$.
   However, one can still construct
    an ${\Bbb A}^1$-family
     $(Q_{{\Bbb A}^1}C^{A\!z}, Q_{{\Bbb A}^1}{\cal E})$
     of Azumaya quantum curves with a fundamental module
     out of $(C^{A\!z},{\cal E})$ and
    a morphism
     $\varphi_{{\Bbb A}^1}:
      (Q_{{\Bbb A}^1}C^{A\!z}, Q_{{\Bbb A}^1}{\cal E})
      \rightarrow Q_{{\Bbb A}^1}\mbox{\boldmath $\Omega$}_C$
     as spaces over ${\Bbb A}^1$,
    using the notion of `$\lambda$-connections'
     and `$\lambda$-connection deformations of $\phi$',
    such that
    \begin{itemize}
     \item[$\cdot$]
      $\varphi_0:= \varphi_{{\Bbb A}^1}|_{\lambda=0}$
       is the composition
       $\,(Q_0C^{A\!z},Q_0{\cal E})\,
           \longrightarrow\, (C^{A\!z},{\cal E})
            \stackrel{\varphi}{\longrightarrow}
             \mbox{\boldmath $\Omega$}_C\,$,
       where\\
         $(Q_0C^{A\!z},Q_0{\cal E})\rightarrow (C^{A\!z},{\cal E})$
         is a built-in dominant morphism from the construction;

     \item[$\cdot$]
      $\varphi_\lambda := \varphi_{{\Bbb A}^1}|_{\lambda}\,
        :\: (Q_{\lambda}C^{A\!z},Q_{\lambda}{\cal E})
            \longrightarrow Q_{\lambda}\mbox{\boldmath $\Omega$}_C$,
      for $\lambda\in {\Bbb A}^1-\{\mathbf 0\}\,$,
      is a morphism of Azumaya quantum curves with a fundamental module
      to the deformation-quantized noncommutative space
      $Q_{\lambda}\mbox{\boldmath $\Omega$}_C$.
    \end{itemize}
   In other words, we {\it
    replace the notion of `quantum spectral curves' by
    `quantum deformation $\varphi_{\lambda}$ of the morphism
    $\varphi$'}.
   In this way, both notions of classical and quantum spectral curves
    are covered in the notion of morphisms from Azumaya spaces.

  \vspace{-1.6ex}
  \item[] $\hspace{1.2em}$
   See [L-Y6: Sec.~5.2] (D(5))
    for more general discussions, details, and more references.
 \end{itemize}

\bigskip

\subsection{Remarks on general Azumaya-type noncommutative schemes.}

{From} the pure mathematical/geometric point of view,
it should be clear that the notion of (trivial or nontrivial)
 `Azumaya noncommutative schemes with a fundamental module'
 alone is not a final/complete picture.
Beginning with such a space $(X^{\!A\!z},{\cal E})$,
 let $X_{\cal A}$ be a surrogate of $X^{\!A\!z}$,
the category ${\cal C}$ that contains
  all Azumaya noncommutative schemes with a fundamental module
 should contain also $(X_{\cal A}^{A\!z}\,,\,_{\cal A}{\cal E})$,
  where ${\cal O}_{X_{\cal A}}^{A\!z}
                =\Endsheaf_{\cal A}(_{\cal A}{\cal E})$.
{From} this one starts to extend the set of objects of ${\cal C}$
 to include sheaves of orders with a generically fundamental module,
 ...$\,$.
Furthermore, from
  the naturality of operations on the category of sheaves of modules
   ([Il] and [Kas-S]) and
  the later development of D-branes since 1999 ([Sh1] and [Dou6]),
 one expects that one finally has to consider everything
 in the derived(-category) sense.

\bigskip
\bigskip

\begin{flushleft}
{\large\bf String-theoretical remarks on Sec.~2.}
\end{flushleft}
(1) [{\it Matrix gauged linear sigma model}$\,$]

\medskip

\noindent
{\bf Conjecture [matrix gauged linear sigma model].} {\it
 For each gauged linear sigma model in the sense of [Wi1],
 there exists a canonically constructed
  matrix/Azumaya gauged linear sigma model
 so that the moduli space of vacua of the latter
  is the stack of D0-branes,
   in the sense of Polchinski-Grothendieck Ansatz,
  on the moduli space of vacua of the former.
}

\bigskip

\noindent
(2) [{\it Evidence of Polchinski-Grothendieck Ansatz}$\,$]

\medskip

\noindent
Each of the six tests reviewed/presented compactly in Sec.~2.4
 has their own distinct feature.
Passing one does not imply passing another. Thus, all six tests
together give us a first evidence
 of the Polchinski-Grothendieck Ansatz
 as a foundational feature of D-branes.

\bigskip

\noindent
(3) [{\it Too much information?}$\,$]

\medskip

\noindent
G\'{o}mez and Sharpe pointed out in [G-Sh]
 that the scheme structure on D-branes may carry more information
 than D-branes do physically.
This rings also with the mathematical fact
 that not all morphisms are good, e.g.,
   in the sense of the existence of a perfect obstruction theory
   on the moduli stack of morphisms of a fixed combinatorial type.
So having too much information -- i.e.\
  the necessity to single out ``good" morphisms in our setting --
 is a question one should definitely address
 -- if not for stringy reasons, then for mathematical reasons.

\bigskip

\section{The differential/symplectic topological aspect:
         Azumaya noncommutative $C^{\infty}$-manifolds
         with a fundamental module and morphisms therefrom.}

Having reviewed Azumaya (noncommutative) geometry
 in the algebro-geometric setting,
we now take it as background to introduce
 Azumaya noncommutative $C^{\infty}$-manifold
  with a fundamental module  and
 smooth morphisms/maps therefrom to a complex projective manifold $Y$.
This Azumaya geometry
  in the differential/symplective topological category
 will be our prototypical picture of D-branes of A-type
  in superstring theory
  along the line of the Polchinski-Grothendieck Ansatz.
{For} simplicity, we assume that there is no $B$-field on $Y$
 in Sec.~3 and Sec.~4.
In particular, any Azumaya structure is assumed to be
 (Zariski-/analytic-/$C^{\infty}$-topology-)locally trivializable
 as $M_r(R)$ for $R$ the function/coordinate ring of a local chart.

\bigskip

\subsection{Azumaya noncommutative $C^{\infty}$-manifolds
    with a fundamental module and morphisms therefrom.}

Aspect~II and Aspect~IV presentations of
 morphisms from Azumaya $C^{\infty}$-manifolds are given.
The notion of $1$-forms and of their tensor products
 on an Azumaya $C^{\infty}$-manifold are introduced.

\bigskip

\begin{flushleft}
{\bf Maps from a real manifold(-stratified space)
     to a complex manifold/variety/stack.}
\end{flushleft}
Let $X$ be a real smooth manifold.

\begin{convention} {\rm [{\it $\,C^{\infty}$ structure sheaf}$\:$].
 To uniform the notation with algebro-geometry,
  we will denote the sheaf $C^{\infty}_X$ of smooth functions on $X$
  interchangeably by ${\cal O}_X^{\infty}$ or, simply, ${\cal O}_X$.
}\end{convention}

\noindent
Let $Y$ be a complex manifold.
Again, we will assume that $Y$ is projective
 though some constructions below do not require this.
Then a (smooth) map $f:X\rightarrow Y$ is specified by
 the pullback of functions
 $f^{\sharp}:{\cal O}_Y
     \rightarrow {\cal O}_X^{\infty}\otimes_{\Bbb R}{\Bbb C}$.
Identify $Y$ with a smooth variety over ${\Bbb C}$,
we shall think of $f$ as specifying a (real) $X$-family
 of ${\Bbb C}$-points on $Y$.

This interpretation applies also for $Y$ a singular complex variety.
Treating $X$ as a gluing system of local charts,
 the same picture applies when $Y$ is an Artin stack over ${\Bbb C}$.

The same applies also
 when $X$ is a topological space that is stratified by real manifolds.
In this case, the smooth structure sheaf
 ${\cal O}_X^{\infty}=:{\cal O}_X$ of $X$ is defined to be the sheaf
 of continuous functions on $X$
 that is smooth in the interior of all its manifold strata.

\bigskip

\begin{flushleft}
{\bf Morphisms from Azumaya $C^{\infty}$-manifolds
     with a fundamental module.}
\end{flushleft}
The theory of sheaves on real manifolds was already developed
  long ago, e.g.~[Kas-S].
It can be used to parallelly construct Azumaya geometry
 in the $C^{\infty}$-category
as is done in the algebro-geometric category.

\begin{definition}
{\bf [Azumaya $C^{\infty}$-manifold].} {\rm
 Let $X$ be a real (smooth) manifold with structure sheaf
  ${\cal O}_X^{\infty}$.
 An {\it Azumaya $C^{\infty}$-manifold with a fundamental module}
  is a triple
  $$
   (X\,,\,
    {\cal O}_X^{\infty,A\!z} :=
    \Endsheaf
     _{{\cal O}_X^{\infty}\otimes_{\Bbb R}{\Bbb C}} ({\cal E})\,,\,
    {\cal E})\;
   =:\; (X^{\!A\!z},{\cal E})
  $$
   where
    ${\cal E}$ is the sheaf of (smooth) sections of
    a (smooth) complex vector bundle $E$ over $X$.
 The sheaf ${\cal O}_X^{\infty,A\!z}$ of
  ${\cal O}_X^{\infty}\otimes_{\Bbb R}{\Bbb C}$-algebras
  is called the {\it Azumaya (noncommutative) structure sheaf}
  of $X^{\!A\!z}$.
 It contains ${\cal O}_X^{\infty}\otimes_{\Bbb R}{\Bbb C}$
  as its sheaf of centers.
}\end{definition}

With the interpretation of a map
  from a real-manifold-stratified topological space
  to a complex manifold/variety/stack
  in the previous theme,
all four aspects of
 a morphism from an Azumaya scheme with a fundamental module to $Y$
 are expected to be adoptable to the $C^{\infty}$-category
 to give four equivalent aspects of a morphism
 $\varphi: (X,{\cal O}_X^{\infty,A\!z},{\cal E})\rightarrow Y$.
However, for Aspect~I, to develop a theory in its own right
 to characterize
 an ${\cal O}_X^{\infty}\otimes_{\Bbb R}{\Bbb C}$-subalgebra
  ${\cal O}_X^{\infty}\otimes_{\Bbb R}{\Bbb C}\subset {\cal A}
   \subset {\cal O}_X^{\infty,A\!z}$  of ${\cal O}_X^{\infty,A\!z}$
  such that ${\cal A}_{\redscriptsize}$
   is the complexified  structure sheaf
   ${\cal O}_{X^{\prime}}^{\infty}\otimes_{\Bbb R}{\Bbb C}$
   of a real-manifold-stratified space $X^{\prime}$
 is a very technical language issue.
An easier path to take is treat Aspects~II and IV
 with slightly higher weight.
Once either is taken as the starting point,
 the remaining Aspects~I and III become a matter of translation.
In the following two definitions,
 we actually mix Aspacts~II and IV slightly
 to take care of the notion of `flat over $X$' and
 of `piecewise smooth surrogate':

\begin{definition}
{\bf [morphism: Aspect~II].}
 An Aspect~II presentation of
  a {\it morphism} from an Azumaya $C^{\infty}$-manifold
   with a fundamental module of rank $r$ to $Y$
  is the following data:
  \begin{itemize}
   \item[$\cdot$]
    A torsion sheaf $\widetilde{\cal E}$ of
    ${\cal O}_X\otimes_{\Bbb R}{\cal O}_Y$-modules on $X\times Y$
    that satisfies:
    \begin{itemize}
     \item[$(1)$]
      on each $\{p\}\times Y$,
       $\widetilde{\cal E}|_{\{p\}\times Y}$
       is a $0$-dimensional ${\cal O}_Y$-module of length $r$;

     \item[$(2)$]
      $(\Supp\widetilde{\cal E})_{\redscriptsize}\subset X\times Y$,
       with the induced topology, is stratified by
       smooth manifolds\footnote{Here
                        a subtle issue comes in:
                        In algebraic geometry,
                          the support of a sheaf ${\cal F}$
                            on a scheme $Z$ is the subscheme
                           defined by the ideal sheaf
                         {\it Ker}$\,({\cal O}_Z\rightarrow\,
                           {\cal E}\mbox{\it nd}_{{\cal O}_Z}({\cal F}))$
                         of ${\cal O}_Z$.
                        This is the most natural notion
                         of `support of a sheaf'
                         as it encodes some fuzzy structure
                           related to sections of ${\cal F}$.
                        Here, we are in $C^{\infty}$-category.
                        Naively, {\it Supp}$\,\widetilde{\cal E}$
                         would be just the set of points,
                          with the induced subset-topology,
                          on $X\times Y$
                         such that the stalk of $\widetilde{\cal E}$
                          at which is non-zero.
                        However, along the $Y$-direction,
                         it still makes sense to talk about
                         scheme-type structure.
                        Indeed, one wishes to define
                         {\it Supp}$\,\widetilde{\cal E}$
                         as close to scheme-theoretical support
                         as possible to reflect D-branes.
                        In the current work, we do not finalize
                         the resolution of this issue
                         in the $C^{\infty}$/symplectic category.
                        Rather, we use Aspect~IV to encode this
                         not-yet-defined structure
                         as much as possible.
                        Here, the redundant notation
                         ({\it Supp}$\,\widetilde{\cal E}$)$_{red}$
                         for the point-set support of
                         $\widetilde{\cal E}$ on $X\times Y$
                         is meant to keep this subtle point in mind.
                        See Sec.~2.4 Test (4) with [G-Sh]
                         and Remark~4.2.5.
                         };

     \item[$(3)$]
      for any $p\in X$, there exists a neighborhood $U$ of $p$
       such that
       there is a continuous map
        $f_U: U\rightarrow
         (\Quot^{\,H^0}\!({\cal O}_Y^{\oplus r},r))_{\redscriptsize}$
        with $f_U(p)\simeq \widetilde{\cal E}|_{\{p\}\times Y}$.
    \end{itemize}
  \end{itemize}
  Here we treat
   $(\Quot^{\,H^0}\!({\cal O}_Y^{\oplus r},r))_{\redscriptsize}$
   as a singular complex space with the analytic topology.
\end{definition}

\begin{remark} {\rm [$\,${\it from Aspect~II to Aspect~I}$\,$].
 Given an Aspect~II presentation $\widetilde{\cal E}$ on $X\times Y$
   of a morphism,
  let $\pr_1:X\times Y\rightarrow X$ and
    $\pr_2:X\times Y\rightarrow Y$ be the projection maps.
 Then, ${\cal E}$ on $X$ is recovered by
  $\pr_{1\ast}\widetilde{\cal E}$.
 The composition
  ${\cal O}_Y \stackrel{pr_2^{\sharp}}{\longrightarrow}
    {\cal O}_X\boxtimes_{\,\Bbb R}{\cal O}_Y \longrightarrow
    \Endsheaf_{{\cal O}_X\boxtimes_{\Bbb R}{\cal O}_Y}(\widetilde{\cal E})$
  induces a map between sheaves of rings
   $\varphi^{\sharp}:{\cal O}_Y\rightarrow
     \Endsheaf_{{\cal O}_X^{\infty}\otimes_{\Bbb R}{\Bbb C}}({\cal E})
     ={\cal O}_X^{\infty,A\!z}$.
  This recovers $\varphi:(X^{\!A\!z},{\cal E})\rightarrow Y$.
}\end{remark}

\begin{definition}
{\bf [morphism: Aspect~IV].}
 An Aspect~IV presentation of a {\it morphism}
  from an Azumaya $C^{\infty}$-manifold with a fundamental module
  of rank $r$ to $Y$ is the following data:
  \begin{itemize}
   \item[$\cdot$]
    A $\GL_r({\Bbb C})$-equivariant continuous map
     $f_{P_X}: P_X \rightarrow
      (\Quot^{\,H^0}\!({\cal O}_Y^{\oplus r},r))_{\redscriptsize}$,
     where $P_X$ is a smooth principal $\GL_r({\Bbb C})$-bundle
      over $X$,
    that satisfies:
    \begin{itemize}
     \item[$(1)$]
      the point-set support
      $(\Supp\widetilde{\cal E})_{\redscriptsize}$,
       with the subset topology,
      of the sheaf $\widetilde{\cal E}$ on $X\times Y$
       associated to $f_{P_X}$
       is stratified by smooth manifolds.
    \end{itemize}
  \end{itemize}
\end{definition}

\begin{remark} {\rm [$\,${\it from Aspect~IV to Aspect~I}$\,$].
 Note that
  the $\widetilde{\cal E}$ on $X\times Y$ associated
  to $f_{P_X}$ automatically satisfies Conditions (1) and (3)
  in Definition~3.1.3.
 Thus, it gives an Aspect~II presentation of a morphism,
  which can recover Aspect~I, `The fundamental setting',
  of a morphism by Remark~3.1.4.
}\end{remark}

\bigskip

\begin{flushleft}
{\bf K\"{a}hler differentials
     on an Azumaya noncommutative scheme/$C^{\infty}$-manifold.}
\end{flushleft}
Before leaving this subsection, we introduce the basic notion of
 {\it K\"{a}hler differentials} (i.e.\ {\it $1$-forms})
 and their {\it tensor products} on an Azumaya noncommutative space.
This is a notion we can still bypass in this review/work
 (but see Remark~3.2.2);
however, they become important in [L-Y7].
Such notion for an associative unital algebra appeared earlier
 in, e.g., [C-Q] and [K-R].

\begin{definition}
{\bf [${\Bbb C}$-linear derivation].} {\rm
 Let $S$ be an associative (unital) ${\Bbb C}$-algebra and
  $M$ be a (two-sided) $S$-module.
 A map $d:S\rightarrow M$ (as abelian groups) is called
  a {\it ${\Bbb C}$-linear derivation}
  if it is a homomorphism of ${\Bbb C}$-modules that satisfies
  the Leibniz rule
  $$
   d(fg)\;=\; (df)g + fdg \hspace{1em}\mbox{for $f,g\,\in\, S$}\,.
  $$

}\end{definition}

\begin{definition}
{\bf [K\"{a}hler differential].} {\rm
 Let $M_r(R)$ be the $r\times r$ matrix ring over
  a commutative ${\Bbb C}$-algebra $R$.
 Denote by $\Omega_{M_r(R)}$ the {\it
  module of (K\"{a}hler) differentials of $M_r(R)$ over ${\Bbb C}$}
  the (two-sided) $M_r(R)$-module
   generated by the set $\{dm\,:\, m\in M_r(R)\}$
   subject to the relations
   $$
    \begin{array}{llcl}
     \mbox{(${\Bbb C}$-linearity)}
       & \hspace{-2.2em}
         d(am+bm^{\prime}) & =
       & a\,dm\,+\,b\,dm^{\prime} \hspace{1em}
         \mbox{for $a,b\,\in\,{\Bbb C}$
               and $m,m^{\prime}\,\in\,M_r(R)$}\,,\\[.2ex]
     \mbox{(Leibniz rule)}
       & d(mm^{\prime}) & =
       & (dm)m^{\prime}\,+\, m\,dm^{\prime}\,,\\[.2ex]
     \mbox{(commutativity pass-over)}
       & m(dm^{\prime}) & =
       & (dm^{\prime})m \hspace{1em}
         \mbox{for $m,m^{\prime}\,\in\, M_r(R)$}\\
     &&& \hspace{7em}\mbox{that commutes: $mm^{\prime}=m^{\prime}m$}\,.
    \end{array}
   $$
 By construction, it is equipped with a built-in
  ${\Bbb C}$-linear derivation
  $$
   d\;:\;M_r(R)\; \longrightarrow\; \Omega_{M_r(R)}
     \hspace{1em}\mbox{defined by}\hspace{1em} m\;\longmapsto\; dm\,.
  $$
} \end{definition}

\begin{definition} {\bf [tensor product].} {\rm
 The {\it $s$-fold tensor product} $\otimes_{M_r(R)}^s\Omega_{M_r(R)}$
  of $\Omega_{M_r(R)}$ over $M_r(R)$ for $s\in{\Bbb Z}_{\ge 0}$
  is the bi-$M_r(R)$-module with
   generators $dm_1\otimes\,\cdots\,\otimes dm_s$, $m_j\in M_r(R)$,
    and
   relators
   $$
    dm_1\otimes\,\cdots\,\otimes
                  (dm_i)m \otimes dm_{i+1}
                      \otimes\,\cdots\,\otimes dm_s
                = dm_1\otimes\,\cdots\,\otimes
                   dm_i\otimes m dm_{i+1}
                      \otimes\,\cdots\,\otimes dm_s\,,
   $$
   $i=1,\,\ldots\,,\,s-1$, for all $m$, $m_j\in M_r(R)$.
}\end{definition}

{For} $r=1$, the above reduces to the usual notion of
 K\"{a}hler differentials and their tensor products
 in commutative ring theory; e.g., [Ei] and [Ma].

As the above construction commutes with central localizations
 of $M_r(R)$, one has the following sheaves
 via central localizations and gluings:

\begin{lemma-definition}
{\bf [sheaf of K\"{a}hler differential].}
 Given an Azumaya scheme
  $X^{\!A\!z}=(X,{\cal O}_X^{A\!z}=\Endsheaf_{{\cal O}_X}{\cal E})$,
  one obtains a sheaf $\Omega_{X^{\!A\!z}}$ on $X^{\!A\!z}$ via gluing
  $\Omega_{M_r(R)}$'s from affine open sets $U=\Spec R$ of $X$.
 {\rm $\Omega_{X^{\!A\!z}}$ is called
  the {\it sheaf of K\"{a}hler differentials} on $X^{\!A\!z}$.}
 Similarly, one has its tensor products
  $\otimes_{{\cal O}_{X^{\!A\!z}}}^s\Omega_{X^{\!A\!z}}
   =: \Omega_{X^{\!A\!z}}^{\otimes s}$
     over ${\cal O}_X^{A\!z}$ for $s\in {\Bbb Z}_{\ge 0}$.
\end{lemma-definition}

Similar construction applies to Azumaya $C^{\infty}$-manifolds,
 in which case we will call the resulting K\"{a}hler differentials
 also {\it $1$-forms}.

In terms of this, the commutativity-pass-over property of K\"{a}hler
 differentials/$1$-forms implies the following induced map
 defined in exactly the same way as in the commutative case:

\begin{lemma}
 Given a morphism $\varphi:X^{\!A\!z}\rightarrow Y$
  either from an Azumaya scheme to a (commutative) scheme
  or from a $C^{\infty}$ Azumaya manifold to a complex manifold,
 then there is an induced map
  $\varphi^{\ast}:\Omega_Y\rightarrow \Omega_{X^{\!A\!z}}$
  as ${\cal O}_Y$-modules, defined locally by
  $df \mapsto d(\varphi^{\sharp}(f))$.
 Here, $\varphi^{\sharp}:{\cal O}_Y\rightarrow {\cal O}_X^{A\!z}$
  is the defining pull-back-of-functions underlying $\varphi$.
 Similarly; for the existence of
  $\varphi^{\ast}:\Omega_Y^{\otimes s}
                  \rightarrow \Omega_{X^{\!A\!z}}^{\otimes s}$.
\end{lemma}

\bigskip

\subsection{Lagrangian morphisms and special Lagrangian morphisms:\\
    Donaldson and Polchinski-Grothendieck.}

\begin{definition}
{\bf [Lagrangian/special Lagrangian morphism].} {\rm
 Let
  $X$ be a smooth manifold of (real) dimension $n$ and
  $(Y,\omega)$ be a complex manifold of (complex) dimension $n$
   from a projective variety$/{\Bbb C}$
   with a K\"{a}hler form $\omega$.
 A morphism $\varphi:(X^{\!A\!z},{\cal E})\rightarrow Y$
  is said to be a {\it Lagrangian morphism}
  if in its Aspect~II presentation, say, by
  a torsion sheaf $\widetilde{\cal E}$ of
    ${\cal O}_X\otimes_{\Bbb R}{\cal O}_Y$-modules on $X\times Y$,
  the following conditions are satisfied:
  \begin{itemize}
   \item[(1)] ({\it generically immersion})\hspace{1em}
    the restriction of the projection map
     $\pr_2:X\times Y\rightarrow Y$
    to $(\Supp\widetilde{\cal E})_{\redscriptsize}$
     is an immersion on a dense open subset;

   \item[(2)] ({\it Lagrangian condition})\hspace{1em}
    the restriction of $\pr_2^{\ast}\omega$ to
     $(\Supp\widetilde{\cal E})_{\redscriptsize}$ vanishes.
  \end{itemize}

 If furthermore $(Y,\omega)$ is a Calabi-Yau manifold
  with a calibration given by $\Real\Omega$ of
  a holomorphic $n$-form $\Omega$,
 then $\varphi:(X^{\!A\!z},{\cal E})\rightarrow Y$
  is said to be a {\it special Lagrangian morphism}
 if it is a Lagrangian morphism and
  \begin{itemize}
   \item[(3)] ({\it calibration condition})\hspace{1em}
    $\pr_2^{\ast}\Real\Omega=\pr_2^{\ast}\vol_n$
    holds on an open dense subset of
    $(\Supp\widetilde{\cal E})_{\redscriptsize}$.
  \end{itemize}
 Here, $\vol_n$ is the real volume-$n$-form associated to
  the K\"{a}hler metric of $Y$ from $\omega$.
}\end{definition}

\begin{remark}{\rm
 [$\,${\it intrinsic Lagrangian/calibration condition}$\,$].
 In terms of $1$-forms on $X^{\!A\!z}$,
  the Lagrangian and the calibration condition
  can be expressed more intrinsically/compactly as
  $$
   \varphi^{\ast}\omega\,=\, 0
    \hspace{1em}\mbox{and}\hspace{1em}
   \varphi^{\ast}\Real\Omega\,=\,\varphi^{\ast}\vol_n\,.
  $$
}\end{remark}

\bigskip

{To} reflect the behavior of D-branes of A-type correctly,
we need a notion of ``multiplicity" of a Lagrangian cycle
 that takes into account not just the cycle alone
 but also the bundle/sheaf it carries.
The following version of such is what we will use.\footnote{C.-H.L.\
                would like to thank Katrin Wehrheim for a discussion
                on coincident Lagrangian submanifolds and their
                multiplicities, spring 2008.}

\begin{definition}
{\bf [$\varphi_{\ast}[X]$ and $\varphi_{\ast}[(X,{\cal E})]$
      for Lagrangian $\varphi$].}
{\rm
 Let
  $X$ be oriented with the associated fundamental cycle $[X]$,
  $\varphi:(X^{\!A\!z},{\cal E})\rightarrow Y$ a morphism,
  $\widetilde{\cal E}$ on $X\times Y$ its presentation,  and
  $\pr_1:X\times Y\rightarrow X$,
  $\pr_2:X\times Y\rightarrow Y$ the projection maps.
 Then the orientation on $X$ induces an orientation on the
  $n$-dimensional strata of
  $(\Supp\widetilde{\cal E})_{\redscriptsize}$.
 This defines a fundamental cycle
  $[(\Supp\widetilde{\cal E})_{\redscriptsize}]$
  of $(\Supp\widetilde{\cal E})_{\redscriptsize}$.
 Define $\varphi_{\ast}[X]$ to be the (real) $n$-cycle
  $\pr_{2\ast}[(\Supp\widetilde{\cal E})_{\redscriptsize}]$ on $Y$.

 If furthermore $\varphi$ is a Lagrangian morphism,
  then let
   $$
    [(\Supp\widetilde{\cal E})_{\redscriptsize}]\;
     =\; \sum_{\alpha}\Delta_{\alpha}
   $$
   be a fine enough triangulation of
    $(\Supp\widetilde{\cal E})_{\redscriptsize}$
   such that the following hold:
    \begin{itemize}
     \item[(1)]
      $\sum_{\alpha}\Delta_{\alpha}$
       refines the manifold-stratification of
        $(\Supp\widetilde{\cal E})_{\redscriptsize}\,$;

     \item[(2)]
      for each $n$-simplex $\Delta_{\alpha}$
       as a subset of $X\times Y$,
      $\pr_1:\Delta_{\alpha}\rightarrow X$ is an embedding;

     \item[(3)]
      for each $n$-simplex $\Delta_{\alpha}$
        as a subset of $X\times Y$,
       define $l(p)$, $p\in \Delta_{\alpha}$, to be the length
       of the stalk $(\widetilde{\cal E}|_{\{pr_1(p)\}\times Y})_p$
       on $\{\pr_1(p)\}\times Y$;
      then the function $l(\,\bullet\,)$
       is constant in the interior of $\Delta_{\alpha}$;
      denote this constant by $l_{\alpha}$.
    \end{itemize}
 Define
  $$
   \varphi_{\ast}[(X,{\cal E})]\; :=\;
    \pr_{2\ast}
     \left(\sum_{\alpha}\,l_{\alpha}\,\Delta_{\alpha}\right)\,.
  $$
 By taking a common refinement of triangulations of
  $(\Supp\widetilde{\cal E})_{\redscriptsize}$,
 it is clear that $\varphi_{\ast}[(X,{\cal E})]$ is well-defined.
}\end{definition}

\begin{notation} {\rm {\bf
  [Chan-Paton-sheaf-adjusted multiplicity of Lagrangian cycle].}
 With the notation from above, we will denote
 a Lagrangian cycle of the form $\varphi_{\ast}[(X,{\cal E})]$
  also as $[(\varphi_{\ast}[X],\varphi_{\ast}{\cal E})]$
  to emphasize this adjustment of multiplicities
  along the Lagrangian cycle $\varphi_{\ast}[X]$ due to
  the Chan-Paton bundle/sheaf $\varphi_{\ast}{\cal E}$ over it.
}\end{notation}

\bigskip

Recall Donaldson's description [Don] of Lagrangian submanifolds/cycles
  and special Lagrangian submanifolds/cycles $L$
  in a Calabi-Yau manifold $Y$
 as the image of a special class of maps $f$
  from a smooth manifold $S$, equipped with a volume-form $\sigma$,
  to $Y$,
 selected by a moment map associated to the (right) action of
  the group $\Diff(S,\sigma)$ of volume-preserving diffeomorphisms
  of $(S,\sigma)$
  on the space $\Map(S,Y)$ of smooth maps from $S$ to $Y$ by
  precomposition of maps;
 see also [Hi2].
The aspect of treating Lagrangian or special Lagrangian $L\subset Y$
  as the image $f(S)$ of $f:S\rightarrow Y$  and
 the fact that such $L$ are candidates for supersymmetric D-branes
 naturally make one wonder:
 \begin{itemize}
  \item[$\cdot$] {\bf Q.}\
  \parbox[t]{13cm}{\it {\bf [Donaldson + Polchinski-Grothendieck]}\\
   Can Donaldson's aspect of Lagrangian and special Lagrangian
    submanifolds/cycles
   and Polchinski-Grothendieck's aspect of D-branes
   merge?}
 \end{itemize}
In this subsection, we discuss a special class of morphisms $\varphi$
 from Azumaya spaces $(X^{\!A\!z}, {\cal E})$
  with a fundamental module over a fixed $X$
  to a Calabi-Yau manifold $Y$
 when the answer is yes.
The symplectic construction needed in the discussion
 follows [Don].
{For} simplicity of presentation, we
 identify a vector bundle $V$
  with its sheaf ${\cal V}$ of local sections and
 denote both by the sheaf notation ${\cal V}$.

\bigskip

Let
 $(X,\sigma_X)$ be a smooth manifold of real dimension $n$
   with a volume-form $\sigma_X$,
 $S$ be a smooth manifold of real dimension $n$,
 ${\cal V}$ be a (smooth) complex vector bundle on $S$
   of (${\Bbb C}$-)rank $r_0$,  and
 $(Y, \omega,\Omega)$ be a Calabi-Yau manifold of complex dimension $n$
  with a K\"{a}hler form $\omega$ and a holomorphic $n$-form $\Omega$.
Let $\pr_1:X\times Y\rightarrow X$ and $\pr_2:X\times Y$
 be the projection maps.
Then a smooth maps $(c,g):S \rightarrow X\times Y$,
  where $c:S\rightarrow X$ is a finite cover of $X$, and
        $g:S\rightarrow Y$,
 defines a torsion sheaf
  $\widetilde{\cal E}_{(c,g)}:= (c,g)_{\ast}{\cal V}$ on $X\times Y$.
Its pushforward
 ${\cal E}_{(c,g)}:= \pr_{1\ast}\widetilde{\cal E}_{(c,g)}
  = c_{\ast}{\cal V}$
 is a complex vector bundle on $X$ of rank $r=dr_0$,
 where $d$ is the degree of $c$.
It follows from Sec.~2.2 that
 $(c,g)$ induces
 a morphism
  $$
   \varphi_{(c,g)}\;:\;
    (X^{\!A\!z}_{(c,g)},{\cal E}_{(c,g)})\; \longrightarrow\; Y
  $$
  with surrogates $X_{\varphi_{(c,g)}} = (c,g)(S)$
  and maps $\pi_{\varphi}:X_{\varphi_{(c,g)}}\rightarrow X$
       and $f_{\varphi_{(c,g)}}:X_{\varphi_{(c,g)}}\rightarrow Y$
  induced by $\pr_1$ and $\pr_2$ respectively;
denote the built-in map $S\rightarrow X_{\varphi_{(c,g)}}$
 by $h_{(c,g)}$:
$$
 \xymatrix{
  {\cal V} \ar@{.>}[d]
        & & & & & \\
  S \ar[rd]^{h_{(c,g)}} \ar@/^2ex/[rrrrrd]^-g \ar@/_/[rddd]_-c
          & & \\
        & X_{\varphi_{(c,g)}}
          \ar[rrrr]_-{f_{\varphi_{(c,g)}}}
          \ar[dd]^-{\pi_{\varphi_{(c,g)}}}
          & & & & Y  \\  \\
        & X
 }
$$
Let $\Map^{\scriptsizecover/X}(S,X\times Y)$
 be the space of smooth maps $S\rightarrow X\times Y$
 such that its first component $S\rightarrow X$ is a cover of $X$.
This is an infinite-dimensional smooth manifold
 locally modelled on
 $C^{\infty}(S,c^{\ast}T_{\ast}X\oplus g^{\ast}T_{\ast}Y)$
 at a point $[(c,g)]$.
The correspondence
  $$
    (c,g)\;\longmapsto\;
    \left( \varphi_{(c,g)}:(X^{\!A\!z}_{(c,g)},{\cal E}_{(c,g)})
           \rightarrow Y \right)
  $$
 defines then a map
  $$
   \Map^{\scriptsizecover/X}(S,X\times Y)\;
    \longrightarrow\; \Map^{\!A\!z^{\!f}}\!(X,Y)
  $$
 from $\Map^{\scriptsizecover/X}(S,X\times Y)$
 into the space $\Map^{\!A\!z^{\!f}}\!(X,Y)$ of morphisms
  from Azumaya noncommutative manifolds
   with a fundamental module, supported on the fixed $X$, to $Y$.

Let $\sigma_c=c^{\ast}\sigma_X$ be a volume-form on $S$
 by lifting $\sigma_X$ via $c$.
Then, the $\Diff(S,\sigma_c)$-action on $\Map(S,Y)$ in [Don],
  recalled in the beginning of this subsection,
 induces a $\Diff(S,\sigma_c)$-action
 on $\Map^{\scriptsizecover/X}(S,X\times Y)$
 by acting on the second component $g$ of $(c,g)$.

\begin{lemma}
{\bf [$h_{(c,g)}$ and $\pi_{\varphi_{(c,g)}}$: generically covers].}
 There exists an open dense subset $U$ of $X$ such that:
  \begin{itemize}
   \item[$(1)$]
    $\pi_{\varphi_{(c,g)}}^{-1}(U)$
     is open dense in $X_{\varphi_{(c,g)}}$  and
    $c^{-1}(U)=h^{-1}(\pi_{\varphi_{(c,g)}}^{-1}(U))$
     is open dense in $S$;

   \item[$(2)$]
    the restrictions
     $\,h_{(c,g)}: c^{-1}(U)
         \rightarrow \pi_{\varphi_{(c,g)}}^{-1}(U)\,$ and
     $\,\pi_{\varphi_{(c,g)}}:
        \pi_{\varphi_{(c,g)}}^{-1}(U) \rightarrow U\,$
    are covering maps.
  \end{itemize}
\end{lemma}

\noindent
This follows from the fact
 that $(c,g)$ and
  hence $h_{(c,g)}$ and $\pi_{\varphi_{(c,g)}}$ are all smooth maps
 and that $c=\pi_{\varphi_{(c,g)}}\circ h_{(c,g)}$ is a covering map.
Recall the volume-form
 $\sigma_{X_{\varphi_{(c,g)}}}$ on the possibly singular
 $X_{\varphi_{(c,g)}}$ from lifting $\sigma_X$
 via $\pi_{\varphi_{(c,g)}}$.
The above lemma implies immediately:

\begin{corollary}
{\bf [$\varphi_{(c,g)}$ and $g$:
      same Lagrangian/special Lagrangian property].}\\
 {\rm (1)}
  $f_{\varphi_{(c,g)}}^{\ast}\omega =0$
  if and only if $g^{\ast}\omega=0$. \\
 {\rm (2)}
  $f_{\varphi_{(c,g)}}^{\ast}\Real\Omega
   = \sigma_{X_{\varphi_{(c,g)}}}$
   (resp.\ $f_{\varphi_{(c,g)}}^{\ast}\Imaginary\Omega =0$)
  if and only if
  $g^{\ast}\Real\Omega=\sigma_c$
   (resp.\ $g^{\ast}\Imaginary\Omega=0$).

 In particular,
  $\varphi_{(c,g)}$ is a Lagrangian morphism
    if and only if $g$ is a Lagrangian morphism; and
  $\varphi_{(c,g)}$ is a special Lagrangian morphism
    if and only if $g$ is a special Lagrangian morphism.
\end{corollary}

\noindent
Thus,
 Donaldson's picture of Lagrangian/special Lagrangian submanifolds
  in a Calabi-Yau space  and
 Polchinski-Grothendieck's picture of supersymmetric D-branes of A-type
  in a Calabi-Yau space are tied together
 for a special class of such submanifolds/branes.
The construction of Donaldson
 in [Don: Sec.~1.1 and Sec.~3.1]
 can now be applied to this special submoduli space of
  morphisms from Azumaya manifolds with the fixed base $X$
  to the Calabi-Yau manifold $Y$
 to characterize Lagrangian morphisms,
with special Lagrangian morphisms in this special class
 selected further by the calibration condition.

{To} proceed,
let
 $\Cover(S,X)$ be the space of covering maps $S\rightarrow X$,
 $C_0:\Cover(S,X)\times S\rightarrow X$ the universal covering map, and
 $$
  \Map^{\scriptsizecover/X}(S,X\times Y)\,
  =\, \Cover(S,X)\times \Map(S,Y)\; \longrightarrow\; \Cover(S,X)
 $$
  the forgetful map.
A connected component of $\Map^{\scriptsizecover/X}(S,X\times Y)$
 is a product of that of $\Cover(S,X)$ and that of $\Map(S,Y)$.
{For} each connected component of
 $\Map^{\scriptsizecover/X}(S,X\times Y)$
  with a base-section
  $[(C_0;g_0):\Cover(S,X)\times S\rightarrow X\times Y]$
  over $\Cover(S,X)$,
 where $g_0: \Cover(S,X)\times S\rightarrow Y$
  is induced from a map, also denoted by $g_0\in \Map (S,Y)$,
  via the composition
   $\Cover(S,X)\times S\rightarrow S\stackrel{g_0}{\rightarrow} Y$,
fix a base-reference relative $2$-form $\nu$
 in the relative de Rham cohomology class
 $g_0^{\ast}([\omega])
  \in H^2((\Cover(S,X)\times S)/\Cover(S,X);{\Bbb R})$.
Let
 $$
  \Diff( (\Cover(S,X)\times S , C_0^{\ast}\sigma_X) / \Cover(S,X) )
 $$
  be the relative group
  of relative-volume-preserving $S$-bundle-diffeomorphisms
  of $(\Cover(S,X)\times S, C_0^{\ast}\sigma_X)$ over $\Cover(S,X)$
  and
 $$
  \Calabi\;:\;
   \Diff( (\Cover(S,X)\times S , C_0^{\ast}\sigma_X) / \Cover(S,X) )\;
     \longrightarrow\; H^{n-1}(S;{\Bbb R})/H^{n-1}(S;{\Bbb Z})
 $$
 be a relative Calabi-homomorphism.
Then ${\cal G}_0 := \Ker(\Calabi)$
 is a relative Lie group over $\Cover(S,X)$
 whose relative Lie algebra ${\frak G}_0$
 can be identified with the relative exact $(n-1)$-forms
 on $(\Cover(S,X)\times S)/\Cover(S,X)$.

{For} $g:S\rightarrow Y$ in the same connected component as $g_0$,
 $g^{\ast}([\omega])=g_0^{\ast}([\omega])$;
thus, one can choose an $a\in \Omega^1(S)$
 so that $g^{\ast}\omega-\nu=da$.
Fix a such $a$ for $g$.
Then for any $c\in \Cover(S,X)$ and
 any vector field $\xi$ on $S$, one can define a pairing
 $$
  \langle\,a\,,\,\xi\, \rangle\; :=\; \int_S\,a(\xi)\,\sigma_c\,.
 $$

\begin{proposition}
{\bf [moment map].} {\rm ([Don: Sec.~1.1].)}
 For a vector field $\xi$ on $S$ associated to an element
   in ${\frak G}_0$ over $c\in \Cover(S,X)$,
  the pairing $\langle a,\xi\rangle$
   depends only on $\xi$ (and $(c,g)$),not on the choice of $a$.
 The relative linear functional on ${\frak G}_0/\Cover(S,X)$
    defined by $\xi\mapsto \langle a,\xi\rangle$
  gives a map
  $$
   \mu\;:\;
    \Map^{\scriptsizecover/S}(S,X\times Y)\;
     \longrightarrow\; {\frak G}_0^{\ast}
   \hspace{1em}(\,\mbox{as spaces over $\Cover(S,X)$}\,)
  $$
  which is a relative moment map for the action of
  ${\cal G}_0$ on $\Map^{\scriptsizecover/X}(S,X\times Y)$
  over $\Cover(S,X)$.
\end{proposition}

The locus
  ${\cal U}_{\scriptsizeLag}
     \subset \Map^{\scriptsizecover/X}(S,X\times Y)$
  of $(c,g)$
  whose associated $\varphi_{(c,g)}$ is a Lagrangian morphism
 lies in the zero-locus $\mu^{-1}(0)$ of the relative moment map $\mu$.
Inside ${\cal U}_{\scriptsizeLag}$ resides the locus
 ${\cal U}_{\scriptsizesLag}$ of $(c,g)$
 whose associated $\varphi_{(c,g)}$ is a special Lagrangian morphism.

\bigskip
\bigskip

\begin{flushleft}
{\large\bf String-theoretical remarks on Sec.~3.}
\end{flushleft}
(1) [{\it D-branes of A-type}$\,$: {\it cycles vs.\ maps}$\,$]

\medskip

\noindent
Due to the special Lagrangian submanifold nature
 of D-branes of A-type,
they are usually thought of as a sub-object in a Calabi-Yau manifold.
Donalson's picture changes that.
Thinking of them as maps into a target-space(-time)
 from an Azumaya manifold with a fundamental module
 makes it very direct to see how they behave
 under deformations or collidings-into-one.
It also gives us an anticipation of what structure should be there
 on these branes in its own right;
cf.~[D-K-S] and Sec.~4.2, theme:
`The generically filtered structure on the Chan-Patan bundle
 over a special Lagrangian cycle on a Calabi-Yau torus'
 of the current work.

\bigskip

\section{D-branes of A-type on a Calabi-Yau torus
         and their\\ transitions.}

So far in this project we have illustrated featural behaviors of
 D-branes of B-type as they fit well in the realm of algebraic geometry.
In this section\footnote{In
                 this section,
                 a $1$-{\it cycle}
                  means a cycle of ${\Bbb R}$-dimension $1$,
                 while the {\it rank of a complex vector bundle}
                  over a real manifold is by definition
                  the ${\Bbb C}$-rank.},
 we give an example of such behaviors for
 nonsupersymmetric D-branes and D-branes of A-type.
This example is only a toy model but has several pedagogical meanings.
Its simplicity allows one to see/learn things about D-branes
 in this category
 without being blocked/complicated by mathematical technicality
before one studies the same issue at the level of, e.g.,
 [Lee] and [Joy2].

\bigskip

Let
 ${\Bbb C}$ be the complex plane with coordinate $z$  and
 $C = C_{\tau}:={\Bbb C}/({\Bbb Z}+{\Bbb Z}\tau)$
 be the complex torus
  of modulus $\tau\in{\Bbb C}$,
   defined up to the $\SL(2,{\Bbb Z})$ transformation
    $\tau\mapsto (a\tau + b)/(c\tau+d)$,
    {\tiny
     $\left(
      \begin{array}{cc} a & b \\ c& d  \end{array}
     \right)$} $\in \SL(2,{\Bbb Z})$,
 with the flat metric specified by the K\"{a}hler form
 $\omega=dz\wedge d\overline{z}$.
With this metric,
a covariantly constant holomorphic $1$-form on $C$
  whose real part defines a calibration on $(C,\omega)$
 such that there exists a calibrated cycle on $(C,\omega)$
 is given by $\Omega_{\theta}=e^{i\theta}dz$
  for $\theta\in A_{\tau} \subset S^1$,
  parameterized by
  $\{\Arg(z): z\in {\Bbb Z}+{\Bbb Z}\tau\}\subset [0,2\pi)$.
The volume-minimizing property of a calibrated cycle implies that
 such a cycle is lifted to a collection of parallel straight lines
 in ${\Bbb C}$
 (with the standard K\"{a}hler form also denoted by
  $dz\wedge d\overline{z}$).
In particular, all the calibrated $1$-cycles on
 $(C,\omega,\Omega_{\theta})$ are submanifolds,
 with their connected components differing
 by isometric translations on $(C,\omega)$.
{For} the current case,
any $1$-cycle on $(C,\omega)$ is a Lagrangian cycle
while:

\begin{convention}
{\bf [special Lagrangian cycle/morphism on/to $(C,\omega)$].} {\rm
 A $1$-cycle $L$ on $(C,\omega)$
   is called a {\it special Lagrangian cycle}
  if it is a calibrated $1$-cycle with respect to $\Omega_{\theta}$
   for some $\theta\in S^1$.
 Similarly, for the notion of
  a {\it special Lagrangian morphism} $\varphi$ to $(C,\omega)$.
}\end{convention}

Naively, one may wonder that
 D-branes of A-type in the current situation
 is too trivial to even be considered as an example.
However, it should be remembered that
 D-branes are not just the underlying supporting cycles.
Among other things,
 they are equipped with a Chan-Paton bundle/sheaf/module
 which even in this example can has less trivial structures.
{To} easily see such latter structure,
 we employ Aspect~IV of morphisms in Sec.~2.2 for the discussion.

\bigskip

\subsection{The stack ${\frak M}^{\,0^{Az^f}}_{\,r}\!\!(C)$.}

Some details of the stack ${\frak M}^{\,0^{Az^f}}_{\,r}\!\!(C)$
 are required to understand morphisms to $C$ from Aspect~IV.

\bigskip

\begin{flushleft}
{\bf The stack ${\frak M}^{\,0^{Az^f}}_{\,r}\!\!(C)$  and
     its representation-theoretical atlas
      $\Quot^{\,H^0}\!({\cal O}_C^{\oplus r},r)$.}\footnote{In
                    this theme, we adopt notations from
                     complex/symplectic geometry.
                    However, one can treat the whole theme
                     in the realm of algebraic geometry over ${\Bbb C}$
                     and then take its valid analytic counterpart.}
\end{flushleft}
Recall the representation-theoretical atlas
 $$
  \Quot^{\,H^0}\!({\cal O}_C^{\oplus r},r)\;
   :=\;
   \{\, {\cal O}_C^{\oplus r}
          \rightarrow \widetilde{\cal E}\rightarrow 0\,,\;
        \length\widetilde{\cal E}=r\,,\;
        H^0({\cal O}_C^{\oplus r})
          \rightarrow H^0(\widetilde{\cal E})\rightarrow 0\,
      \}
 $$
 of the stack ${\frak M}^{\,0^{Az^f}}_{\,r}\!\!(C)$ of morphisms
  from a (non-fixed) Azumaya point
   with a fundamental module $\simeq {\Bbb C}^r$
  to $C$.

\begin{lemma}
{\bf [$\Quot^{\,H^0}\!({\cal O}_C^{\oplus r},r)$ smooth].}
 $\Quot^{\,H^0}\!({\cal O}_C^{\oplus r},r)$ is smooth
 of (complex) dimension $r^2$ for $C$ a smooth complex curve.
\end{lemma}

\begin{proof}
 One only needs to prove the statement around a point
  $[{\cal O}^{\oplus r}\rightarrow \widetilde{\cal E}\rightarrow 0]
   \in \Quot^{\,H^0}\!({\cal O}_C^{\oplus r},r)$
  with $\widetilde{\cal E}$ supported at a point $p\in C$.
 In this case, since there exists a branched-covering map
  $C\rightarrow \CP^1\simeq {\Bbb C}\cup \{\infty\}$
  that takes $p$ to $0\in {\Bbb C}$,
 one reduces the problem further to the case $C={\Bbb C}$.
 Since
  $\Quot^{\,H^0}\!({\cal O}_{\Bbb C}^{\oplus r},r)\simeq {\Bbb C}^{r^2}$
  (cf.~Example 2.1.1),
 the lemma follows.

\end{proof}

\begin{lemma}
{\bf [${\frak M}^{\,0^{Az^f}}_{\,r}\!\!(C)$ smooth].}
 ${\frak M}^{\,0^{Az^f}}_{\,r}\!\!(C)$
  is smooth of uniform stacky-dimension $0$.
\end{lemma}

\begin{proof}
 The smoothness of ${\frak M}^{\,0^{Az^f}}_{\,r}\!\!(C)$
  follows from Lemma~4.1.1 that it has a smooth atlas, and [Schl].
 For the stacky-dimension, note that
  the first projection map of the fibered-product
  $$
   \Quot^{\,H^0}\!({\cal O}_C^{\oplus r},r)
    \times_{{\frak M}^{\,0^{Az^f}}_{\,r}\!\!(C)}
   \Quot^{\,H^0}\!({\cal O}_C^{\oplus r},r)\;
   \stackrel{pr_1}{\longrightarrow}\;
   \Quot^{\,H^0}\!({\cal O}_C^{\oplus r},r)
  $$
  has uniform relative-dimension $r^2$.
 Since $\Quot^{\,H^0}\!({\cal O}_C^{\oplus r},r)$
  is smooth of dimension $r^2$,
  the stacky-dimension of ${\frak M}^{\,0^{Az^f}}_{\,r}\!\!(C)$
   is also uniform, whose value is given by
   $$
    \dimm {\frak M}^{\,0^{Az^f}}_{\,r}\!\!(C)\;
    =\;  \dimm \Quot^{\,H^0}\!({\cal O}_C^{\oplus r},r)\,
          -\, \reldimm\,\pr_1\;
    =\; 0\,.
   $$

\end{proof}

It follows from the proof of Lemma~4.1.1
 that the generalized Hilbert-Chow morphism
 $$
  \begin{array}{cccccl}
  \rho  & : & \Quot^{\,H^0}\!({\cal O}_C^{\oplus r},r)
        & \longrightarrow  &  C^{(r)} \\[.6ex]
       && [{\cal O}_C^{\oplus r}\rightarrow \tilde{\cal E}\rightarrow 0]
        & \longmapsto
        & \sum_{p\in C}\,\length(\widetilde{\cal E}_p)\,[p]
        &,
  \end{array}
 $$
 where
  $C^{(r)}$ is the $r$-th symmetric product of $C$  and
  $\widetilde{\cal E}_p$ is the stalk of $\widetilde{\cal E}$ at $p$,
 is locally modelled on the morphism
 $$
  \begin{array}{cccccl}
  \rho^{\prime}
    & : & {\Bbb C}^{r^2} & \longrightarrow  & {\Bbb C}^r       \\[.6ex]
       && m              & \longmapsto      & \det(\lambda-m)  &,
  \end{array}
 $$
 where
  ${\Bbb C}^{r^2}$ is the space of $r\times r$-matrices over ${\Bbb C}$
   and
  ${\Bbb C}^r$ is the space of monic polynomials of degree $r$
    in $\lambda$.
Let $\GL_r({\Bbb C})$ acts on $C^{(r)}$ on the right
 trivially by the identity map,
then $\rho$ is $\GL_r({\Bbb C})$-equivariant.
Thus, the decomposition of $C^{(r)}$ by the strata of subdiagonals
 lifts to a $\GL_r({\Bbb C})$-invariant stratification of
 $\Quot^{\,H^0}\!({\cal O}_C^{\oplus r},r)$.

{For} $r\ge 2$, the fibration $\rho$ has no sections.
However, through the canonical isomorphism of $C^{(r)}$ with
 the Hilbert-scheme $C^{[r]}$ of $0$-dimensional subschemes of $C$ of
 length $r$ and the universal subscheme on $(C^{[r]}\times C)/C^{[r]}$,
one obtains a map
 $$
  \sigma_{[r]}\,=:\,\sigma_{(r)}\; :\; C^{[r]}\,=\,C^{(r)}\;
   \longrightarrow\;  {\frak M}^{\,0^{Az^f}}_{\,r}\!\!(C)
 $$
 and, hence,
 a $\GL_r({\Bbb C})$-equivariant morphism over $C^{[r]}=C^{(r)}$
 via the Isom-functor construction:
 $$
  \widetilde{\sigma}_{[r]}\,=:\,\widetilde{\sigma}_{(r)}\; :\;
   P_{C^{[r]}}\,=:\,  P_{C^{(r)}}\;
     \longrightarrow\;  \Quot^{\,H^0}\!({\cal O}_C^{\oplus r},r)\,.
 $$
Here, $P_{C^{[r]}}$ (resp.\ $P_{C^{(r)}}$) is a principal
 $\GL_r({\Bbb C})$-bundle over $C^{[r]}$ (resp.\ $C^{(r)}$).

Another natural map into $\Quot^{\,H^0}\!({\cal O}_C^{\oplus r},r)$
 that is related to $C^{(r)}$ more directly
 can be constructed as follows.
Regard $C$ as the moduli space of $0$-dimensional subschemes
 of itself of length $1$ and
consider the universal subscheme
 ${\cal O}_{C\times C}\rightarrow {\cal O}_{\Delta_C} \rightarrow 0$
 on $(C\times C)/C$.
Here,
 all the products are over ${\Bbb C}$,
 $\Delta_C$ is the diagonal of $C\times C$, and
 $(C\times C)/C$ corresponds to the first projection map
  $\pr_1:C\times C\rightarrow C$.
Let $\pr_i:C^{\times r}\rightarrow C$, $1\le i\le r$,
 be the projection map to the $i$-th component.
Then, the direct sum of the exact complexes
 $$
  \oplus_{i=1}^r\,
   (\pr_i^{\ast}{\cal O}_{C\times C}\;\longrightarrow\;
      \pr_i^{\ast}{\cal O}_{\Delta_C}\; \longrightarrow\; 0)
 $$
 on $(C^{\times r}\times C)/C^{\times r}$
 defines a quotient
 $$
  {\cal O}_{C^{\times r}\times C}^{\oplus r}\;
   \longrightarrow\; \oplus_{i=1}^r\,\pr_i^{\ast}{\cal O}_{\Delta_C}\;
   \longrightarrow\; 0
 $$
 on $(C^{\times r}\times C)/C^{\times r}$.
This realizes $C^{\times r}$ as the moduli space of
  the quotients
  $\oplus_{i=1}^r\,({\cal O}_C\rightarrow {\cal O}_{p_i}\rightarrow 0)$,
  where $p_i\in C$,   and
it defines a map over $C^{(r)}\,$:
 $$
  \hat{\sigma}\; :\;  C^{\times r}\; \longrightarrow\;
   \Quot^{\,H^0}\!({\cal O}_C^{\oplus r},r)\,.
 $$
One can view $\hat{\sigma}$ as a {\it canonical multi-section} of
 $\rho : \Quot^{\,H^0}\!({\cal O}_C^{\oplus r},r) \rightarrow C^{(r)}$.
As such, the set-value $\hat{\sigma}^{-1}(\mbox{\boldmath $p$})$
 of $\hat{\sigma}$ at a point
 $\mbox{\boldmath $p$}=[(p_1,\,\cdots\,,p_r)]\in C^{(r)}$
 lies in a single $\GL_r({\Bbb C})$-orbit.

\bigskip

\begin{flushleft}
{\bf Orbit-closure inclusion relations.}
\end{flushleft}
The orbit-closure inclusion relations of the $\GL_r({\Bbb C})$-orbits
 in $\Quot^{\,H^0}\!({\cal O}_C^{\oplus r},r)$
 can be characterized in terms of Jordan forms, as follows.

\begin{definition} {\bf [support-length data].} {\rm
 Let
  $[{\cal O}_C^{\oplus r}\rightarrow \widetilde{\cal E}\rightarrow 0]
    \in  \Quot^{\,H^0}\!({\cal O}_C^{\oplus r},r)$,
 define the {\it support-length data} of
  $[{\cal O}_C^{\oplus r}\rightarrow \widetilde{\cal E}\rightarrow 0]$
  to be
  $\mbox{\boldmath $l$}
   :=\{(p, r_p): p\in C\,,\, r_p=\length({\widetilde{\cal E}_p})>0\}$.
 This is a finite set with $\sum_p\,r_p =r$  and
  is invariant under the $\GL_r({\Bbb C})$-action
  on $\Quot^{\,H^0}\!({\cal O}_C^{\oplus r},r)$.
}\end{definition}

\noindent
{For} two orbits $O_1$ and $O_2$
 in $\Quot^{\,H^0}\!({\cal O}_C^{\oplus r},r)$
 to have non-empty
 $\overline{O_1}\cap O_2$ or $O_1\cap \overline{O_2}$,
it is necessary that $O_1$ and $O_2$ are mapped under $\rho$
 to the same point on $C^{(r)}$.
The latter condition holds if and only if
 their associated support-length data are identical.
Observe also that if $O_1\cap \overline{O_2}\ne \emptyset$,
 then $O_1\subset \overline{O_2}$.

\begin{definition}
{\bf [tamed representative for orbit].} {\rm
 Let $O$ be an orbit with support-length data
  $\mbox{\boldmath $l$}=\{(p_i,r_{p_i}):\, i=1,\,\ldots\,,k\}$.
 Then
  $[{\cal O}_C^{\oplus r}\rightarrow \widetilde{\cal E}\rightarrow 0]
   \in O$
  is called a {\it tamed representative} of $O$
 if  it is a direct sum
  $\oplus_{i=1}^k\, ({\cal O}_C^{\oplus r_{p_i}}
                 \rightarrow \widetilde{\cal E}_i\rightarrow 0)$,
  where $(\Supp(\widetilde{\cal E}_i))_{\redscriptsize}=p_i$.
}\end{definition}

\noindent
Note that for an orbit $O$ with support-length data
 $\mbox{\boldmath $l$}=\{(p_i,r_{p_i}):\, i=1,\,\ldots\,,k\}$,
$\prod_{i=1}^k\,\GL_{r_{p_i}}({\Bbb C})$ acts transitively
 on the set of tamed representatives of $O$.

\begin{lemma} {\bf [big orbit relation from small orbit relation].}
 Let
  $O_1$ and $O_2$ be two $\GL_r({\Bbb C})$-orbits
   with identical support-length data
   $\mbox{\boldmath $l$}=\{(p_i,r_{p_i})\}_i$  and
  ${\cal O}_C^{\oplus r}
          \rightarrow \widetilde{\cal E}_1\rightarrow 0$,
   ${\cal O}_C^{\oplus r}
          \rightarrow \widetilde{\cal E}_2\rightarrow 0$
   are tamed representatives for $O_1$ and $O_2$ respectively.
 Then $O_1\subset \overline{O_2}$
 if and only if
  the corresponding $\prod_i \GL_{r_{p_i}}({\Bbb C})$-orbits satisfy
  the same relation
   $$
    (\mbox{$\prod$}_i \GL_{r_{p_i}}({\Bbb C}))
     \cdot [{\cal O}_C^{\oplus r}
               \rightarrow \widetilde{\cal E}_1\rightarrow 0]\,
    \subset\,
    \overline{
     (\mbox{$\prod$}_i \GL_{r_{p_i}}({\Bbb C}))
      \cdot [{\cal O}_C^{\oplus r}
              \rightarrow \widetilde{\cal E}_2\rightarrow 0]}\,.
   $$
\end{lemma}

\begin{proof}
 We need to show the `only-if' part.
 Assume that $O_1\subset \overline{O_2}$,
 then
  there exists a path $\gamma:[0,\infty)\rightarrow \GL_r({\Bbb C})$
   with $\gamma(0)=\Id$
  such that
  $\lim_{t\rightarrow \infty}\gamma(t)
    \cdot
     [{\cal O}_C^{\oplus r}
        \rightarrow \widetilde{\cal E}_2\rightarrow 0]
   = [{\cal O}_C^{\oplus r}
        \rightarrow \widetilde{\cal E}_1\rightarrow 0]$.
 Since
  $H^0({\cal O}_C^{\oplus r})
   \rightarrow H^0(\widetilde{\cal E}_1)\rightarrow 0$,
  $H^0({\cal O}_C^{\oplus r})
   \rightarrow H^0(\widetilde{\cal E}_2)\rightarrow 0$,  and
  both representatives are tamed representatives for orbits
   with the same support-length data,
 through the common decomposition
  $H^0({\cal O}_C^{\oplus r})={\Bbb C}^r=\oplus_i{\Bbb C}^{r_{p_i}}$,
 there exists a path
  $\gamma^{\prime}:[0,\infty)\rightarrow \prod_i\GL_{r_{p_i}}({\Bbb C})$
  with $\gamma(0)=\Id$ such that
 $\gamma$ and $\gamma^{\prime}$ are asymptotically the same
  in the sense that
   $\lim_{t\rightarrow \infty}(\gamma(t)-\gamma^{\prime}(t))=0$
   in the matrix-representation
   $\GL_r({\Bbb C})\hookrightarrow M_r({\Bbb C})$
   of $\GL_r({\Bbb C})$ with respect to
   the above decomposition of $H^0({\cal O}_C^{\oplus r})$.
 This implies that
  $\lim_{t\rightarrow \infty}\gamma^{\prime}(t)
    \cdot
     [{\cal O}_C^{\oplus r}
        \rightarrow \widetilde{\cal E}_2\rightarrow 0]
   = [{\cal O}_C^{\oplus r}
        \rightarrow \widetilde{\cal E}_1\rightarrow 0]$.
 The lemma follows.

\end{proof}

Let $J^{(\lambda)}_j\in M_j({\Bbb C})$ be the matrix
{\scriptsize
 $$
  \left[
   \begin{array}{cccc}
    \lambda   &          &        & 0       \\
    1         & \lambda  &        &         \\
              &  \ddots  & \ddots &         \\
    0         &          &   1    & \lambda
   \end{array}
  \right]_{j\times j}
 $$
{\normalsize A}}
Jordan form $J$ in $M_n({\Bbb C})$ is a matrix of the following form
$$
 \left[
  \begin{array}{ccc}
   A_1 &         & 0 \\
       & \ddots  &   \\
   0   &         & A_k
  \end{array}
 \right]
  \hspace{1em}
   \mbox{with each $A_i\in M_{n_i}({\Bbb C})$ of the form}\hspace{1em}
 \left[
  \begin{array}{ccc}
   J^{(\lambda_i)}_{d_{i1}} &         & \\
             & \ddots  &              \\
             &         &  J^{(\lambda_i)}_{d_{ik_i}}
  \end{array}
 \right]\,.
$$
Here, omitted entries are all zero,
 $n_1\ge \,\cdots\,\ge n_k> 0$, and $d_{i1}\ge\,\cdots\,\ge d_{ik_i}>0$.

Let
 $O$ be a $\GL_r({\Bbb C})$-orbit
  in $\Quot^{\,H^0}\!({\cal O}_C^{\oplus r},r)$ of support-length dada
  $\mbox{\boldmath $l$}=\{(p_i,r_{p_i}):\, i=1,\,\ldots\,,k\}$  and
 $[{\cal O}_C^{\oplus r}\rightarrow \widetilde{\cal E}\rightarrow 0]$
  be a tamed representative of $O$.
Then it follows from
  the proof of Lemma~4.1.1,
  Example~2.1.1, and
  a shifting of $z$ by $z-c$ for an appropriate $c\in {\Bbb C}$
   in that Example
 that each stalk $\widetilde{\cal E}_{p_i}$ of $\widetilde{\cal E}$
  can be represented by a matrix $m_i\in M_{r_{p_i}}({\Bbb C})$
  with characteristic polynomial $\det(\lambda-m_i)$
  equal to $\lambda^{r_{p_i}}$.

\begin{definition}
{\bf [Jordan-form data].} {\rm
 With the above notation, the tuple
  $\mbox{\boldmath $J$}_O := (J_{p_i})_{i=1}^k$,
  where $J_{p_i}$ is the Jordan form of $m_i$,
  is called the {\it Jordan-form data} associated to $O$.
}\end{definition}

\begin{definition}
{\bf [partial order on the set of Jordan-form data].} {\rm
 Given two Jordan-form data
  $\mbox{\boldmath $J$}_1=(J_{p_i})_{i=1}^{k_1}$ and
  $\mbox{\boldmath $J$}_2=(J_{q_j})_{j=1}^{k_2}$
  from the above construction,
 we say that
  $\mbox{\boldmath $J$}_1 \prec \mbox{\boldmath $J$}_2$
  if the following two conditions are satisfied:
  \begin{itemize}
   \item[(1)]
    The underlying support-length data are identical:
     $\mbox{\boldmath $l$}_1=\mbox{\boldmath $l$}_2$.
    I.e., $k_1=k_2$ and,
     up to a relabelling, $p_i=q_i$ with $r_{p_i}=r_{q_i}$.

   \item[$(2)$]
    $\rank((J_{p_i})^j) \le \rank (J_{q_i})^j)$ for all $j\in {\Bbb N}$.
  \end{itemize}
}\end{definition}

The above discussion reduces the problem of
 a characterization of the orbit-closure inclusion relations
 to the case of [Mo-T], [Ge], and [Dj]:

\begin{proposition}
{\bf [orbit-closure inclusion relation].} {\rm ([Mo-T], [Ge], and [Dj].)}
 Let $O_1$ and $O_2$ be two $\GL_r({\Bbb C})$-orbits
  in $\Quot^{\,H^0}\!({\cal O}_C^{\oplus r},r)$.
 Then,
  $O_1\subset \overline{O_2}$ if and only of
  $\mbox{\boldmath $J$}_{O_1} \prec \mbox{\boldmath $J$}_{O_2}$.
\end{proposition}

\noindent
In particular,
over each $\mbox{\boldmath $p$}=[(p_1,\,\cdots\,,p_r)]\in C^{(r)}$,
there are
 a unique maximal orbit, given by the image
  $\widetilde{\sigma}_{(r)}(P_{C^{(r)}})$ over {\boldmath $p$},  and
 a unique minimal orbit, given by the orbit
  that contains $\hat{\sigma}^{-1}(\mbox{\boldmath $p$})$.

\bigskip

\subsection{Special Lagrangian cycles with a bundle/sheaf
    on a Calabi-Yau torus\\ \`{a} la Polchinski-Grothendieck Ansatz.}

With the description of
 the stack ${\frak M}^{\,0^{Az^f}}_{\,r}\!\!(C)$,
 its representation-theoretical atlas
  $\Quot^{\,H^0}\!({\cal O}_C^{\oplus r},r)$, and
 the orbit-closure inclusion relations of the $\GL_r({\Bbb C})$-orbits
  therein in Sec.~4.1,
we are now ready to see,
 by studying Lagrangian/special Lagrangian morphisms to $(C,\omega)$
  from Aspect~IV in Sec.~2.2,
 that, indeed,
 even D-branes of A-type on a flat torus can have
 less visible structures.

\bigskip

\begin{flushleft}
{\bf Lagrangian/special Lagrangian morphisms to $(C,\omega)$
     as $\GL_r({\Bbb C})$-equivariant maps.}
\end{flushleft}
Any principal $\GL_r({\Bbb C})$-bundle $P$ over $S^1$ is trivial.
Fix a trivialization $P\simeq S^1\times \GL_r({\Bbb C})$
 with the identity section $S^1\rightarrow S^1\times \{e\}$.
Here, $e$ denotes the identity of $\GL_r({\Bbb C})$.
Whenever needed,
 we will identify the base $S^1$ with $S^1\times\{e\}\subset P$.
It follows that
 a morphism $\varphi: (S^{1,A\!z}, {\cal E})\rightarrow C$
 is specified by a smooth map
 $f: S^1 \rightarrow \Quot^{\,H^0}\!({\cal O}_C^{\oplus r},r)$.
Denote the universal quotient sheaf on
 $(\Quot^{\,H^0}\!({\cal O}_C^{\oplus r},r)\times C)/
                         \Quot^{\,H^0}\!({\cal O}_C^{\oplus r},r)$
 by
 ${\cal O}
      _{\scriptsizeQuot^{\,H^0}\!({\cal O}_C^{\oplus r},r)\times C}
      ^{\oplus r}
  \rightarrow {\cal Q}\rightarrow 0$.
Then, the $S^1$-family $f^{\ast}{\cal Q}$
 of $0$-dimensional ${\cal O}_C$-modules on $C$
 defines a map $S^1\rightarrow {\frak M}^{\,0^{Az^f}}_{\,r}\!\!(C)$,
 Aspect~III of $\varphi$.
Regarding the $S^1$-family as moving along $S^1$ in $(S^1\times C)/S^1$
 gives Aspect~II of $\varphi$.
The further projection/push-forward to $S^1$ and to $C$ recover $\varphi$.

{To} incorporate the flat geometry and the calibrations on $(C,\omega)$
 into the construction,
it is convenient to treat $f$ as a lifting of
 a $\underline{f}:S^1\rightarrow C^{(r)}$:
 $$
  \xymatrix{
    && \Quot^{\,H^0}\!({\cal O}_C^{\oplus r},r)\ar[d]^-{\rho} \\
   S^1\ar[rru]^-{f} \ar[rr]_-{\underline{f}}
    && C^{(r)} &.
  }
 $$
The following lemma is then immediate.

\begin{lemma}
{\bf [Lagrangian/special Lagrangian morphism].}
 Let
  $\pi:C^{\times r}\rightarrow C^{(r)}$ be the quotient map.
 Then,
  $f$ defines a Lagrangian morphism $\varphi$
  if and only if
  $\underline{f}$ is an immersion and
   the projection of the $1$-complex $\pi^{-1}(f(S^1))$
   to each factor of $C^{\times r}$ is also an immersion on the edges
   of the $1$-complex;
  $f$ defines a special Lagrangian morphism $\varphi$
  if and only if
  $\underline{f}$ is an immersion and
  there exists an $\Omega_{\theta}$ such that
   the projection of the edges of the $1$-complex $\pi^{-1}(f(S^1))$
   to each factor of $C^{\times r}$ are calibrated
   with respect to $\Omega_{\theta}$.
\end{lemma}

The following lemma follows from
 either the canonical identification $C^{(r)}=C^{[r]}$
 or the fact that $\Quot^{\,H^0}\!({\cal O}_C^{\oplus r},r)$ is smooth
                  and all fibers of $\rho$ are path-connected:

\begin{lemma}
{\bf [$\underline{f}$ liftable].}
 Any $\underline{f}:S^1\rightarrow C^{(r)}$ lifts to
 an $f:S^1\rightarrow \Quot^{\,H^0}\!({\cal O}_C^{\oplus r},r)$
 such that $\underline{f}=\rho\circ f$.
\end{lemma}

\begin{example}
{\bf [Lagrangian morphism].} {\rm
 Any generic
  $f: S^1\rightarrow \Quot^{\,H^0}\!({\cal O}_C^{\oplus r},r)$
  defines a Lagrangian morphism
  $\varphi:(S^{1,A\!z},{\cal E})\rightarrow (C,\omega)$.
}\end{example}

\begin{example}
{\bf [special Lagrangian morphism].} {\rm
 Denote the domain $S^1$ by $X$ and
 equip all $S^1$ in the discussion with an orientation
  via $S^1\simeq {\Bbb R}^1/(2\pi{\Bbb Z})$.
 The Donaldson/Polchinski-Grothendieck picture
   and Aspects~II and IV of morphisms
  can be combined to construct a special Lagrangian morphism $\varphi$,
  as follows.

 Let
  $L$ be a special Lagrangian $1$-cycle on $C$.
 For simplicity of presentation, we assume that $L$ is connected.
 Let
  $S=\amalg_iS^1$ and
  $\mbox{\boldmath $m$}=(m_i)_i$,
   $\mbox{\boldmath $d$}=(d_i)_i$ with $m_i, d_i\in {\Bbb Z}$
   be multiplicity vectors.
 Let $g:S\rightarrow L$ (resp.\ $c:S\rightarrow X$)
  be a covering map with degree specified by
  {\boldmath $m$} (resp.\ {\boldmath $d$}).
 This defines a map $(c,g):S\rightarrow X\times C$
  with degree $r=|\mbox{\boldmath $d$}|:= \sum_id_i$ over $X$
  and, hence, a map $\underline{f}:X\rightarrow C^{(r)}$.
 Any lifting
  $f$ of $\underline{f}$ defines then a special Lagrangian morphism
  $\varphi:(S^{1,A\!z},{\cal E})\rightarrow C$,
    with ${\cal E}$ of (complex) rank $r$,
   whose image is supported on $L$.
 For example, the (trivial) complex line bundle on $S$
  in the Donaldson/Polchinski-Grothendieck picture
  induces a lifting of $\underline{f}$
 while the identification $C^{(r)}=C^{[r]}$ induces another.
 They give rise to different special Lagrangian morphisms.
 (See next theme for further explanation.)

 When $L=\amalg_jL_j$ has more than one connected component,
  the above separate construction for each $L_j$
  can be combined/merged to one construction to give a morphism
  $\varphi$ with image supported on $L$.
}\end{example}

\bigskip

\begin{flushleft}
{\bf The generically filtered structure on the Chan-Patan bundle
     over a special Lagrangian cycle on
     a Calabi-Yau torus.}\footnote{A
                  hidden mild subtitle to this theme is:
                   {\it Donagi-Katz-Sharpe vs.\ Polchinski-Grothendieck}.
                  Readers are highly recommended to read the theme
                   alongside with the work [D-K-S] of
                   Donagi, Katz, and Sharpe
                   from open-string-states point of view
                   for the nilpotent/filtered structure
                   on the Chan-Paton sheaf addressed here.}
\end{flushleft}
{To} single out and manifest better the current theme,
consider the class of morphisms
 $\varphi:(S^{1,A\!z},{\cal E})\rightarrow C$
 such that both $\pi_{\varphi}:S^1_{\varphi}\rightarrow S^1$
  and $S^1_{\varphi}\rightarrow C$ are embeddings.
The corresponding
 $f:S^1\rightarrow \Quot^{\,H^0}\!({\cal O}_C^{\oplus r},r)$
 has the image of $\underline{f}:S^1\rightarrow C^{(r)}$
 contained in the lowest subdiagonal $C\subset C^{(r)}$,
 consisting of points of the form $[p,\,\cdots\,,p]$, $p\in C$.
Recall from Sec.~4.1
 that the subset $\rho^{-1}([p,\,\cdots\,,p])$ in
  $\Quot^{\,H^0}\!({\cal O}_C^{\oplus r},r)$
 consists of a partial-ordered collection of $\GL_r({\Bbb C})$-orbits.
A point in $\rho^{-1}([p,\,\cdots\,,p])$
 represents a punctual subscheme $Z$ of $C$ with
 $Z_{\redscriptsize}=p$ together with an ${\cal O}_Z$-module ${\cal F}$
 with a decoration
  $H^0({\cal O}_C^{\oplus r})={\Bbb C}^r
   \rightarrow H^0({\cal F})\rightarrow 0$.
As a $0$-dimensional scheme,
 $Z\simeq\Spec({\Bbb C}[z]/(z^{r^{\prime}}))$
 for some $r^{\prime}\le r$.
In terms of this expression,
the $z$-action on ${\cal F}$ induces a filtration
 $0\subset z^{r^{\prime}-1}{\cal F}\subset\,\cdots\,\subset
     z^2{\cal F}\subset z{\cal F}\subset {\cal F}$
 of ${\cal F}$.
It follows that the collection
 $\{f^{-1}(O):\, \mbox{$O$ is a $\GL_r({\Bbb C})$-orbit}\}$
 gives a finite decomposition of $S^1$ into a disjoint cyclic union
 $I_1\cup\{p_{12}\}\cup I_2\cup\{p_{23}\}
     \cup\,\cdots\,\cup I_k\cup\{p_{k1}\}$
 with
  $I_i$ an open interval in $S^1$ and
  $p_{i,i+1}=\overline{I_i}\cap \overline{I_{i+1}}$,
  $i=1,\,\ldots\,,k$.
  (By convention, $k+1\equiv 1$.)
{For} each $I_i=f^{-1}(O_i)$ (resp.\ $p_{i,i+1}=f^{-1}(O_{i,i+1})$),
 $(\varphi|_{I_i})_{\ast}({\cal E}|_{I_i})$
 (resp.\ $(\varphi|_{p_{i,i+1}})_{\ast}({\cal E}|_{p_{i,i+1}})$)
 is thus endowed with a filtration specified by $O_i$
 (resp.\ $O_{i,i+1}$).
Since $O_{i,i+1}\subset \overline{O_i}\cap\overline{O_{i+1}}$,
the filtration associated to $O_i$ and that associated to $O_{i+1}$
 can be regarded as refinements of
 the filtration associated to $O_{i,i+1}$.
Since $\rho^{-1}([p,\,\cdots\,,p])$ contains
 a unique maximal orbit,
  associated to the case $Z\simeq\Spec({\Bbb C}[z]/(z^r))$,  and
 a unique minimal orbit, associated to the case $Z\simeq\Spec{\Bbb C}$.
A generic $\varphi$ in the current class of morphisms under discussion
 has $\varphi_{\ast}{\cal E}$ filtered by a complete flag of subbundles.
In contrast, the case that $\varphi_{\ast}{\cal E}$ is not filtered
 is a most non-generic situation.

\begin{remark}{\rm
[{\it hidden scheme structure in symplectic geometry}].
 While the above filtration is very natural from scheme-theoretical
  aspect, in symplectic/calibrated geometry one does not usually
  think of having a scheme structure
  on/along the special Lagrangian cycle $L$.
 Since the filtration is on $\varphi_{\ast}{\cal E}$,
  one can still have the filtration structure
  without resorting to its hidden scheme-like source.
 Symplectic geometers may think of such a filtration
  on bundles/sheaves over Lagrangian/special Lagrangian cycles
  as a (weak) datum to encode details of
  how the latter merge or split under deformations.
}\end{remark}

\bigskip

\subsection{Amalgamation/decomposition (or assembling/disassembling)
    of special Lagrangian cycles with a bundle/sheaf.}

In this subsection\footnote{
                  Though this subsection, as it is, is not yet ready
                   to be subtitled
                    ``Denef vs.\ Polchinski-Grothendieck",
                   it is written with the work [De] of Denef in mind.
                  One cannot help but notice the similarity of
                   the assembling/disassembling behavior of
                   D-branes of A-type discussed
                    there via split attractor flows  and
                    here via deformations of morphisms
                     from Azumaya noncommutatve spaces
                     with a fundamental module.
                  Readers are highly recommended to read ibidem
                   alongside with the current subsection.}
 we discuss the amalgamation/decomposition (or assembling/disassembling)
  of special Lagrangian cycles with a bundle/sheaf on a Calabi-Yau torus
  through deformations of morphisms from an $(S^{1,A\!z},{\cal E})$.

{For} convenience,
fix a basis $H_1(C;{\Bbb Z})\simeq {\Bbb Z}\oplus {\Bbb Z}$
 with intersection form
 {\tiny $\left[\begin{array}{cc} 0 & 1 \\ -1& 0 \end{array}\right]$}.
With respect to this basis,
 the algebraic intersection number
  $\gamma_{(p_1,q_1)}\cdot\gamma_{(p_2,q_2)}$
  of a $(p_1,q_1)$-curve with a $(p_2,q_2)$-curve $\gamma_{(p_2,q_2)}$
  is given by
  {\tiny $\left|\begin{array}{cc}
           p_1 & q_1 \\ p_2 & q_2 \end{array}\right|$} $=p_1q_2-p_2q_1$
 and
a special Lagrangian cycle, oriented as a $1$-cycle, in $(p,q)$-class
 on $(C,\omega)$ is described by a straight/geodesic $(p,q)$-curve,
 which can have several connected components with multiplicities
 if $p$ and $q$ are not co-prime.
The purpose of this subsection is
 to explain the following proposition:

\begin{proposition}
{\bf [amalgamation/decomposition of D-branes of A-type].}
 Let
  $$
   (L_i,{\cal V}_i)\;
   =\; (\varphi_{i\,\ast}[S^1]\,,\,\varphi_{i\,\ast}{\cal E}_i)\,,
   \hspace{1em}i=1,\,\ldots,\,k\,,
  $$
  for some special Lagrangian morphisms
  $\varphi_i:(S^{1,A\!z},{\cal E}_i)\rightarrow C$,
  $i=1,\,\ldots,\,k\,$, respectively.
 Let
  $[(L_i,{\cal V}_i)]=(p_i,q_i)\in H_1(C;{\Bbb Z})$.
 Then there exists a special Lagrangian morphism
  $\varphi:(S^{1,A\!z},{\cal E})\rightarrow C$ with image class
  $[(\varphi_{\ast}[S^1]\,,\,\varphi_{\ast}{\cal E})]$
  in $(\sum_{i=1}^kp_i\,,\,\sum_{i=1}^kq_i)\in H_1(C;{\Bbb Z})$
  such that $\varphi$ can be deformed into a Lagrangian morphism
  $\varphi^{\prime}:(S^{1,A\!z},{\cal E})\rightarrow C$
  with
   image cycle $\varphi^{\prime}_{\ast}[S^1]=\sum_{i=1}^kL_i$ and
   push-forward
    $\varphi^{\prime}_{\ast}{\cal E}=\oplus_{i=1}^k{\cal V}_i$.
\end{proposition}

This a consequence of the following two lemmas:

\begin{lemma}
{\bf [amalgamation of morphisms].}
 Let $\varphi_i:(S^{1,A\!z},{\cal E}_i)\rightarrow C$, $i=1,\,2$,
  be two morphisms.
 Then there exists a morphism
  $\varphi_3:(S^{1,A\!z},{\cal E}_1\oplus{\cal E}_2)\rightarrow C$
  such that
  \begin{itemize}
   \item[$\cdot$]
    $\Image\varphi_3=\Image\varphi_1+\Image\varphi_2$
    as cycles on $C$;

   \item[$\cdot$]
   $\varphi_{\ast}{\cal E}_3
    = \varphi_{1\ast}{\cal E}_1 \oplus \varphi_{2\ast}{\cal E}_2$
   as torsion sheaves of ${\cal O}_C$-modules on $C$.
  \end{itemize}
\end{lemma}

\begin{proof}
 This follows from the following canonical morphism
 induced by taking the direct sum of two complexes:
 $$
  \begin{array}{cccl}
   \Quot^{\,H^0}\!({\cal O}_C^{\oplus r_1},r_1)
    \times \Quot^{\,H^0}\!({\cal O}_C^{\oplus r_2},r_2)
    & \longrightarrow
    & \Quot^{\,H^0}\!({\cal O}_C^{\oplus (r_1+r_2)},r_1+r_2) \\[.6ex]
   \left(\,
    [{\cal O}_C^{\oplus r_1}
          \rightarrow \widetilde{\cal E}_1\rightarrow 0]\,,\,
    [{\cal O}_C^{\oplus r_2}
          \rightarrow \widetilde{\cal E}_2\rightarrow 0]\, \right)
    & \longmapsto
    & [{\cal O}_C^{\oplus (r_1+r_2)} \rightarrow
        \widetilde{\cal E}_1\oplus \widetilde{\cal E}_2\rightarrow 0)]
    &.
  \end{array}
 $$
\end{proof}

\begin{lemma}
{\bf [special Lagrangian representative in deformation class].}
 Any morphism $\varphi:(S^{1,A\!z},{\cal E})\rightarrow C$
  can be deformed into a special Lagrangian morphism
  $\varphi_{sL}:(S^{1,A\!z},{\cal E})\rightarrow C$.\footnote{For
                the uniformization of the statement, here we allow
                $\varphi_{sL}$ to have {\it Im}$\,\varphi_{sL}$
                supported on
                a (possibly empty) special Lagrangian cycle
                plus possibly $0$-cycles on $C$.}
\end{lemma}

\begin{proof}
 Suppose that ${\cal E}$ is of (complex) rank $r$.
 Let $f:S^1\rightarrow \Quot^{\,H^0}\!({\cal O}_C^{\oplus r},r)$
  be the map associated to $\varphi$.
 Recall
  $\rho:\Quot^{\,H^0}\!({\cal O}_C^{\oplus r},r)\rightarrow C^{(r)}$.
 For the convenience of presentation,
  let $X=S^1$ and
  endow the smooth manifold $C^{(r)}$
   with the natural orbifold structure
   from $C^{(r)}=C^{\times r}/\Sym(r)$,
   where the permutation group $\Sym(r)$ acts on $C^{(r)}$
   by permuting the $r$-many components in the product.
 The union of all subdiagonals in $C^{(r)}$ corresponds to
   the set of orbifold-points on $C^{(r)}$.
 Homotope $f$ so that the corresponding
  $\underline{f}:S^1\rightarrow C^{(r)}$
  has image containing no orbifold-points of $C^{(r)}$.
 Denote the new map/morphism still by $f$ and $\varphi$.
 Then the surrogate $X_{\varphi}$ of $\varphi$
    is a smooth curve (with possibly several connected components)
    that covers $X$ with degree $r$.
 Let
  $H_1(X\times C;{\Bbb Z})
               \simeq {\Bbb Z}\oplus({\Bbb Z}\oplus{\Bbb Z})$
  be induced from the product and
  $H_1(C;{\Bbb Z})\simeq {\Bbb Z}\oplus {\Bbb Z}$
 we fixed at the beginning.
 Then $X_{\varphi}$ is a smooth curve representing a class
  $(r;p,q)\in H_1(X\times C;{\Bbb Z})$.

 Let $L_{(r;p,q)}$ be another smooth curve representative
  of $(r;p,q)$ that maps to a special Lagrangian representative
  of $(p,q)\in H_1(C)$.
 Consider
  the $4$-manifold $M:=[0,1]\times (X\times C)$
   with boundary $\{1\}\times (X\times C)-\{0\}\times (X\times C)$ and
  smooth curves $X_{\varphi}\subset \{0\}\times (X\times C)$
                and $L_{(r;p,g)}\subset \{1\}\times (X\times C)$.
 Then from the long-exact sequence of homologies
  (with ${\Bbb Z}$-coefficient)
  $$
   \cdots\; \longrightarrow\; H_2(\partial M)\;
   \stackrel{\alpha_2}{\longrightarrow}\; H_2(M)\;
   \stackrel{\beta_2}{\longrightarrow}\;  H_2(M,\partial M)\;
   \stackrel{\delta_2}{\longrightarrow}\; H_1(\partial M)\;
   \stackrel{\alpha_1}\longrightarrow\;   H_1(M)\;
   \longrightarrow\; \cdots
  $$
  with $\alpha_1$ and $\alpha_2$ surjective and, hence,
   $\beta_2$ a zero-map and $\delta_2$ injective.
 If follows that
  $$
   H_2(M,\partial M)\;\simeq\; \Ker\alpha_1\;
   \simeq\; \{(\gamma, -\gamma)\,:\, \gamma\in H_1(X\times C) \}\,
   \subset\, H_1((X\times C)\mbox{$\amalg$}(X\times C))\,.
  $$
 Consequently, $(-X_{\varphi}, L_{(r;p,q)})$ bounds a $2$-chain
  $\Sigma$ in $(M,\partial M)$
  with $\partial\Sigma=L_{(r;p,q)}-X_{\varphi}$.
 Furthermore, one can choose $\Sigma$ to be an embedded
  smooth, orientable surface with boundary.\footnote{This
                      follows from the proof of a classical theorem
                       in smooth $4$-manifold topology which,
                       in our case, says that
                      {\it for $M$ a smooth orientable $4$-manifold
                           with boubdary,
                       any class in $H_2(M,\partial M; {\Bbb Z})$
                       can be represented by a smooth embedded
                       orientable surface with boundary}.
                      See [Ki: II.1, Theorem~1.1 and remark] and
                       [G-St: Chap.~1,
                               Proposition~1.2.3 and Remark~1.2.4;
                              Chap.~4, Exercise~4.5.12(b)].
                      C.-H.L.\ would like to thank
                       Robert Gompf
                        for teaching him non-gauge-theory-type
                        4-manifold theory around 1998.}
 Since $\pi_2(X\times C)=0$,
 with a further deformation of the embedding of $\Sigma$
   in $[0,1]\times (X\times C)$ if necessary,
 one can assume that the map: $\pi:\Sigma\rightarrow [0,1]$ from
  the restriction of projection map is a Morse function on $\Sigma$
  with the index of any critical point of $\pi$, if exists,
  equal to $1$ only.
 It follows that $\Sigma$ defines a homotopy
 $F:[0,1]\times S^1
    \rightarrow \Quot^{\,H^0}\!({\cal O}_C^{\oplus r},r)$
  such that
   \begin{itemize}
    \item[(1)]
     $F|_{{\{0\}}\times S^1}=f$,

    \item[(2)]
     the image $\Image\underline{F}$ of $\underline{F}=\rho\circ F$
     in $C^{(r)}$ contains no other orbifold-points except possibly
     finitely many orbifold-points with the structure group
     ${\Bbb Z}/2$,

    \item[(3)]
     $f_1:= F|_{\{1\}\times S^1}$
     defines a special Lagrangian morphism $\varphi_{sL}$.
   \end{itemize}
 This proves the lemma.
 Cf.~{\sc Figure}~4-3-1.
 \begin{figure}[htbp]
  \epsfig{figure=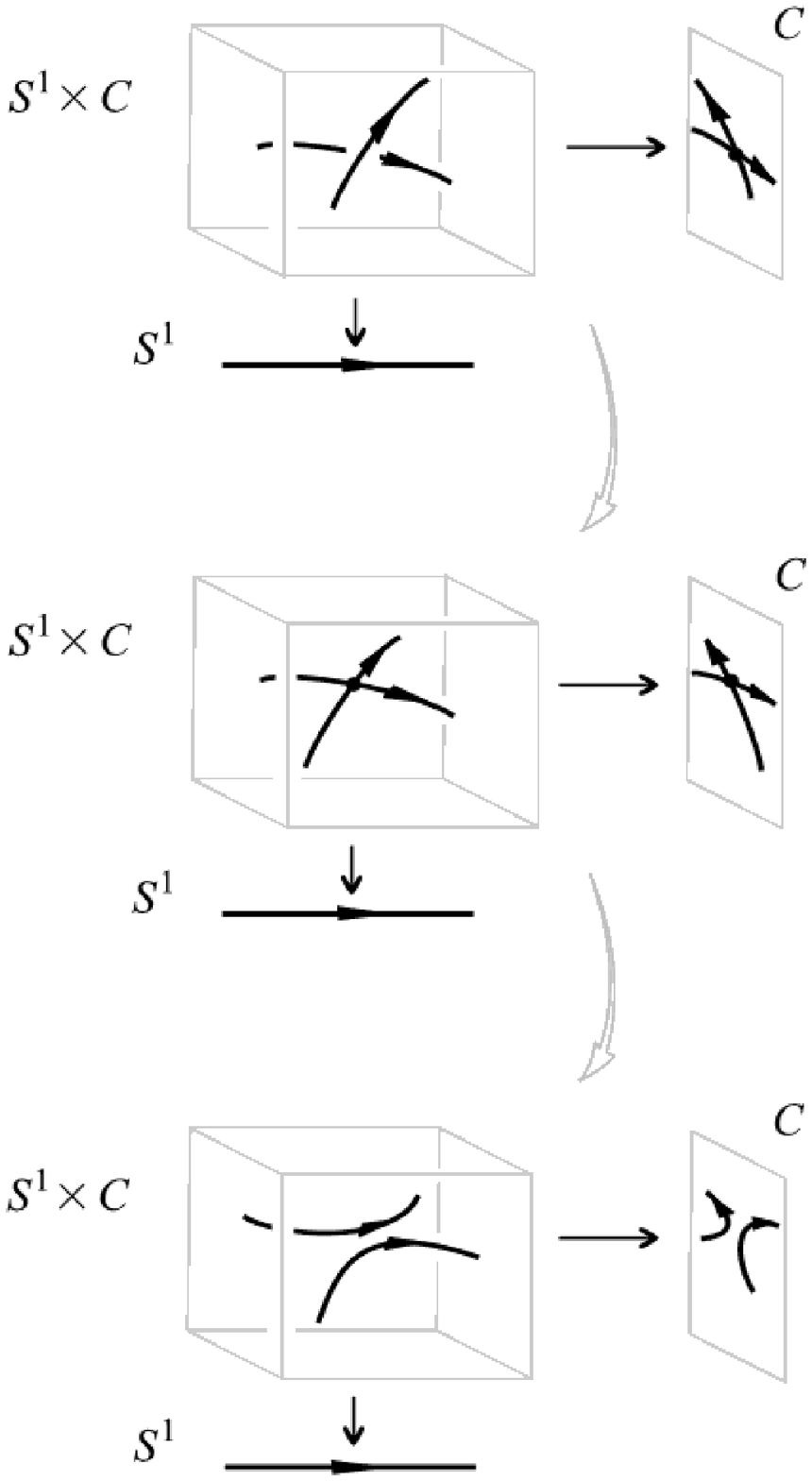,width=16cm}
  \centerline{\parbox{13cm}{\small\baselineskip 12pt
   {\sc Figure} 4-3-1.
   The basic local move/deformation-of-morphism that
    corresponds to crossing an index-$1$ critical point
    of $\pi:\Sigma\rightarrow [0,1]$.
   This has an effect of turning a short-string wrapping
    to a longer-string wrapping or a long string wrapping
    to a shorter string wrapping.
   Here, Aspect~II of a morphism is used.
   }}
 \end{figure}

\end{proof}

\begin{example}
{\bf [brane-anti-brane cancellation].} {\rm
 In particular, the situation of amalgamating special Lagrangian
  morphisms $\varphi_1$ and $\varphi_2$ with image class
  $(p,q),\,(-p,-q)\in H_1(C)$
 corresponds to a brane-anti-brane cancellation.
 Cf.~{\sc Figure}~4-3-2.
 \begin{figure}[htbp]
  \epsfig{figure=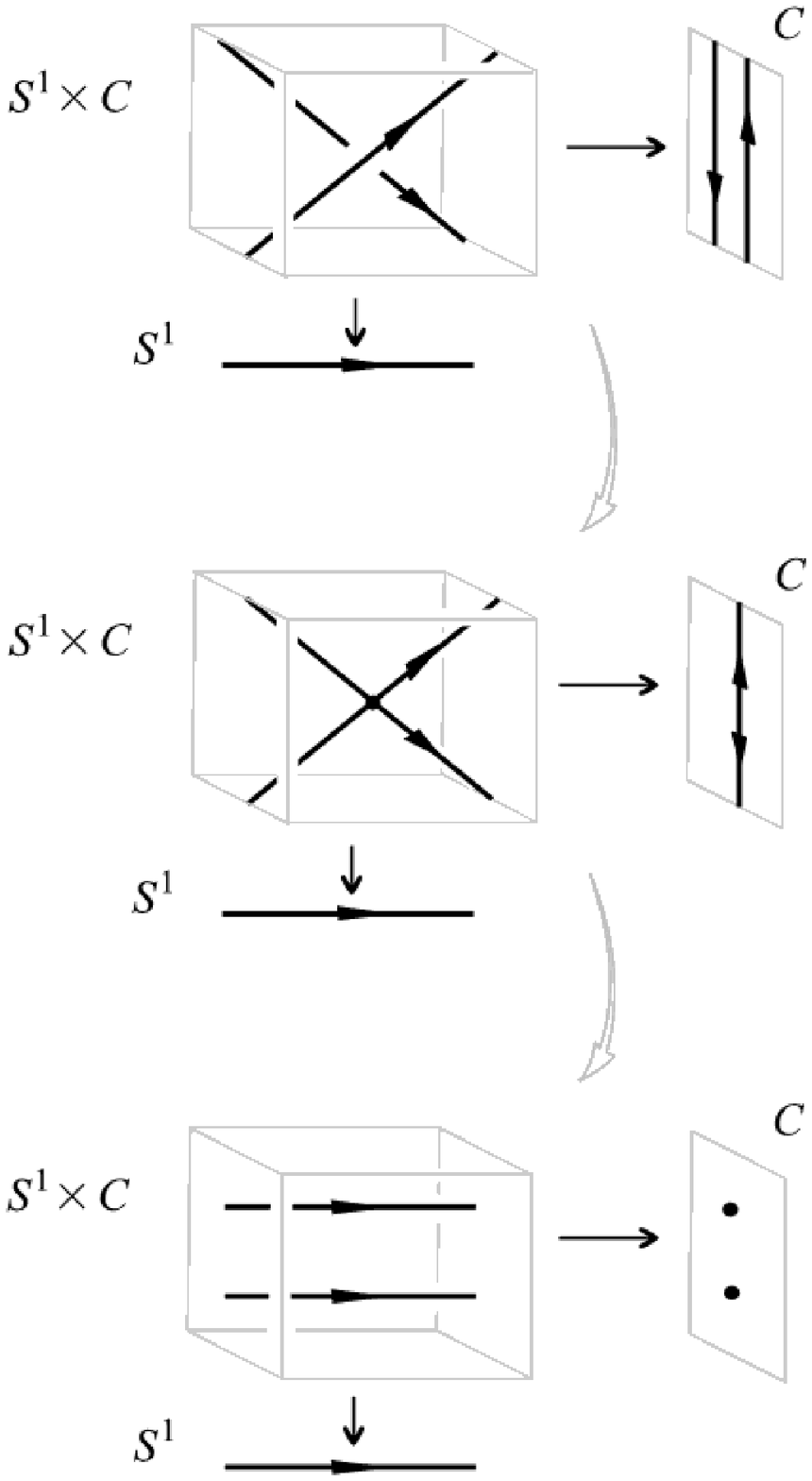,width=16cm}
  \centerline{\parbox{13cm}{\small\baselineskip 12pt
   {\sc Figure} 4-3-2.
   The brane-anti-brane cancellation procedure corresponds to
    a deformation of a morphism to one with $0$-dimensional image.
    In this figure, each opposite pair of faces of a parallelepiped
    are identified.
   Here, Aspect~II of a morphism is used.
   }}
 \end{figure}
}\end{example}

\begin{remark}{\rm
[{\it varying the moduli of the Calabi-Yau torus}].
It should be noted that
 the proof of Lemma~4.3.3 is essentially topological,
 depending only on the homology class of the special Lagrangian cycles.
Consequently,
 the mechanism of amalgamation/decomposition of
  special Lagrangian cycles with a bundle/sheaf in a Calabi-Yau torus
  through deformations of morphisms from an $(S^{1,A\!z},{\cal E})$
  as discussed
 works completely the same way
 even if the modulus of the Calabi-Yau torus varies
  in this process of assembling/disassembling of branes thereupon.
}\end{remark}

\bigskip
\bigskip

\begin{flushleft}
{\large\bf String-theoretical remarks for Sec.~4.}
\end{flushleft}
(1) [{\it The leftover/residual of brane-anti-brane cancellation}$\,$]

\medskip

\noindent
In morphism-from-Azumaya-space picture of
 the brane-anti-brane cancellation, all the process is simply
 a deformation of a nonconstant-type morphism to a constant-type
 morphism.
As such, after the brane cancellation process, there is no more
 local Ramond-Ramond charge of the original branes and yet
 there is still something left over, namely
the $0$-dimensional image-cycle/push-forward of the
 Azumaya space with a fundamental module under the final constant-type
 morphism.
This could represent an energy-lump that remains
 from the brane-anti-brane cancellation.

\bigskip

\noindent
(2) [{\it Short vs.\ long string wrapping}$\,$]

\medskip

\noindent
The short vs.\ long string wrapping behavior of
 matrix-strings in the string-theory literature
 (e.g., [D-V-V], [Joh: Sec.~16.3.3]; also[Ma-S])
 can be produced in this context by the same manner
 via morphisms from $S^{1,A\!z}$ and their deformations
 as well.
Cf.~{\sc Figure}~4-3-3.
 \begin{figure}[htbp]
  \epsfig{figure=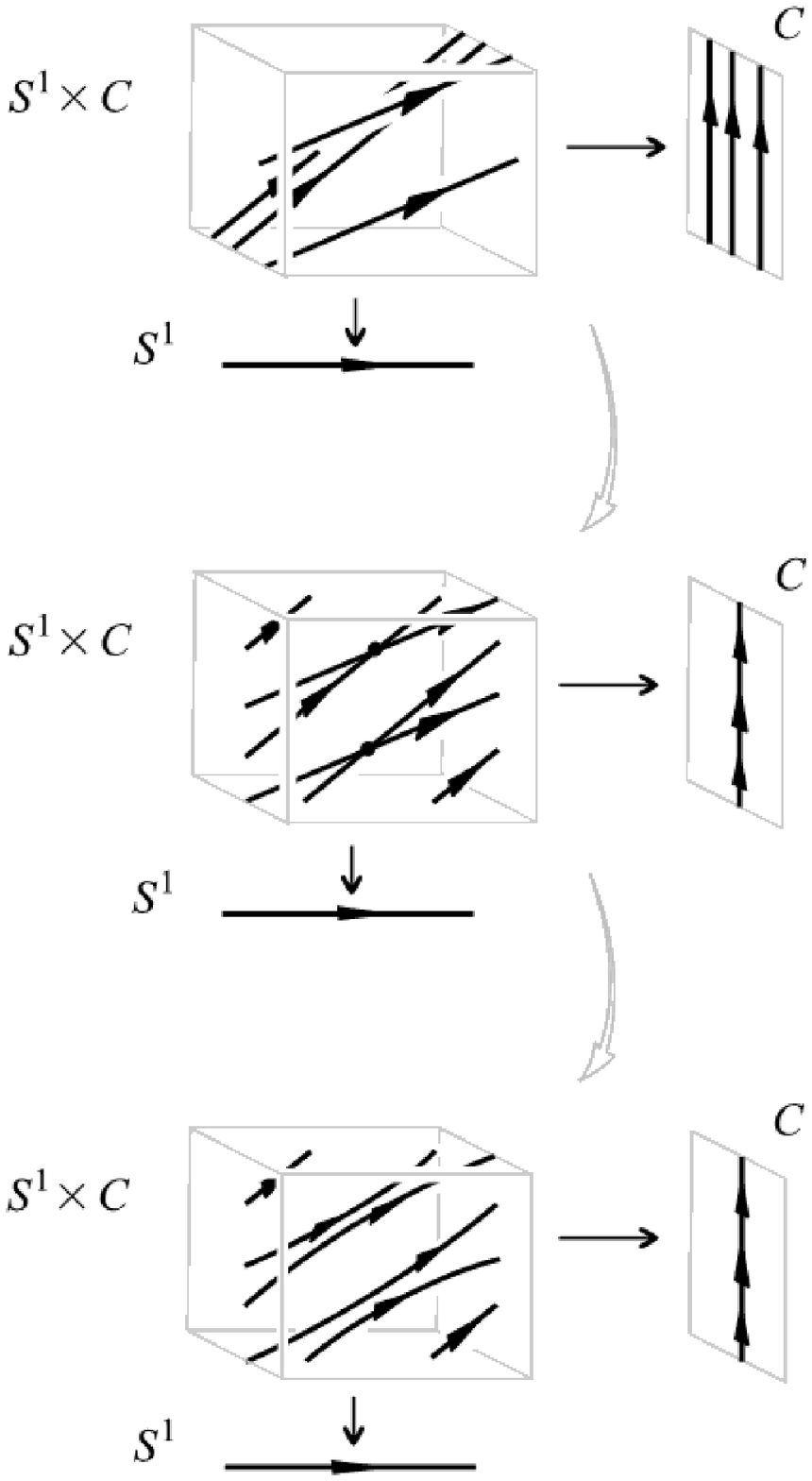,width=16cm}
  \centerline{\parbox{13cm}{\small\baselineskip 12pt
   {\sc Figure} 4-3-3.
   Under deformations of morphisms from an Azumaya string,
    the surrogate associated to the morphisms can change
    from a collection of short strings to a single long string.
   Here, Aspect~II of a morphism is used and
    a short-to-long string-wrapping transition
     corresponding to a merging
     $(1;1,0)+(1;1,0)+(2;1,0)\rightarrow (4;3,0)$
     in $H_1(S^1\times C;{\Bbb C})$
     of the surrogates associated to morphisms involved
    is illustrated.
   The multiplicity of an image cycle in $C$ in terms of
    the associated primitive one is indicated by the number
    of arrowhead.
   }}
 \end{figure}

\bigskip

\noindent
(3) [{\it Azumaya geometry and tensionless string}$\,$]

\medskip

\noindent
In type IIA superstring model of the superstring theory,
an open D$2$-brane can have its boundary attached to NS$5$-branes.
When a pair of parallel NS$5$-branes become coincident,
 the open D$2$-brane sandwiched between them
  becomes degenerate and $1$-dimensional:
  tensionless string\footnote{We
           thank Frederik Denef for a discussion
           on smeared D-branes and degenerate D-branes.}.
As it comes from the boundary of D$2$-branes,
 one anticipates that the Azumaya noncommutative structure
 on the original D$2$-brane passes to the tensionless string
 in the NS$5$-brane.
Morphisms from an Azumaya $S^1$ to the target NS$5$-brane
 become the most basic fields on the tensionless string.
In this way, a tensionless string is linked
 with a matrix/Azumaya string in a natural way.
We thus leave this section and, hence, this review/work
 with the following guiding question:
 \begin{itemize}
  \item[$\cdot$] {\bf Q.}\
  \parbox[t]{13cm}{\it
   {\bf [Azumaya geometry and tensionless string]}$\;$
   How do Azumaya geometry and tensionless string theory relate?
   Will it shed some light to or provide a mathematical
    language/foundation for the theory of tensionless strings?}
 \end{itemize}

\vspace{6em}

Behind the writing of this unexpected review/work,
 it is a wish to transfer the following sense to readers:
 \begin{itemize}
  \item[$\cdot$]
   {\bf [unity in geometry vs.\ unity in string theory]}\\[1.6ex]
   \framebox[18.6em][c]{\parbox{17.6em}{\it
    the master nature of morphisms from\\ Azumaya-type
    noncommutative spaces\\ with a fundamental module in geometry}}
    \hspace{1em} in parallel to \hspace{1em}
   \framebox[9em][c]{\parbox{8em}{\it
    the master nature of D-branes in \\superstring theory}}
  \end{itemize}

\bigskip

\noindent
This would be highly surprising/un-anticipated/unthinkable
  on the mathematics side
 if not because of the Polchinski-Grothendieck Ansatz,
 which realizes morphisms from Azumaya-type manifolds/schemes/stacks
  with a fundamental module
  as the {\it lowest level} presentation of D-branes,
 and superstring theory dictates the master nature of such an object.
We hope this brief review helps give
 both mathematicians and string-theorists working on D-branes
 a sense of
 {\it richness hidden in the Azumaya-type noncommutative geometry}
  or, even better,
 a {\it new perspective} of what they have been or are doing.
There are many themes along the line of
 the Polchinski-Grothendieck Ansatz yet to be studied/generalized.

\newpage
\baselineskip 13pt

\end{document}